\documentclass{amsart}
\allowdisplaybreaks[4]
\usepackage{graphicx}
\usepackage{color}

\newcommand{\Z}{\mathbf{Z}}
\newcommand{\R}{\mathbf{R}}

\newcommand{\sig}{\sigma}

\newcommand{{\ba}}{\bf a}
\newcommand{\ve}{\varepsilon}
\newcommand{\la}{\lambda}
\newcommand{\La}{\Lambda}
\newcommand{\ga}{\gamma}
\newcommand{\Ga}{\Gamma}
\newcommand{\pa}{\partial}
\newcommand{\ra}{\rightarrow}

\newcommand{\del}{\delta}

\newcommand{\cd}{\cdot}

\newcommand{\De}{\Delta}
\newcommand{\al}{\alpha}

\newcommand{\be}{\begin{equation}}
\newcommand{\ee}{\end{equation}}

\newtheorem{lem}{Lemma}{\bf}{\it}
{\it}{\rm}
\newtheorem{rem}{Remark}{\it}{\rm}
{\it}{\rm}

\newtheorem{theorem}{Theorem}
\newtheorem{proposition}{Proposition}

\numberwithin{theorem}{section}
\numberwithin{lem}{section}
\numberwithin{equation}{section}
\numberwithin{proposition}{section}
\numberwithin{corollary}{section}
\title[Carr-Penrose Model]{On large time behavior and selection principle for a diffusive Carr-Penrose Model }

\author{Joseph G. Conlon, Michael Dabkowski and Jingchen Wu}

\address{(Joseph G. Conlon): University of Michigan\\ Department of Mathematics\\ Ann Arbor,
  MI 48109-1109}
\email{conlon@umich.edu}

\address{(Michael Dabkowski): University of Michigan\\ Department of Mathematics\\ Ann Arbor,
  MI 48109-1109}
\email{mgdabkow@umich.edu}

\address{(Jingchen Wu):  500 9th Avenue, Seattle, WA 98109}

\email{jcwu@amazon.com}


\begin{document}

\maketitle

\begin{abstract}
This paper is concerned with the study of a diffusive perturbation of the linear LSW model introduced by Carr and Penrose. A main subject of interest is to understand how the presence of diffusion acts as a selection principle, which singles out a particular self-similar solution of the linear LSW model as determining the large time behavior of the diffusive model. A selection principle is rigorously proven for a model which is a semi-classical approximation to the diffusive model. Upper bounds on the rate of coarsening are also obtained for the full diffusive model.
\end{abstract}

\section{Introduction.}
In \cite{cp} Carr and Penrose introduced a linear version of the  Lifschitz-Slyozov-Wagner (LSW)  model \cite{ls,w}. In this model the density function $c_0(x,t), \ x>0,t>0,$ evolves according to the system of equations, 
\begin{eqnarray}
\frac{\pa c_0(x,t)}{\pa t}&=& \frac{\pa}{\pa x}\left[1-\frac{x}{\La_0(t)}\right]c_0(x,t) \ ,  \quad x>0, \label{A1}\\
\int_0^\infty x c_0(x,t) dx&=& 1. \label{B1}
\end{eqnarray}
The parameter $\La_0(t) > 0$ in (\ref{A1}) is determined by the conservation law  (\ref{B1}) and is therefore given by the formula,
\be \label{C1}
\La_0(t) = \int^\infty_0 \ x c_0(x,t)dx \Big/ \int^\infty_0 c_0(x,t) dx.
\ee
One can also see that the derivative of $\La_0(t)$ is given by   
\be \label{D1}
\frac{d\La_0(t)}{dt} \ = \ c_0(0,t)\Big/\left[ \int^\infty_0 c_0(x,t) dx\right]^2 \ ,
\ee
whence $\La_0(\cdot)$ is an increasing function.

The system (\ref{A1}), (\ref{B1}) can be interpreted as an evolution  equation for the probability density function (pdf) of random variables. Thus let us assume that the initial data $c_0(x)\ge 0, \ x>0,$ for (\ref{A1}), (\ref{B1}) satisfies $\int_0^\infty c_0(x) \ dx<\infty$, and let $X_0$ be the non-negative random variable with  pdf $c_0(\cdot)/\int_0^\infty c_0(x) \ dx$. The conservation law (\ref{B1}) implies that the mean $\langle X_0\rangle$ of $X_0$ is finite, and this is the only absolute requirement on the variable $X_0$.  If for $t>0$ the variable $X_t$ has pdf  $c_0(\cdot,t)/\int_0^\infty c(x,t) \ dx$, then (\ref{A1}) with $\La_0(t)=\langle  X_t\rangle$ is an evolution equation for the pdf of $X_t$.  Equation (\ref{D1}) now tells us that $\langle X_t\rangle$ is an increasing function of $t$. 

There is an infinite  one-parameter family of self-similar solutions to (\ref{A1}), (\ref{B1}). Using the normalization $\langle X_0\rangle =1$, the initial data for these solutions are given by
\be \label{E1}
P(X_0>x) \ = \ 
\begin{cases}
   [1-(1-\beta)x]^{\beta/(1-\beta)} \ , \ 0<x<1/(1-\beta),  \quad &\text{ if \ } 0<\beta<1, \\
   e^{-x}  &\text{ if \ } \beta=1, \\
 [1+(\beta-1)x]^{\beta/(1-\beta)} \ , \ 0<x<\infty,  \quad  & \ \text{if \ } \beta>1.
\end{cases}
\ee
The random variable $X_t$ corresponding to the evolution (\ref{A1}), (\ref{B1}) with initial data (\ref{E1}) is then given by
\be \label{F1}
X_t \ = \ \langle X_t\rangle X_0 \ , \quad \frac{d}{dt} \langle X_t\rangle  \ = \beta \ .
\ee
The main result of \cite{cp} (see also \cite{carr}) is that a solution of (\ref{A1}), (\ref{B1}) converges at large time to the self-similar solution with parameter $\beta$, provided the initial data  and the self similar solution of parameter $\beta$ behave in the same way at the end of their supports.  In $\S2$ we give a simple proof of the Carr-Penrose convergence theorem using the {\it beta function} of a random variable introduced in \cite{c2}. 

The large time behavior of the Carr-Penrose (CP) model  is qualitatively similar to the conjectured large time behavior of the LSW model \cite{np1}, provided the initial data  has compact support. In the LSW model there is a one-parameter family of self-similar solutions with parameter $\beta,  \ 0<\beta\le 1$, all of which have compact support. The self-similar solution with parameter $\beta<1$ behaves in the same way towards the end of its support as does the CP self-similar solution with parameter $\beta$.  It  has been conjectured \cite{np1} that a solution of the LSW model converges at large time to the LSW self-similar solution  with parameter $\beta$, provided the initial data  and the self similar solution of parameter $\beta$ behave in the same way at the end of their supports.  A weak version of this result has been proven in \cite{cn}. 

It was already claimed in \cite{ls,w} that the only physically relevant self-similar LSW solution is the one with parameter $\beta=1$. This  has been explained in a heuristic way in several papers \cite{meer, rz,vel},  by considering a model in which a second order diffusion term is added to the first order LSW equation. It is then argued that  diffusion acts as a {\it selection principle}, which singles out  the $\beta=1$ self-similar solution as giving the large time behavior.    In this paper we  study a diffusive version of the Carr-Penrose model, with the goal of understanding how a selection principle for the $\beta=1$ self-similar solution (\ref{E1}) operates.  

 In our diffusive CP model we simply add a second order diffusion term with coefficient $\ve/2>0$ to the CP equation (\ref{A1}). Then  the density function  $c_\ve(x,t)$ evolves according to a linear diffusion equation, subject to the linear mass conservation constraint as follows:
\begin{eqnarray}
\frac{\pa c_\ve(x,t)}{\pa t}&=& \frac{\pa}{\pa x}\left[1-\frac{x}{\La_\ve(t)}\right]c_\ve(x,t)+\frac{\ve}{2}\frac{\pa^2 c_\ve(x,t)}{\pa x^2},  \quad x>0, \label{G1}\\
\int_0^\infty x c_\ve(x,t) dx&=& 1. \label{H1}
\end{eqnarray}
We also need to impose a boundary condition at $x=0$ to ensure that (\ref{G1}), (\ref{H1}) with given initial data $c_\ve(x,0) = c_0(x) \ge 0, \ x > 0$, satisfying the constraint (\ref{H1})  has a unique solution.  We impose the Dirichlet boundary condition $c_\ve(0,t)=0, \ t>0,$ because in this case  the parameter $\La_\ve(t) > 0$ in (\ref{G1}) is  given by the formula
\be \label{I1}
\La_\ve(t) = \int^\infty_0 \ x c_\ve(x,t)dx \Big/ \int^\infty_0 c_\ve(x,t) dx.
\ee
Hence the diffusive CP model is an evolution equation for the pdf $c_\ve(\cdot,t)/\int_0^\infty c_\ve(x,t) \ dx$ of a random variable $X_{\ve,t}$ and $\La_\ve(t)=\langle X_{\ve,t}\rangle$. Furthermore, it is easy to see from (\ref{G1}), (\ref{H1}) that
\be \label{J1}
\frac{d\La_\ve(t)}{dt} \ = \ \frac{\ve}{2}\frac{\pa c_\ve(0,t)}{\pa x}\Big/\left[ \int^\infty_0 c_\ve(x,t) dx\right]^2.
\ee
It follows from (\ref{J1}) and the maximum principle \cite{pw} applied to (\ref{G1}) that the function $t\ra\La_\ve(t)$ is increasing.

In \cite{smereka} Smereka studied a discretized CP model and rigorously established a selection principle  for arbitrary initial data with finite support. He also proved that the {\it rate of convergence} to the  $\beta=1$ self-similar solution (\ref{E1}) is {\it logarithmic} in time.  Since discretization of a first order  PDE introduces an effective diffusion, one can just as well apply the discretization algorithm of \cite{smereka} to (\ref{G1}).  In the discretized model time $t$  is left continuous and the $x$ discretization $\De x$ is required to satisfy the condition $\ve=2\De x$.  In \cite{smereka} the large time behavior of solutions to this discretized model is  studied by using a Fourier method. The Fourier method cannot be implemented if the assumption $\ve=2\De x$ is dropped.  In $\S2$ we show that the discretized model is a $2$ dimensional dynamical system if and only if $\ve=2\De x$,  and that this dynamics is associated with the unique $2$ dimensional non-Abelian Lie algebra.  This places the Smereka discrete model in the same category as  the  CP model  (\ref{A1}), (\ref{B1}) and the quadratic model introduced in \cite{cn}, which have also been shown to be $2$ dimensional  with $2$ dimensional non-Abelian Lie algebra. 

In $\S2$ we begin by studying the CP model. If the initial data for (\ref{A1}), (\ref{B1}) is Gaussian, then it follows from \cite{cp} that the solution converges at large time to the $\beta=1$ self-similar solution (\ref{E1}). We prove that the rate of convergence is logarithmic in time as in the Smereka model. Next we consider how to extend this result to the diffusive CP model  (\ref{G1}), (\ref{H1}). We introduce a family of models with parameter $\nu, \ 0\le \nu\le 1,$  which interpolate between the  diffusive CP model  (\ref{G1}), (\ref{H1}) and a simpler model for which we can prove a selection principle. Each of these models is an evolution equation for the pdf of a non-negative random variable $X_t, \ t\ge 0,$ in which the function $t\ra \langle X_t\rangle$ is increasing. The evolution PDE is however now {\it non-linear} of viscous Burgers' type \cite{hopf}  with viscosity coefficient proportional to $\nu$. The $\nu=1$ model is identical to the diffusive CP model (\ref{G1}), (\ref{H1}), but the $\nu=0$ model is not the same as the  CP model  (\ref{A1}), (\ref{B1}). We shall refer to it as the {\it inviscid} CP model since its evolution PDE is an inviscid Burgers' equation \cite{smoller}.  Similarly we refer to the model with $0<\nu\le1$ as the {\it viscous} CP model with viscosity $\nu$.  

In $\S3$ we study the large time behavior of the inviscid CP model and obtain the following theorem:
\begin{theorem} Suppose the initial data for the inviscid CP model corresponds to the non-negative random variable $X_0$, and assume that $X_0$ satisfies
\be \label{K1}
\ve< \langle X_0 \rangle, \  \|X_0\|_\infty<\infty, \quad x\ra E[X_0-x \ | \ X_0>x] \ {\rm is \ decreasing \ for \ }  0\le x<\|X_0\|_\infty \ . 
\ee
Then $\lim_{t\ra\infty} \langle X_t\rangle/t=1$.  If in addition the function $x\ra E[X_0-x \ | \ X_0>x] $ is $C^1$ and convex for $x$ close to $\|X_0\|_\infty$  with 
\be \label{O1}
\liminf_{x\ra\|X_0\|_\infty} \frac{\pa}{\pa x} \  \frac{1}{E[X_0-x \ | \ X_0>x] } \ > \ 0 \ ,
\ee
 then there exists $C,T>0$ such that
\be \label{L1}
1-\frac{C}{\log t} \ \le \frac{d}{dt} \langle X_t\rangle \ \le \ 1 \quad {\rm for \ } t\ge T \ .
\ee
\end{theorem}
Observe that a $\beta<1$ self-similar solution (\ref{E1}) of the CP model has $\|X_0\|_\infty=1/(1-\beta)$ and $E[X_0-x \ | \ X_0>x]=1-(1-\beta)x, \ 0\le x<\|X_0\|_\infty$. Hence the $\beta<1$ self-similar solution satisfies all the conditions of Theorem 1.1 provided $\ve<1$.  The condition $\ve<1$ is not crucial since for any $\ve>0$ one can rescale the initial data so that $\ve<\langle X_0\rangle$. Therefore Theorem 1.1 proves a selection principle for the $\beta=1$ self-similar solution (\ref{E1}) and establishes a rate of convergence which is logarithmic in time. 

The remainder of the present paper is devoted to the study of the diffusive CP model (\ref{G1}), (\ref{H1}). Since existence and uniqueness has already been proven for a diffusive version of the LSW model \cite{c1} we do not revisit this issue, but concentrate on understanding large time behavior. In $\S7$ we obtain the following:
\begin{theorem} Suppose the initial data for the diffusive CP model (\ref{G1}), (\ref{H1}) corresponds to the non-negative random variable $X_0$ with integrable pdf. Then $\lim_{t\ra\infty} \langle X_t\rangle=\infty$.  If in addition the function $ x\ra E[X_0-x \ | \ X_0>x], \   0\le x<\|X_0\|_\infty,$ is decreasing, then 
there are constants $C,T>0$ such that
\be \label{M1}
0 \ \le \frac{d}{dt} \langle X_t\rangle \ \le \ C \quad {\rm for \ } t\ge T \ .
\ee
\end{theorem}
To establish the upper bound in (\ref{M1}) requires some delicate semi-classical analysis of the ratio of the Dirichlet Green's function for (\ref{G1}) on the half line $x>0$ to the whole line Green's function. We carry this out in $\S5$ by observing that the ratio of Green's functions is a probability for a generalized Brownian-bridge process, and obtaining a representation of the bridge process in terms of Brownian motion. In PDE terms this amounts to a boundary layer analysis of the solution to (\ref{G1}).  To see why this is the case,  observe that one can always rescale $\langle X_0\rangle$ to be equal to $1$ in both the CP and diffusive CP models. Since the CP model is dilation invariant, the evolution PDE (\ref{A1}) remains the same. However for the diffusive CP model the diffusion coefficient in (\ref{G1})  changes from $\ve$ to $\ve/\langle X_0\rangle$. Since $\lim_{t\ra\infty} \langle X_t\rangle=\infty$ in the diffusive CP model, an analysis of large time behavior must therefore involve an analysis of solutions to (\ref{G1}), (\ref{H1}) as $\ve \ra 0$. 

The simplest problem to understand concerning $\ve\ra 0$ behavior of the diffusive CP model is the problem of proving convergence to the solution of the CP model (\ref{A1}), (\ref{B1}) over some fixed time interval $0\le t\le T$. Thus we assume that the CP and diffusive CP models have the same initial data corresponding to a random variable $X_0$. If $X_t, \ t>0,$ is the random variable corresponding to the solution of (\ref{A1}), (\ref{B1}) and $X_{\ve,t}, \ t>0,$ the random variable corresponding to the solution of (\ref{G1}), (\ref{H1}) then $X_{\ve,t}$ converges in distribution as $\ve\ra 0$ to $X_t$, uniformly in the interval $0\le t\le T$.  Boundary layer analysis becomes necessary in proving the convergence of the diffusive coarsening rate (\ref{J1}) as $\ve\ra 0$ to the CP coarsening rate (\ref{D1}). In the diffusive model there exists a boundary layer with length of order $\ve$ so that $c_\ve(x,t)\simeq c_0(x,t)$ for $x/\ve>>1$. Since $c_\ve(0,t)=0$ one has that $\pa c_\ve(0,t)/\pa x\simeq 1/\ve$, whence the RHS of (\ref{J1}) remains bounded above $0$ as $\ve\ra 0$, and in fact converges to the RHS of (\ref{D1}). 

In $\S6$ we prove using the estimates of $\S5$ the convergence results over a finite time interval  described in the previous paragraph. Analogous results for the diffusive LSW model have already been proven in \cite{c1}.  However the semi-classical estimates obtained in \cite{c1} to prove convergence are not strong enough to prove a uniform in time upper bound on the rate of coarsening as given in (\ref{M1}). Although there is a formal analogy between the problem of understanding large time behavior of the diffusive CP model and the problem of $\ve\ra 0$ convergence of the diffusive CP model over a fixed time interval, the former problem is considerably more difficult than the latter.  

\vspace{.1in}

\section{The Carr-Penrose Model and Extensions}
The analysis of the CP model  \cite{cp} is based on the fact that the characteristics for the first order PDE (\ref{A1}) can be easily computed. Thus let $b:\R\times\R^+\ra\R$ be given by $b(y,s)=A(s)y-1, \ y\in\R, \ s\ge 0$, where $A:\R^+\ra\R^+$ is a continuous non-negative function. We define the mapping $F_A:\R\times\R^+\ra\R$ by setting
\be \label{A7}
F_A(x,t)=y(0), \quad {\rm where \ } \frac{dy(s)}{ds}=b(y(s),s), \ 0\le s\le t, \quad y(t)=x.
\ee
From (\ref{A7}) we see that the function $F_A$ is given by the formula
\begin{multline}  \label{B7}
F_A(x,t)= \frac{x+m_{2,A}(t)}{m_{1,A}(t)}  \ , \quad {\rm where \ } \\
m_{1,A}(t)=\exp\left[\int_0^tA(s) \ ds\right] \ , \quad 
m_{2,A}(t)=\int_0^t \exp\left[\int_s^tA(s') \ ds'\right] \ ds \ .
\end{multline} 
If we let $w_0:\R^+\times\R^+\ra\R^+$ be the function
\be \label{C7}
w_0(x,t) \ = \ \int_x^\infty c_0(x',t) \ dx' \ ,  \quad x,t\ge 0,
\ee
where $c_0(\cdot,\cdot)$ is the solution to (\ref{A1}), then from the method of characteristics we have that
\be \label{D7}
w_0(x,t) \ = \ w_0(F_{1/\La_0}(x,t), 0) \ , \quad x,t\ge 0. 
\ee
The conservation law (\ref{B1}) can also be expressed in terms of $w_0$ as
\be \label{E7}
\int_0^\infty w_0(x,t) \ dx \ = \ \int_0^\infty w_0(F_{1/\La_0}(x,t), 0) \ dx \ = \ 1.
\ee
Observe now that the functions $m_{1,A}, \ m_{2,A}$ of (\ref{B7}) are related by the differential equation
\be \label{F7}
\frac{d}{dt}\left[\frac{m_{2,A}(t)}{m_{1,A}(t)}\right] \ = \ \frac{1}{m_{1,A}(t)} \ .
\ee
It follows from (\ref{E7}), (\ref{F7}) that if we define variables $[u(t),v(t)],   \ t\ge 0,$ by
\be \label{G7}
u(t) \ = \ \frac{1}{m_{1,1/\La_0}(t)} \ ,  \quad v(t) \ = \ \frac{m_{2,1/\La_0}(t)}{m_{1,1/\La_0}(t)} \ ,
\ee
then the CP model (\ref{A1}), (\ref{B1}) with given initial data $c_0(\cdot,0)$  is equivalent to the $2$ dimensional dynamical system
\be \label{H7}
\frac{dv(t)}{dt} \ = u(t) \ , \quad \frac{d}{dt} \int_0^\infty w_0\left(u(t)x+v(t),0\right) \ dx \ = \ 0.
\ee 
Note however that the {\it dynamical law} for the $2$ dimensional evolution depends on the initial data for (\ref{A1}), (\ref{B1}), whereas the initial condition is always $u(0)=1, \ v(0)=0$.  

We can understand the $2$ dimensionality of the CP model and relate it to some other models of coarsening by using some elementary Lie algebra theory.  Thus observe that for operators $\mathcal{A}_0,\mathcal{B}_0$ defined by  
\be \label{I7}
\mathcal{A}_0 \ = \ \frac{d}{dx} \ , \quad \mathcal{B}_0=\frac{d}{dx} x \ , \quad {\rm then \ } \mathcal{A}_0\mathcal{B}_0-\mathcal{B}_0\mathcal{A}_0 \ = \ \mathcal{A}_0 \ .
\ee
The initial value problem  (\ref{A1}) can be written in operator notation as
\be \label{J7}
\frac{\pa c_0(\cdot,t)}{dt} \ = \ \left[\mathcal{A}_0-\frac{\mathcal{B}_0}{\La_0(t)}\right] c_0(\cdot, t)  \  {\rm for \ } t>0, \quad c_0(\cdot,0) \ = \ {\rm given} . 
\ee
It follows from (\ref{I7}) that the Lie Algebra generated by $\mathcal{A}_0,\mathcal{B}_0$ is the unique two dimensional non-Abelian Lie algebra. The corresponding
 $2$ dimensional Lie group is the affine group of the line (see Chapter  4 of \cite{still}). That is the Lie group consists of all transformations $z\ra az+b, \ z\in\R,$ with $a>0, \ b\in\R$. The solutions of equation (\ref{J7}) are a flow on this group. Hence solutions of (\ref{J7}) for all possible functions $\La_0(\cdot)$ lie on a two dimensional manifold. 

Next we consider the discretized version of the CP model studied by Smereka \cite{smereka}.  Letting  $\De x$ denote space discretization, then a standard discretization of (\ref{G1}) with Dirichlet boundary condition is given by
\begin{multline} \label{K7}
\frac{\pa c_\ve(x,t)}{\pa t}+ \frac{J_\ve(x,t)-J_\ve(x-\De x,t)}{\De x} \\
= \ \frac{\ve}{2}\frac{c_\ve(x+\De x,t)+c_\ve(x-\De x,t)-2c_\ve(x,t)}{(\De x)^2} \ , \quad x=(n+1)\De x,  \ \ n=0,1,2,..,
\end{multline}
where
\be \label{L7}
J_\ve(x,t) \ = \ \left[\frac{x}{\La_\ve(t)}-1\right]c_\ve(x,t) \ , \quad c_\ve(0,t) \ = \ 0. 
\ee
The backward difference approximation for the derivative of $J_\ve(x,t)$ is chosen in (\ref{K7}) to ensure stability of the numerical scheme for large $x$. 
Let $D,D^*$ be the discrete derivative operators acting on functions $u:(\De x)\Z\ra\R$ defined by
\be \label{M7}
Du(x) \ = \ \frac{u(x+\De x)-u(x)}{\De x} \ ,  \quad D^*u(x) \ = \ \frac{u(x-\De x)-u(x)}{\De x} \ .
\ee
Then using the notation of (\ref{M7}) we can rewrite (\ref{K7})  as
\be \label{N7}
\frac{\pa c_\ve(x,t)}{\pa t}-D^*J_\ve(x,t) \ = \ \frac{\ve}{2\De x} \left[ (D+D^*) c_\ve(x,t)\right] \ .
\ee
Observe that for operators $\mathcal{A}_{\De x},\mathcal{B}_{\De x}$ defined by  
\be \label{O7}
\mathcal{A}_{\De x} \ = \ D \ , \quad \mathcal{B}_{\De x}= -D^*x \ , \quad {\rm then \ } \mathcal{A}_{\De x}\mathcal{B}_{\De x}-\mathcal{B}_{\De x}\mathcal{A}_{\De x} \ = \ \mathcal{A}_{\De x} \ .
\ee
Choosing $\ve= 2\De x$ in (\ref{N7}), we see that the equation can be expressed in terms of   $\mathcal{A}_{\De x},\mathcal{B}_{\De x}$ as
\be \label{P7}
\frac{\pa c_\ve(\cdot,t)}{\pa t} \ = \ \left[\mathcal{A}_{\De x}-\frac{\mathcal{B}_{\De x}}{\La_\ve(t)}\right] c_\ve(\cdot,t) \ .
\ee

Comparing (\ref{I7}), (\ref{J7}) to (\ref{O7}), (\ref{P7}),  we see that  we can obtain a representation for the solution to (\ref{O7}), (\ref{P7}) by using the fact that the solution to (\ref{I7}), (\ref{J7}) is given by  (\ref{D7}). To see this we use the fact that for $\mathcal{A}_0,\mathcal{B}_0$ as in (\ref{I7}) then
\be \label{Q7}
e^{\mathcal{A}_0s} f(x) \ = \ f(x+s) \ , \quad e^{\mathcal{B}_0s}f(x) \ = \ e^sf(e^s x) \ , \quad x,s\in\R.
\ee
From (\ref{D7}), (\ref{G7}), (\ref{Q7}) it follows that the solution to (\ref{I7}), (\ref{J7}) is given by
\be \label{R7}
c_0(\cdot,t) \ = \ u(t)^{\mathcal{B}_0} e^{\mathcal{A}_0v(t)} c_0(\cdot,0) \ ,
\ee
where $u(t),  v(t)$ are given by (\ref{G7}).  Hence the solution to (\ref{O7}), (\ref{P7}) is given by
\be \label{S7}
c_\ve(\cdot,t) \ = \ u(t)^{\mathcal{B}_{\De x}} e^{\mathcal{A}_{\De x}v(t)} c_\ve(\cdot,0) \ ,
\ee
where $\mathcal{A}_{\De x},\mathcal{B}_{\De x}$ are given by (\ref{O7}), and  $u(t),  v(t)$ are given by (\ref{G7}) with $\La_\ve$ in place of $\La_0$. 

The operator $\mathcal{A}_{\De x}$ of (\ref{O7}) is the generator of a Poisson process. Thus
\be \label{T7}
\int_{(\De x)\Z} dx \ g(x)e^{\mathcal{A}_{\De x}s}f(x) \ = \ \int_{(\De x)\Z} dx \ E[g(X_s) \ | \ X_0=x] f(x) \ ,
\ee
where $X_s$ is the discrete random variable taking values in $(\De x)\Z$  with pdf
\be \label{U7}
P\left(X_s= y\ \big| \ X_0=x\right) \ = \ \frac{(s/\De x)^n}{n!}\exp\left[-\frac{s}{\De x}\right]  \ , \quad n \ = \ \frac{x-y}{\De x} \ , \ n=0,1,2,...
\ee
If $f,g:\R\ra\R$ are continuous functions of compact support then it is easy to see from (\ref{U7}) that
\be \label{V7}
\lim_{\De x\ra 0} \int_{(\De x)\Z} dx \ g(x)e^{\mathcal{A}_{\De x}s}f(x) \ = \ \int_{-\infty}^\infty dx \ g(x-s)f(x)  \ ,
\ee
as we expect from (\ref{Q7}).  The operator $\mathcal{B}_{\De x}$ of (\ref{O7})  is the generator of a Yule process \cite{kt}. Thus
\be \label{W7}
\int_{(\De x)\Z} dx \ g(x)e^{-\mathcal{B}_{\De x}s}f(x) \ = \ \int_{(\De x)\Z} dx \ E[g(Y_s) \ | \ Y_0=x] f(x) \ ,
\ee
where $Y_s$ is a discrete random variable taking values in $(\De x)\Z$.  The pdf of $Y_s$ conditioned on $Y_0=\De x$ is given by
\be \label{X7}
P\left(Y_s= y\ \big| \ Y_0=\De x\right) \ = \ e^{-s}\left\{1-  e^{-s}\right\}^{n-1} \ , \quad n \ = \ \frac{y}{\De x} \ , \ n=1,2,...
\ee
Hence $Y_s$ conditioned on $Y_0=\De x$ is a geometric variable.  More generally, the variable $Y_s$ conditioned on $Y_0=m\De x$ with $m\ge 2$ is a sum of $m$ independent geometric variables with distribution (\ref{X7}) and is hence negative binomial. It follows that  if $f(\cdot)$ is supported in the set $\{x\in(\De x)\Z \ : \ x>0\}$ then 
\be \label{Y7}
\int_{(\De x)\Z} dx \ g(x)e^{-\mathcal{B}_{\De x}s}f(x) \ = \ \int_{(\De x)\Z} dx \ E\left[g\left( Y_s^1+\cdots+ Y_s^m\right) \ \big| \ m=x/\De x\right] f(x)  \ ,
\ee
where the $Y_s^j, \ j=1,2,..,$ are independent and  have the distribution (\ref{X7}).  Since the mean of $Y_s$ is $e^s\De x$, it follows from (\ref{Y7}) that
\be \label{Z7}
\lim_{\De x\ra 0} \int_{(\De x)\Z} dx \ g(x)e^{-\mathcal{B}_{\De x}s}f(x) \ = \ \int_0^\infty dx \ g(e^s x)f(x)  \ ,
\ee
as we expect from (\ref{Q7}).

The Smereka model consists of the evolution determined by (\ref{K7}) with $\ve =2\De x$  and the conservation law
\be \label{AA7}
\int_{\{(\De x)\Z \ : \ x>0\}} xc_\ve(x,t) \ dx \ = \ 1.
\ee 
We see from (\ref{S7}) that the model is equivalent to a two dimensional dynamical system with dynamical law depending on the initial data.  The first differential equation in this system is given by the first equation in (\ref{H7}). The second differential equation is determined by differentiating the expression on the LHS of (\ref{AA7}) and setting it equal to zero.  Using (\ref{U7}), (\ref{X7}) we can write the LHS of (\ref{AA7}) in terms of $u(t),v(t)$. In the case when the initial data is given by
\be \label{AB7}
c_\ve(x,0) \ = \ 0 \ {\rm if \ } x\ne \De x \ , \quad c_\ve(\De x,0) \ = \ \frac{1}{(\De x)^2} \ ,
\ee
it has a simple form.  Thus from (\ref{S7}), (\ref{U7}), (\ref{X7}) we have that
\begin{multline} \label{AC7}
c_\ve(\cdot,t) \ = \ u(t)^B e^{Av(t)} c_\ve(\cdot,0) \ = \ u(t)^B \exp\left[-\frac{v(t)}{\De x}\right]c_\ve(\cdot,0) \ , \\
{\rm so  \ \  } c_\ve(x,t) \ = \  u(t)\left[1-u(t)\right]^{n-1}\exp\left[-\frac{v(t)}{\De x}\right]\frac{1}{(\De x)^2}  \ , \quad n= \frac{x}{\De x} \ .
\end{multline}
From (\ref{AC7}) we see that the conservation law (\ref{AA7}) becomes in this case
\be \label{AD7}
\frac{1}{u(t)} \exp\left[-\frac{v(t)}{\De x}\right] \ = \ 1.
\ee
Hence from the first equation of (\ref{H7}) and (\ref{AD7}) we conclude that $v(\cdot)$ is the solution to the initial value problem
\be \label{AE7}
  \exp\left[-\frac{v(t)}{\De x}\right]  \ = \ \frac{dv(t)}{dt}, \quad v(0)=0.
\ee 
The  initial value problem (\ref{AE7}) was derived in $\S3$ of \cite{smereka} by a different method. It can be solved explicitly, and so we obtain the formulas
\be \label{AF7}
u(t) \ = \  \frac{1}{1+t/\De x} \ , \quad v(t) \ = \ \De x \log\left[1+\frac{t}{\De x}\right] \ ,
\ee
when the initial data is given by (\ref{AB7}).  Hence from (\ref{AC7}), (\ref{AF7}) we have an explicit expression for $c_\ve(\cdot,t)$, and it is easy to see that this converges as $t\ra \infty$ to the self-similar solution corresponding to the $\beta=1$ random variable defined by (\ref{E1}). It was also shown  in \cite{smereka} that if the initial data has finite support then $c_\ve(\cdot,t)$ converges as $t\ra \infty$ to the $\beta=1$ self-similar solution.

The large time behavior of the CP model can be easily understood using the {\it beta  function} of a random variable introduced in \cite{c2}.  If $X$ is a random variable with pdf $c_X(\cdot)$, we define functions $w_X(\cdot), \ h_X(\cdot)$ by
\be \label{AG7}
w_X(x) \ = \int_x^\infty c_X(x') \ dx' \ , \quad h_X(x) \ = \  \int_x^\infty w_X(x') \ dx'  \ , \quad x\in\R \ .
\ee
Evidently one has that
\be \label{AH7}
w_X(x) \ = \ P(X>x), \ \quad \frac{h_X(x)}{w_X(x)} \ = \ E\left[X-x \  \big| \ X>x\right] \ \quad x\in\R \ . 
\ee
The beta function $\beta_X(\cdot)$ of $X$ is then defined by
\be \label{AI7}
\beta_X(x)  \ = \ \frac{c_X(x)h_X(x)}{w_X(x)^2} \ = \  1+\frac{d}{dx}  E\left[X-x \  \big| \ X>x\right] \ \quad x\in\R \ . 
\ee
An important property of the beta function is that it is invariant under affine transformations. That is
\be \label{AJ7}
 \beta_X(\la x+\mu) \ = \ \beta_{(X-\mu)/\la}(x) \ , \quad \la>0, \ \mu,x\in\R. 
\ee
One can also see that the function $h_X(\cdot)$ is {\it log concave} if and only if $\sup\beta_X(\cdot)\le 1$. 

To understand the large time behavior of the CP model we first observe that the rate of coarsening equation (\ref{D1}) can be rewritten as 
\be \label{AK7}
\frac{d}{dt} \langle X_t\rangle \ = \ \beta_{X_t}(0) \ , \quad t>0. 
\ee
Furthermore, the beta function of the self-similar variable $X_0$ with pdf defined by (\ref{E1}) and parameter $\beta>0$ is simply a constant  $\beta_{X_0}(\cdot)\equiv \beta$.  We have already shown that the time evolution of the CP equation (\ref{A1}) is given by the affine transformation (\ref{D7}). It is also relatively simple to establish that for a random variable $X_0$ corresponding to the initial data for (\ref{A1}), (\ref{B1}), then  $\lim_{t\ra\infty} F_{1/\La_0}(0,t)=\|X_0\|_\infty$. 
Hence  it follows from (\ref{AI7}), (\ref{AJ7})  that if $\lim_{x\ra\|X_0\|_\infty}\beta_{X_0}(x)=\beta$  for the initial data random variable $X_0$ of (\ref{A1}), (\ref{B1}),  then the large time behavior of the CP model is determined by the self-similar solution (\ref{E1}) with parameter $\beta$.

We have already observed from (\ref{D1}) that the function $\La_0(\cdot)$ in the CP model is  increasing. If we assume that $\inf \beta_{X_0}(\cdot)>0$, we can also see that $\lim_{t\ra\infty} \La_0(t)=\infty$. Hence in this case there exists a doubling time $T_{\rm double}$ for which $\La_0(t)=2\La_0(0)$ when $t=T_{\rm double}$. Evidently $\inf\beta_{X_t}(\cdot)\ge \inf\beta_{X_0}(\cdot)$ and  $\sup\beta_{X_t}(\cdot)\le \sup\beta_{X_0}(\cdot)$. The notion of doubling time can be a useful tool in obtaining an estimate on the rate of convergence of the solution of the  CP model to a self-similar solution at large time. 

We illustrate this by considering the CP model with Gaussian initial data. In particular we assume the initial data $c_0(\cdot)$ is given by the formula
\be \label{AL7}
c_0(x) \ = \ K(L)\exp\left[-a(L)x-\{a(L)x\}^2/2L\right] \ ,
\ee
where $L\ge L_0>0$ and $K(L),a(L)$ are uniquely determined by the requirement  that (\ref{B1}) holds and the function $\La_0(\cdot)$ in (\ref{C1}) satisfies $\La_0(0)=1$. It is easy to see that the beta function for the initial data (\ref{AL7}) is bounded above and below strictly larger than zero, uniformly in $L\ge L_0$. Hence from (\ref{AK7})  there are  constants $T_1,T_2>0$ depending only on $L_0$ such that $T_1\le T_{\rm double}\le T_2$ for all $L\ge L_0$.  It follows also from (\ref{B7}) that there are  constants $\la_0,\la_1,\mu_0,\mu_1>0$ depending only on $L_0$ such that $F_{1/\La_0}(x,T_{\rm double})=\la(L)x+\mu(L), \ x\in\R,$ where $0<\la_0\le\la(L)\le \la_1<1$ and $ 0<\mu_0\le\mu(L)\le \mu_1$ for  $L\ge L_0$. Since $F_{1/\La_0}$ is a linear function, $c(\cdot,T_{\rm double})$ is also Gaussian. Rescaling so that the mean of $X_t$ is now $1$ at time $T_{\rm double}$, we see that $c(x,T_{\rm double})$ is given by the formula (\ref{AL7}) with $L$ replaced by $A(L)$, where
\be \label{AM7}
A(L) \ = \  L[1+\mu(L)a(L)/L]^2=L+2\mu(L)a(L)+\mu(L)^2a(L)^2/L \ .
\ee
Since we are assuming that $\La_0(0)=1$, there are constants $a_0,a_1>0$ depending only on $L_0$ such that $a_0\le a(L)\le a_1$ for $L\ge L_0$.  We conclude then from (\ref{AM7}) that
\be \label{AO7}
L+\del_0 \ \le A(L) \ \le L+\del_1 \ , \quad   {\rm where \ } \del_0,\del_1>0 \ {\rm depend \ only \ on \ }L_0 \  .
\ee

It is easy to estimate from (\ref{AO7}) the rate of convergence to the $\beta=1$ self similar solution $c(x,t)=(1+t)^{-2}\exp[-x/(1+t)]$ for solutions to the CP model with Gaussian initial data.  
 First we estimate the beta function of a Gaussian random variable. 
\begin{lem}
Let $L>0$ and  $Z_L$ be a positive random variable with pdf proportional to $e^{-z-z^2/2L}, \ z>0$. 
Then for any $L_0>0$ there is a constant $C$ depending only on $L_0$  such that if $L\ge L_0$ the beta function $\beta_L$ for $Z_L$ satisfies the inequality
\be \label{AP7}
\left| \ \beta_L(z)-1+\frac{1}{L(1+z/L)^2} \ \right| \ \le \ \frac{C}{L^2(1+z/L)^4} \quad {\rm for \ } z\ge 0. 
\ee
\end{lem}
\begin{proof}
We use the formula for the beta function $\beta(\cdot)$ of the pdf $c(\cdot)$ given by (\ref{AI7}). Thus 
\be \label{AQ7}
\beta(z) \ = \ \frac{c(z)h(z)}{w(z)^2} \ , \quad w(z) \ =  \ \int_0^\infty c(z+z') \ dz' \ , \ \  h(z) \ =  \ \int_0^\infty z' c(z+z') \ dz' \ .
\ee 
Letting $c(z)=e^{-z-z^2/2L}$, we have that
\be \label{AR7}
w(z) \ =  \ c(z) \int_0^\infty e^{-z'[1+z/L]-z'^2/2L} \ dz'  \ , \quad  h(z) \ =  \ c(z) \int_0^\infty z' e^{-z'[1+z/L]-z'^2/2L} \ dz'  \ .
\ee
It follows from (\ref{AQ7}), (\ref{AR7}) on making a change of variable that
\be \label{AS7}
\beta_L(z) \ = \  \int_0^\infty xe^{-x-\del x^2/2} \ dx \  \Big/ \left[\int_0^\infty e^{-x-\del x^2/2} \ dx \ \right]^2 \  ,
\ee
where $\del =1/L[1+z/L]^2$. It is easy to see that there is a universal constant $K$ such that the RHS of (\ref{AS7}) is bounded above by $K$ for all $\del>0$.  We also have by Taylor expansion in $\del$ that $\beta_L(z)=1-\del+O(\del^2)$ if $\del$ is small. The inequality (\ref{AP7}) follows. 
\end{proof}
\begin{proposition}
Let  $c_0(\cdot,\cdot)$ be the solution  to the  CP system (\ref{A1}), (\ref{B1}) with Gaussian initial data and $\La_0(\cdot)$ be given by (\ref{C1}). Then there exists $t_0>2$ and  constants $C_1,C_2>0$ such that
\be \label{AT7}
1-\frac{C_1}{\log t} \ \le \ \frac{d\La_0(t)}{dt} \ \le \ 1-\frac{C_2}{\log t} \quad {\rm for \ } \ t\ge t_0.
\ee
\end{proposition}
\begin{proof}
The initial data can be written in the form $c_0(x,0)=K_0\exp[-A_0(x+B_0)^2]$ where $K_0,A_0,B_0$ are constants with $K_0,A_0>0$. It follows from (\ref{B7}), (\ref{D7}) that for $t>0$ one has 
$c_0(x,t)=K_t\exp[-A_t(x+B_t)^2]$, where $B_t=[B_0+F_{1/\La_0}(0,t)](A_0/A_t)^{1/2}$. Since $\lim_{t\ra\infty} F_{1/\La_0}(0,t)=\infty$, it follows that we may assume wlog that the initial data is of the form (\ref{AL7}) and $\La_0(0)=1$. Evidently then $\beta_{X_0}(x)=\beta_L(a(L)x), \ x\ge0,$ where $\beta_L$ satisfies the inequality (\ref{AP7}).

Assume now that  the initial data for (\ref{A1}), (\ref{B1}) is  given by (\ref{AL7}) where $L=L_0>0$, and let $L_t$ be the corresponding value of $L$ determined by $c_0(\cdot,t)$.  We have then from (\ref{AO7}) and the discussion preceding it that for $N=1,2,..,$ there exist times $t_N$ such that
\be \label{AU7}
[2^N-1]T_1  \le \ t_N  \le [2^N-1]T_2, \quad L_0+N\del_0\le L_{t_N}\le L_0+N\del_1 .
\ee
Since  $L_t$ is an increasing function of $t$, it follows from (\ref{AP7}), (\ref{AU7})  that $\beta_{X_t}(0)$ is bounded above and below as in (\ref{AT7}).    Now using the identity (\ref{AK7}) we obtain the inequality (\ref{AT7}). 
\end{proof}
We wish next to compare the foregoing to the situation of the diffusive CP model with Gaussian initial data.    From (\ref{M2}) the solution to (\ref{G1}) with initial data $c_0(\cdot)$ is given by  
\be \label{AV7}
c_\ve(x,t) \ = \ \int_0^\infty G_{\ve,D}(x,y,0,t) c_0(y) \ dy, \quad x>0, \  t>0,
\ee
where $G_{\ve,D}$ is the Dirichlet Green's function for the half space $\R^+$ defined by  (\ref{L2}) with $A(s)=1/\La_\ve(s)$.  
If we replace the Dirichlet Green's function $G_{\ve,D}$  by the full space Green's function $G_\ve$ of (\ref{K2}) then  the solution $c_\ve(\cdot,t)$ is Gaussian for $t>0$ provided $c_\ve(\cdot,0)$ is Gaussian, just as in the CP model. We shall see in  $\S5$ that it is legitimate to approximate $G_{\ve,D}(x,y,0,t)$ by $G_\ve(x,y,0,t)$ provided $x,y\ge M\ve$ for some large constant $M$. Making the approximation $G_{\ve,D}\simeq G_\ve$ in (\ref{AV7}), we obtain a formula similar to (\ref{AM7}) for the length scale $A_\ve(L)$ of the Gaussian at doubling time.  It is given by
\be \label{AW7}
A_\ve(L) \ = \ L[1+\mu_{\ve}(L)a(L)/L]^2[1+\ve a(L)^2\sig^2_\ve(L)\la_{\ve}(L)^2/L]^{-1} \ .
\ee
As in (\ref{AM7}) the functions $\la_\ve(L),\mu_\ve(L)$ are obtained from the coefficients of the linear function $F_{1/\La_\ve}$, when $t=T_{\ve,{\rm double}},$  where $T_{\ve,{\rm double}}$ denotes the doubling time for the diffusive model. The expression $\sig^2_\ve(L)$ is given by the formula for $\sig^2_A(T)$ in  (\ref{J2}) with $A(s)=1/\La_\ve(s),  \ s\le T,$ and $T=T_{\ve,{\rm double}}$. In $\S6$ we shall study the $\ve\ra 0$ limit of the diffusive CP model. We prove that if the CP and diffusive CP models have the same initial data, then $\lim_{\ve\ra 0}\La_\ve(t)=\La_0(t)$ uniformly in any finite interval $0\le t\le T$.  It follows that $\lim_{\ve\ra0}A_\ve(L)=A(L)$, where $A(L)$ is defined by (\ref{AM7}).

We wish next to try to understand the evolution of the diffusive CP model when initial data is non-Gaussian. Let $w_\ve(x,t), \ h_\ve(x,t)$  be defined in  terms of the solution $c_\ve(\cdot,\cdot)$  to (\ref{G1}) by 
\be \label{AX7}
w_\ve(x,t) \ = \ \int_x^\infty c_\ve(x',t) \ dx', \quad h_\ve(x,t) \ = \ \int_x^\infty w_\ve(x',t) \ dx' \ .
\ee
Then $w_\ve(\cdot,t),h_\ve(\cdot,t)$ are proportional to the functions (\ref{AG7}) corresponding to the random variable $X_t$ with pdf  $c_\ve(\cdot,t)/\int_0^\infty c_\ve(x,t) \ dx$. Making the approximation $G_{\ve,D}\simeq G_\ve$, we see from (\ref{AV7}), (\ref{AX7}), (\ref{K2}) that 
\begin{eqnarray} \label{AY7}
w_\ve(x,t) \ &=& \ \exp\left[\int_0^t\frac{ds}{\La_\ve(s)}\right]\int_{-\infty}^\infty G_\ve(x,y,0,t) w_\ve(y,0) \ dy  \ ,\\
h_\ve(x,t) \ &=& \ \exp\left[2\int_0^t\frac{ds}{\La_\ve(s)}\right]\int_{-\infty}^\infty G_\ve(x,y,0,t) h_\ve(y,0) \ dy \ . \label{AZ7}
\end{eqnarray}
 Writing $h_\ve(x,t)=\exp[-q_\ve(x,t)], \ x,t>0,$ in (\ref{AZ7}), we see  from (\ref{K2})  that the semi-classical approximation to $q_\ve(x,t)$ is given by the formula
\be \label{BA7}
q_\ve(x,t) \ = \  \frac{1}{2}\log 2\pi\ve+\frac{1}{2}\log \sig_{1/\La_\ve}^2(t)-2\int_0^t \frac{ds}{\La_\ve(s)}+\inf_{y} \left[   \frac{\{x+m_{2,1/\La_\ve}(t)-m_{1,1/\La_\ve}(t)y\}^2}{2\ve\sig_{1/\La_\ve}^2(t)}+ q_\ve(y,0) \right] \ .
\ee
Let us assume that $q_\ve(\cdot,0)$ is given similarly to (\ref{AL7}) by
\be \label{BB7}
q_\ve(y,0) \ = \ {\rm constant}+ a(L)y+ \{a(L)y\}^2/2L \  , \quad y>0. 
\ee
The minimizer in (\ref{BA7}) is then $y_{\rm min}(x,t)$ where
\be \label{BC7}
y_{\rm min}(x,t) \ = \  \frac{1}{1+\ve a(L)^2\sig_{1/\La_\ve}^2(t)/m_{1,1/\La_\ve}(t)^2L}\left[ \frac{x+m_{2,1/\La_\ve}(t)}{m_{1,1/\La_\ve}(t)}- \frac{\ve a(L)\sig_{1/\La_\ve}^2(t)}{m_{1,1/\La_\ve}(t)^2} \right] \  .
\ee
If we substitute $y=y_{\rm min}(x,t)$ into (\ref{BA7})  we obtain a quadratic formula for $q_\ve(x,t)$ similar to (\ref{BB7}). If $t=T_{\ve,{\rm double}}$  then $L$ in (\ref{BB7}) is replaced by $A_\ve(L)$ as in (\ref{AW7}). 

More generally we can consider the case when $q_\ve(\cdot,0)$ is convex so (\ref{BA7}) is a convex optimization problem with  a unique minimizer $y=y_{\rm min}(x,t)$. In that case it is easy to see that
\begin{multline} \label{BD7}
\frac{\pa q_\ve(x,t)}{\pa x} \ = \ \frac{1}{m_{1,1/\La_\ve}(t)} \frac{\pa q_\ve(y_{\rm min}(x,t),0)}{\pa y} \ , \\
\frac{\pa^2 q_\ve(x,t)}{\pa x^2} \ = \ \frac{1}{m_{1,1/\La_\ve}(t)^2} \frac{\pa^2 q_\ve(y_{\rm min}(x,t),0)}{\pa y^2}  \ \Bigg/ \ \left[1+\frac{\ve \sig_{1/\La_\ve}^2(t)}{m_{1,1/\La_\ve}(t)^2} \ \frac{\pa^2 q_\ve(y_{\rm min}(x,t),0)}{\pa y^2} \ \right] \ .
\end{multline}
It follows from (\ref{BD7}) that if the inequality
\be \label{BE7}
 \frac{\pa^2 q_\ve(x,t)}{\pa x^2} \ \le \  \left[ \  \frac{\pa q_\ve(x,t)}{\pa x}  \ \right]^2  \ , \quad x\ge 0,
\ee
holds at $t=0$ then it holds for all $t>0$. 
We define now the function $\beta_\ve:[0,\infty)\times\R^+\ra \R$ in terms of $q_\ve$ by the formula
\be \label{BF7}
\beta_\ve(x,t) \ = \ 1- \frac{\pa^2 q_\ve(x,t)}{\pa x^2} \  \Bigg/ \left[ \  \frac{\pa q_\ve(x,t)}{\pa x}  \ \right]^2 \ .
\ee
We can see from (\ref{AI7}) that the function $h_\ve(\cdot,t)=\exp[-q_\ve(\cdot,t)]$ is proportional to $h_{X_t}(\cdot)$ for some  random variable $X_t$ if and only if $\beta_\ve(\cdot,t)$ is non-negative. Hence by the remark after (\ref{BE7}), if $q_\ve(\cdot,0)$ corresponds to a random variable $X_0$, then $q_\ve(\cdot,t)$ corresponds to a random variable $X_t$ for all $t>0$. 
From (\ref{BD7}), (\ref{BF7}) we have that
\be \label{BG7}
1-\beta_\ve(x,t) \ = \ \left[1-\beta_\ve(y_{\rm min}(x,t),0)\right] \ \Big/ \left[1+\frac{\ve \sig_{1/\La_\ve}^2(t)}{m_{1,1/\La_\ve}(t)^2} \ \frac{\pa^2 q_\ve(y_{\rm min}(x,t),0)}{\pa y^2} \ \right]  \ .
\ee
It follows from (\ref{BG7}) that if $\sup \beta_\ve(\cdot,0)\le 1$ then $\sup\beta_\ve(\cdot,t)\le 1$ for $t>0$.   Furthermore, (\ref{BG7}) also indicates that $\beta_\ve(\cdot,t)$ should increase towards $1$ as $t\ra\infty$. 

It is well known \cite{hopf} that the solution $q_\ve$ to the optimization problem (\ref{BA7}) satisfies a Hamilton-Jacobi  PDE. We can easily see from (\ref{BA7}), (\ref{BD7}) that the PDE is given by
\be \label{BH7}
\frac{\pa q_\ve(x,t)}{\pa t}+\left[\frac{x}{\La_\ve(t)}-1\right]\frac{\pa q_\ve(x,t)}{\pa x} +\frac{\ve}{2}\left[\frac{\pa q_\ve(x,t)}{\pa x}\right]^2
+\frac{1}{\La_\ve(t)}-\frac{1}{2\sig^2(t)}  \ = \  0 \ .
\ee
Differentiating (\ref{BH7}) with respect to $x$ and setting $v_\ve(x,t)=\pa q_\ve(x,t)/\pa x$, we see that $v_\ve(x,t)$ is the solution to  the inviscid Burgers' equation with  linear drift,
\be \label{BI7}
\frac{\pa v_\ve(x,t)}{\pa t}+\left[\frac{x}{\La_\ve(t)}-1+\ve v_\ve(x,t)\right]\frac{\pa v_\ve(x,t)}{\pa x}
 +\frac{v_\ve(x,t)}{\La_\ve(t)} \ = \  0 \ .
\ee
If $q_\ve(\cdot,t)$ corresponds to the random variable $X_t$, then $v_\ve(x,t)=E[X_t-x \ | \ X_t>x]^{-1}, \ x\ge 0$, and since $\La_\ve(t)=E[X_t]$ we have that 
\be \label{BJ7}
v_\ve(0,t) \ = \ \frac{1}{\La_\ve(t)} \ .
\ee 
The system (\ref{BI7}), (\ref{BJ7}) is a model for the evolution of the pdf of a random variable $X_t$ which is intermediate between the CP and diffusive CP models.  To obtain the pdf of $X_t$ from the function $v_\ve(\cdot,t)$, we let $c_\ve(\cdot,t)=c_{X_t}(\cdot),  \ w_\ve(\cdot,t)=w_{X_t}(\cdot), \ h_\ve(\cdot,t)=h_{X_t}(\cdot)$ as in (\ref{AG7}). Then  $v_\ve(x,t)=w_\ve(x,t)/h_\ve(x,t)$ and
\be \label{BK7}
\Gamma_\ve(x,t) \ = \ v_\ve(x,t)^2- \frac{\pa v_\ve(x,t)}{\pa x} \ = \  \frac{c_\ve(x,t)}{h_\ve(x,t)} \  .
\ee 
We also have that
\be \label{BL7}
v_\ve(x,t) \ = \ -\frac{\pa}{\pa x} \log h_\ve(x,t) \ , \quad {\rm whence \ }  h_\ve(x,t) \ =  A_\ve(t)\exp\left[-\int_0^x v_\ve(x',t) \ dx' \right] \ ,
\ee
where $A_\ve(\cdot)$ can be an arbitrary positive function.  Evidently (\ref{BK7}), (\ref{BL7}) uniquely determine the pdf of $X_t$ from the function $v_\ve(\cdot,t)$. 

We can do a more systematic derivation of the model (\ref{BI7}), (\ref{BJ7})  by beginning with the solution $c_\ve$ to the diffusive CP model (\ref{G1}), (\ref{H1}). Setting $w_\ve,h_\ve$ to be given by (\ref{AX7}), then we see on integration of (\ref{G1}) that $w_\ve$ is a solution to the PDE 
\be \label{BM7}
\frac{\pa w_\ve(x,t)}{\pa t}+\left[\frac{x}{\La_\ve(t)}-1\right]\frac{\pa w_\ve(x,t)}{\pa x}
 \ = \  \frac{\ve}{2} \frac{\pa^2 w_\ve(x,t)}{\pa x^2} \ .
\ee
If we integrate (\ref{BM7}) then we obtain a PDE for $h_\ve$, 
\be \label{BN7}
\frac{\pa h_\ve(x,t)}{\pa t}+\left[\frac{x}{\La_\ve(t)}-1\right]\frac{\pa h_\ve(x,t)}{\pa x}
-\frac{h_\ve(x,t)}{\La_\ve(t)} \ = \  \frac{\ve}{2} \frac{\pa^2 h_\ve(x,t)}{\pa x^2} \ .
\ee
Setting $h_\ve(x,t)=\exp[-q_\ve(x,t)]$, it follows from (\ref{BN7}) that $q_\ve(x,t)$ is a solution to the PDE
\be \label{BO7}
\frac{\pa q_\ve(x,t)}{\pa t}+\left[\frac{x}{\La_\ve(t)}-1\right]\frac{\pa q_\ve(x,t)}{\pa x}
+\frac{\ve}{2}\left[ \frac{\pa q_\ve(x,t)}{\pa x} \right]^2 +\frac{1}{\La_\ve(t)} \ = \  \frac{\ve}{2} \frac{\pa^2 q_\ve(x,t)}{\pa x^2} \ .
\ee
If we differentiate (\ref{BO7}) with respect to $x$ we obtain a PDE for the function $v_\ve(x,t)=\pa q_\ve(x,t)/\pa x$, whence  we have
\be \label{BP7}
\frac{\pa v_\ve(x,t)}{\pa t}+\left[\frac{x}{\La_\ve(t)}-1+\ve v_\ve(x,t)\right]\frac{\pa v_\ve(x,t)}{\pa x}
 +\frac{v_\ve(x,t)}{\La_\ve(t)} \ = \  \frac{\ve}{2} \frac{\pa^2 v_\ve(x,t)}{\pa x^2} \ .
\ee

For $0<\nu\le 1$, we define the {\it viscous} CP model with viscosity $\nu$ as the solution to the PDE
\be \label{BQ7}
\frac{\pa v_{\ve,\nu}(x,t)}{\pa t}+\left[\frac{x}{\La_{\ve,\nu}(t)}-1+\ve v_{\ve,\nu}(x,t)\right]\frac{\pa v_{\ve,\nu}(x,t)}{\pa x}
 +\frac{v_{\ve,\nu}(x,t)}{\La_{\ve,\nu}(t)} \ = \  \frac{\ve\nu}{2} \frac{\pa^2 v_{\ve,\nu}(x,t)}{\pa x^2} \ , \ \ x,t>0,
\ee
with boundary condition
\be \label{BR7}
\frac{\pa v_{\ve,\nu}(0,t)}{\pa x} \ = \ v_{\ve,\nu}(0,t)^2 \ , \quad t>0, 
\ee
and with the constraint
\be \label{BS7}
v_{\ve,\nu}(0,t) \ = \ \frac{1}{\La_{\ve,\nu}(t)} \ , \quad t\ge 0. 
\ee 
Assuming that (\ref{BQ7}), (\ref{BR7}) has a classical solution, we show that if the initial data for (\ref{BQ7}) corresponds to a random variable $X_0$, then $v_{\ve,\nu}(\cdot,t)$ corresponds to a random variable $X_t$ for $t>0$ in the sense that $v_{\ve,\nu}(x,t)=E[X_t-x \ | \ X_t>x]^{-1}, \ x\ge 0$. To see this we define $\Ga_{\ve,\nu}$ similarly to $\Ga_\ve$ in (\ref{BK7}) but with $v_\ve$ replaced by $v_{\ve,\nu}$ on the RHS.  It follows from (\ref{BQ7}), (\ref{BR7}) that $\Ga_{\ve,\nu}$ satisfies the PDE
\begin{multline} \label{BT7}
\frac{\pa \Gamma_{\ve,\nu}(x,t)}{\pa t}+\left[\frac{x}{\La_{\ve,\nu}(t)}-1+\ve v_{\ve,\nu}(x,t)\right]\frac{\pa \Gamma_{\ve,\nu}(x,t)}{\pa x} +2\frac{\Gamma_{\ve,\nu}(x,t)}{\La_\ve(t)} \\
 = \  \frac{\ve\nu}{2}\frac{\pa^2 \Gamma_{\ve,\nu}(x,t)}{\pa x^2}+ \ve(1-\nu) \left(\frac{\pa v_{\ve,\nu}(x,t)}{\pa x}\right)^2  \ ,
\end{multline}
with Dirichlet boundary condition $\Ga_{\ve,\nu}(0,t)=0, \ t>0$. 
Hence by the maximum principle \cite{pw}, if $\Ga_{\ve,\nu}(\cdot,0)$ is non-negative then $\Ga_{\ve,\nu}(\cdot,t)$ is non-negative for $t>0$. We see from (\ref{BK7}) that  the non-negativity of $\Ga_{\ve,\nu}(\cdot,t)$ is equivalent to $v_{\ve,\nu}(\cdot,t)$ corresponding to a random variable $X_t$. We have shown that for $0<\nu\le 1$ the viscous CP model corresponds to the evolution of a random variable $X_t, \ t\ge 0$. If $\nu=1$ the model is identical to the diffusive CP model (\ref{G1}), (\ref{H1}) with Dirichlet condition $c_\ve(0,t)=0, \ t>0$. 

We can think of the {\it inviscid} CP model (\ref{BI7}), (\ref{BJ7})  as the limit of the viscous CP model (\ref{BQ7}), (\ref{BR7}), (\ref{BS7}) as the viscosity $\nu\ra0$.  It is not clear however what happens to the boundary condition (\ref{BR7}) in this limit. Unless the initial data $v_\ve(\cdot,0)$ for (\ref{BI7}) is increasing, the solution $v_\ve(\cdot,t)$ develops discontinuities at some finite time  \cite{smoller}. For an entropy satisfying solution $v_\ve$, discontinuities have the property that the solution {\it jumps down} across the discontinuity. Hence if $v_\ve(\cdot,t)$ is discontinuous at the point $z$, then
\be \label{BU7}
\lim_{x\ra z^-}v_\ve(x,t) \ > \ \lim_{x\ra z^+}v_\ve(x,t) \ .
\ee
Observe now that for a random variable $X$, the function $x\ra E[X-x \ |  \ X>x]$ has discontinuities precisely at the atoms of $X$. In that case the function  {\it jumps up} across the discontinuity. Since $v_\ve(x,t)=E[X_t-x \ | \ X_t>x]^{-1}$ for some random variable $X_t$, it follows that  at discontinuities of $v_\ve(\cdot,t)$ the function jumps down.  Thus discontinuities of $v_\ve(\cdot,t)$ correspond to atoms of $X_t$, and the entropy condition for (\ref{BI7}) is automatically satisfied.

We have already observed that the function $t\ra\La_0(t)$ in the CP model (\ref{A1}), (\ref{B1}) is increasing, and that the function  $t\ra\La_\ve(t)$ in the diffusive CP model (\ref{G1}), (\ref{H1}) is also increasing.  To determine whether the function $t\ra\La_{\ve,\nu}(t)$ in the viscous CP model (\ref{BQ7})- (\ref{BS7}) is  increasing, we observe on setting $x=0$ in  (\ref{BQ7}) and using (\ref{BS7}) that $v_{\ve,\nu}(0,t)$ satisfies the equation
\be \label{BV7}
\frac{\pa v_{\ve,\nu}(0,t)}{\pa t}+\left\{1-\ve(1-\nu) v_{\ve,\nu}(0,t)\right\}\Gamma_{\ve,\nu}(0,t)+\ve(1-\nu) v_{\ve,\nu}(0,t)^3 +\frac{\ve\nu}{2}\frac{\pa \Ga_{\ve,\nu}(0,t)}{\pa x} \ = \ 0.
\ee
We have already seen that $\Ga_{\ve,\nu}(\cdot,t)$ is a non-negative function, and from (\ref{BR7}) it follows that $\Ga_{\ve,\nu}(0,t)=0$ for $t>0$.  Hence $\pa \Ga_{\ve,\nu}(0,t)/\pa x\ge 0$ for $t>0$. We conclude then from  (\ref{BV7}) that the function $t\ra v_{\ve,\nu}(0,t)$ is decreasing provided 
\be \label{BW7}
v_{\ve,\nu}(0,0) \ \le \  \frac{1}{\ve(1-\nu)} \ .
\ee
Thus from (\ref{BS7}) we see that if  (\ref{BW7}) holds, then the function $t\ra\La_{\ve,\nu}(t)$ is increasing. 
Note that in the case of the diffusive CP model when $\nu=1$ the condition (\ref{BW7}) is redundant. 

\vspace{.1in}

\section{The Inviscid CP Model-Proof of Theorem 1.1}
We shall restrict ourselves here to considering the solutions of (\ref{BI7}), (\ref{BJ7}) when the initial data $v_\ve(\cdot,0)$ is non-negative,  increasing and also the function $\Gamma_\ve(\cdot, 0)$ of (\ref{BK7})  is non-negative.  The condition (\ref{BW7}) becomes now  $v_\ve(0,t)\le \ve^{-1}$,  and assuming this holds also, we see that in this case (\ref{BI7}) may be solved by the method of characteristics.  To carry this out we set $\tilde{v}_\ve(x,t)=m_{1,1/\La_\ve}(t)v_\ve(x,t)$, where $m_{1,A}(\cdot)$ is defined by (\ref{B7}).  Then (\ref{BI7}) is equivalent to 
\be \label{W6}
\frac{\pa \tilde{v}_\ve(x,t)}{\pa t}+\left[\frac{x}{\La_\ve(t)}-1+\ve \frac{\tilde{v}_\ve(x,t)}{m_{1,1/\La_\ve}t)}\right]\frac{\pa \tilde{v}_\ve(x,t)}{\pa x}
  \ = \  0 \ .
\ee
From (\ref{W6}) it follows that if $x(s), \ s\ge 0,$ is a solution to the ODE
\be \label{X6}
\frac{dx(s)}{ds} \ = \ \frac{x(s)}{\La_\ve(s)}-1+\ve \frac{\tilde{v}_\ve(x(s),s)}{m_{1,1/\La_\ve}(s)} \ ,
\ee
and characteristics do not intersect, then $\tilde{v}_\ve(x(t),t) \ = \ \tilde{v}_\ve (x(0),0)$ for $t\ge 0$. We can therefore calculate the characteristics of (\ref{W6}) by setting $  \tilde{v}_\ve(x(s),s)=\tilde{v}_\ve(x(0),0)=v_\ve(x(0),0)$.  We define the  function $F_{\ve,A}(x,t, v_0(\cdot))$ depending on $x,t\ge 0$ and increasing function $v_0:[0,\infty)\ra (0,\infty)$ by 
\be \label{Y6}
z+\ve \frac{\sig_A^2(t)}{m_{1,A}(t)^2}v_0(z) \ = \ \frac{x+m_{2,A}(t)}{m_{1,A}(t)} \ = \ F_A(x,t) \ , \quad  F_{\ve,A}(x,t, v_0(\cdot)) \ = \ z,
\ee
where  $F_A$ is given by (\ref{B7}) and $\sig_A^2(\cdot)$  by (\ref{J2}).  Since $v_0(\cdot)$ is an increasing function there is a unique solution  $z$ to (\ref{Y6}) for all $x\ge 0$ provided $v_0(0)\le\ve^{-1} m_{1,A}(t)m_{2,A}(t)/\sig_A^2(t)$. If this condition holds then the method of characteristics now yields the solution to (\ref{BI7}) as
\be \label{Z6}
v_\ve(x,t) \ = \ \frac{1}{m_{1,A}(t)} v_\ve\left(F_{\ve,A}(x,t, v_\ve(\cdot,0)), 0\right) 
\ee
with $A=1/\La_\ve$. 
It follows from (\ref{Y6}), (\ref{Z6}) that $v_\ve(0,t) \ \le \  m_{2,A}(t)/\ve \sig_A^2(t)\le \ve^{-1}$. 

We wish to prove a global existence and uniqueness theorem for solutions of (\ref{BI7}), (\ref{BJ7}). To describe our assumptions on the initial data $v_\ve(\cdot,0)$ we shall consider  functions $v_0:[0,x_\infty)\ra\R^+$ with the properties:
\begin{multline} \label{AA6}
{\rm The \  function \ } x\ra v_0(x) \ {\rm is \ increasing \ on \ the \  interval \ } [0,x_\infty) \\ 
{\rm and \ } v_0(0)>0.  \ 
{\rm If \ } x_\infty<\infty \ {\rm then \ } \lim_{x\ra x_\infty} v_0(x)=\infty \ .
\end{multline} 
\be \label{AB6}
 v_0(x_2)-v_0(x_1)  \ \le \ \int_{x_1}^{x_2} v_0(x)^2 \ dx \quad {\rm for \ } 0\le x_1<x_2<x_\infty \ .
\ee
Note that (\ref{AB6}) implies that $v_0(\cdot)$ is locally Lipschitz continuous in the interval $[0,x_\infty)$. 
\begin{lem}
Assume the function $ v_0(\cdot)=v_\ve(\cdot,0)$ satisfies (\ref{AA6}), (\ref{AB6}) and in addition that $(1+\del_0)v_\ve(0,0)<\ve^{-1}$ for some $\del_0>0$. Then there exists $\del_1>0$ depending only on $\del_0$ such that there is a unique solution to (\ref{BI7}), (\ref{BJ7}) for $0\le t\le T=\del_1/v_\ve(0,0)$.
\end{lem}
\begin{proof}
Let $T,\del_2>0$ and $\mathcal{E}$ be the space of  continuous functions $V:[0,T]\ra\R^+$ satisfying
\be \label{AC6}
V(0) \ = \ v_\ve(0,0), \quad (1+\del_2)^{-1}v_\ve(0,0) \ \le \ V(t) \ \le (1+\del_2) v_\ve(0,0) \quad {\rm for \ } 0\le t\le T. 
\ee
For $V\in\mathcal{E}$ we define a function $\mathcal{B}V(t), \ 0\le t\le T,$ by $\mathcal{B}V(t)=v_\ve(0,t)$ where $v_\ve$ is the function (\ref{Z6}) with $A(s)=V(s), \ 0\le s\le T$. We shall show that if  $T>0$ is sufficiently small then $V\in\mathcal{E}$ implies $\mathcal{B}V\in\mathcal{E}$.  To see this we first observe from (\ref{Y6}) that $\mathcal{B}V(0)=v_\ve(0,0)$. Next we note that for a function $v_0(\cdot)$ satisfying (\ref{AA6}), (\ref{AB6}), then
\be \label{AD6}
v_0(0) \ \le \ v_0(z)  \le \frac{v_0(0)}{1-zv_0(0)} \ , \quad  {\rm for \ } 0 \ \le z \ < 1/v_0(0) \ .
\ee
Since $V\in\mathcal{E}$, it follows from (\ref{AC6}) that with $A(\cdot)=V(\cdot)$, then
\be \label{AE6}
t\exp\left[-(1+\del_2)v_\ve(0,0)t\right] \ \le \ \frac{m_{2,A}(t)}{m_{1,A}(t)} \ \le \ t \ , \quad 0\le t\le T.
\ee
Similarly we have that
\be \label{AF6}
t\exp\left[-2(1+\del_2)v_\ve(0,0)t\right] \ \le \ \frac{\sig_A^2(t)}{m_{1,A}(t)^2} \ \le \ t \ , \quad 0\le t\le T.
\ee
From (\ref{AE6}), (\ref{AF6}) we have that for $\del_1,\del_2>0$ sufficiently small, depending only on $\del_0$,  that
\be \label{AG6}
\ve \frac{\sig_A^2(t)}{m_{1,A}(t)^2}v_\ve(0,0) \ \le \   \frac{m_{2,A}(t)}{m_{1,A}(t)} \quad {\rm for \ }  0\le t\le T.
\ee
Hence there is a unique solution $z(t)\le m_{2,A}(t)/m_{1,A}(t)$ to (\ref{Y6}) with $x=0$ provided $0\le t\le T$.   Since 
$\mathcal{B}V(t) \ = \ v_\ve(z(t),0)/m_{1,A}(t)$, it follows from  (\ref{AD6}), (\ref{AE6}) that on choosing $\del_1>0$ sufficiently small, depending only on  $\del_2$, that the function $\mathcal{B}V(t), \ 0\le t\le T$, also satisfies (\ref{AC6}). 

Next we show that $\mathcal{B}$ is a contraction on the space $\mathcal{E}$ with metric $d(V_1,V_2)=\sup_{0\le t\le T}|V_2(t)-V_1(t)|$ for $V_1,V_2\in\mathcal{E}$. To see this let $z_1(t), \ z_2(t), \ 0\le t\le T,$ be the solutions to (\ref{Y6}) with $x=0$ corresponding to $V_1,V_1\in\mathcal{E}$ respectively.   Then from (\ref{Z6}) we have
\be \label{AH6}
\mathcal{B}V_2(t)-\mathcal{B}V_1(t) \ = \  \frac{v_\ve(z_2(t),0)-v_\ve(z_1(t),0)}{m_{1,V_2}(t)}+
\left[\frac{1}{m_{1,V_2}(t)}-\frac{1}{m_{1,V_1}(t)}\right]v_\ve(z_1(t),0) \ .
\ee
The second term on the RHS of  (\ref{AH6}) is bounded as
\begin{multline} \label{AI6}
\left| \ \left[\frac{1}{m_{1,V_2}(t)}-\frac{1}{m_{1,V_1}(t)}\right]v_\ve(z_1(t),0) \ \right| \ \le \\
  (1+\del_2)v_\ve(0,0)\int_0^t |V_2(s)-V_1(s)| \ ds \ , \quad 0\le t\le T.
\end{multline}
We use (\ref{AB6}) to bound the first term in (\ref{AH6}).  Thus we have that
\be \label{AJ6}
\left| \ \frac{v_\ve(z_2(t),0)-v_\ve(z_1(t),0)}{m_{1,V_2}(t)} \ \right| \ \le \  (1+\del_2)^2v_\ve(0,0)^2 |z_2(t)-z_1(t)| \ , \quad 0\le t\le T.
\ee
From (\ref{Y6}), (\ref{AA6}) it follows that
\be \label{AK6}
|z_2(t)-z_1(t)| \ \le \ \left| \ \frac{m_{2,V_2}(t)}{m_{1,V_2}(t)}-\frac{m_{2,V_1}(t)}{m_{1,V_1}(t)} \ \right| \ , \quad 0\le t\le T.
\ee
The RHS of (\ref{AK6}) can be bounded  similarly to  (\ref{AI6}), and so we obtain the inequality
\be \label{AL6}
|z_2(t)-z_1(t)| \ \le \ t\int_0^t |V_2(s)-V_1(s)| \ ds \ , \quad 0\le t\le T.
\ee
It follows from (\ref{AH6})-(\ref{AL6}) that
\be \label{AM6}
|\mathcal{B}V_2(t)-\mathcal{B}V_1(t) | \ \le \ 10 \del_1\sup_{0\le t\le T}|V_2(t)-V_1(t)| \ , \quad 0\le t\le T=\del_1/v_\ve(0,0),
\ee
provided $\del_1>0$ is chosen sufficiently small depending only on $\del_2$.  Evidently $\mathcal{B}$ is a contraction mapping on $\mathcal{E}$ and therefore has a unique fixed point  if one also has $10\del_1<1$. 
\end{proof}
\begin{lem}
Let $v_\ve(x,t), \ x\ge0, \ 0\le t\le T$ be the solution to (\ref{BI7}), (\ref{BJ7}) constructed in Lemma 3.1. Then for any $t$ satisfying $0< t\le T$ the function  $v_0(\cdot)=v_\ve(\cdot,t)$ satisfies (\ref{AA6}), (\ref{AB6}) with $x_\infty=\infty$.  In addition  the function $t\ra v_\ve(0,t), \ 0\le t\le T,$ is continuous and decreasing. 
\end{lem}
\begin{proof}
Since $\ve>0$ it follows from the fact that (\ref{AA6})  holds for $v_0(\cdot)=v_\ve(\cdot,0)$ that (\ref{Y6}) has a unique solution $z<x_\infty$ for any $x>0$. Hence $x_\infty=\infty$ if $t>0$,  and it is also clear that the function $x\ra v_\ve(x,t), \ x\ge0,$ is increasing. We have therefore shown that (\ref{AA6}) holds for $v_0(\cdot)=v_\ve(\cdot,t)$ and $x_\infty=\infty$  if $0<t\le T$. 

Next we wish to show that (\ref{AB6}) holds for $v_0(\cdot)=v_\ve(\cdot,t)$ with $0<t\le T$. To see this we observe from (\ref{Z6}), (\ref{AB6}) that for $0\le x_1\le x_2<\infty$, 
\be \label{AN6}
v_\ve(x_2,t)-v_\ve(x_1,t) \ = \ \frac{v_\ve(z(x_2,t),0)-v_\ve(z(x_1,t),0)}{m_{1,1/\La_\ve}(t)} \ \le   \
\frac{1}{m_{1,1/\La_\ve}(t)} \int_{z(x_1,t)}^{z(x_2,t)} v_\ve(z,0)^2 \ dz \ ,
\ee
where $z(x,t)=F_{\ve,1/\La_\ve}(x,t,v_\ve(\cdot,0))$. We see from (\ref{Y6}) that $0\le \pa z(x,t)/\pa x \le 1/m_{1,1/\La_\ve}(t)$, whence
\be \label{AO6}
\frac{1}{m_{1,1/\La_\ve}(t)} \int_{z(x_1,t)}^{z(x_2,t)} v_\ve(z,0)^2 \ dz \ = \ m_{1,1/\La_\ve}(t)\int_{x_1}^{x_2} v_\ve(x,t)^2 \frac{\pa z(x,t)}{\pa x} \ dx \ \le \ 
\int_{x_1}^{x_2} v_\ve(x,t)^2 \ dx \ .
\ee

To show that the function $t\ra v_\ve(0,t)$ is continuous and decreasing we write $v_\ve(0,t)=v_0(z(t))/m_{1,1/\La_\ve}(t)$ where $v_0(\cdot)=v_\ve(\cdot,0)$ and $z(t)$ is the solution $z$ to (\ref{Y6}) with $x=0$. We see from (\ref{Y6}) that the function $t\ra z(t)$ is Lipschitz continuous, whence the function $t\ra v_\ve(0,t)$ is continuous.  If $z\ra v_0(z)$ is differentiable at $z=z(t)$ then it follows from (\ref{B7})  that
\be \label{AP6}
\frac{ \pa v_\ve(0,t)}{\pa t} \ = \  \frac{1}{m_{1,1/\La_\ve}(t)}\left[ v_0'(z(t))\frac{dz(t)}{dt}-\frac{v_0(z(t))}{\La_\ve(t)}\right] \ .
\ee
Differentiating (\ref{Y6}) with respect to $t$ at $x=0$ we obtain the equation
\be \label{AQ6}
\left[1+\ve \frac{\sig_{1/\La_\ve}^2(t)}{m_{1,1/\La_\ve}(t)^2}v'_0(z(t))\right]\frac{dz(t)}{dt}  \ = \  \frac{1}{m_{1,1/\La_\ve}(t)}\left[1-\frac{\ve v_0(z(t))}{m_{1,1/\La_\ve}(t)}\right] \ .
\ee
Hence (\ref{AP6}), (\ref{AQ6}) imply that
\begin{multline} \label{AR6}
\left[1+\ve \frac{\sig_{1/\La_\ve}^2(t)}{m_{1,1/\La_\ve}(t)^2}v'_0(z(t))\right] m_{1,1/\La_\ve}(t)\frac{ \pa v_\ve(0,t)}{\pa t} \ = \\
 \frac{v'_0(z(t))}{m_{1,1/\La_\ve}(t)}-\frac{v_0(z(t))}{\La_\ve(t)}-\frac{\ve v_0'(z(t))v_0(z(t))}{m_{1,1/\La_\ve}(t)^2}\left[1+\frac{\sig_{1/\La_\ve}^2(t)}{\La_\ve(t)}\right] \ .
\end{multline}
From (\ref{BJ7}), (\ref{Z6}), (\ref{AB6}) we have that
\be \label{AS6}
\frac{v'_0(z(t))}{m_{1,1/\La_\ve}(t)} \ \le \ \frac{v_0(z(t))^2}{m_{1,1/\La_\ve}(t)} \ = \ v_0(z(t))v_\ve(0,t) \ = \ \frac{v_0(z(t))}{\La_\ve(t)} \ .
\ee 
We conclude from (\ref{AR6}), (\ref{AS6}) that $\pa v_\ve(0,t)/\pa t\le 0$. In the case when the function $z\ra v_0(z)$ is not differentiable at $z=z(t)$, we can do an approximation argument to see that the function $s\ra v_\ve(0,s)$ is decreasing close to $s=t$.  We have therefore shown that the function $t\ra v_\ve(0,t)$ is decreasing, whence $v_\ve(0,t)\le v_\ve(0,0)< \ve^{-1}$ for $0\le t\le T$. Since the RHS of (\ref{AQ6}) is the same as  $[1-\ve v_\ve(0,t)]/m_{1,1/\La_\ve}(t)$ this implies that the function $t\ra z(t)$ is increasing. 
\end{proof}
\begin{proposition}
Assume the initial data for  (\ref{BI7}), (\ref{BJ7}) satisfies the conditions of Lemma 3.1. Then there exists a unique continuous solution $v_\ve(x,t), \ x,t\ge 0,$ globally in time to (\ref{BI7}), (\ref{BJ7}).  The solution $v_\ve(\cdot,t)$ satisfies (\ref{AA6}), (\ref{AB6}) for $t>0$  with $x_\infty=\infty$, and the function $t\ra v_\ve(0,t)$ is  decreasing.  Furthermore there is a constant $C(\del_0)$ depending only on $\del_0$ such that $\La_\ve(t)\le \La_\ve(0)+C(\del_0)[\La_\ve(0)+t], \ t\ge 0$. 
\end{proposition}
\begin{proof}
The global existence and uniqueness follows immediately from Lemma 3.1, 3.2  upon using the fact that the function  $t\ra v_\ve(0,t)$ is decreasing.  To get the upper bound on the function $\La_\ve(\cdot)$ we observe that Lemma 3.1 implies that  with $T_0=0,$
\be \label{AT6}
\La_\ve(t) \ \le \  (1+\del_2)\La_\ve(T_{k-1}) \  {\rm for \ } T_{k-1} \ \le \ t  \ \le \ T_k, \quad T_k=T_{k-1}+\del_1\La_\ve(T_{k-1}) , \ k=1,2,... 
\ee
It follows from (\ref{AT6}) that
\be \label{AU6}
\La_\ve(t) \ \le \ (1+\del_2)(T_k-T_{k-1})/\del_1 \quad {\rm for  \ } T_{k-1}\le t\le T_k \ , \quad k=1,2,...
\ee
We also have  that
\be \label{AV6}
T_k-T_{k-1} \ \le \del_1(1+\del_2)\La_\ve(T_{k-2}) \ \le \ (1+\del_2)T_{k-1} \ , \quad k=2,3,..
\ee
From (\ref{AU6}), (\ref{AV6}) we conclude that $\La_\ve(t) \ \le (1+\del_2)^2t/\del_1$ provided  $t\ge T_1$, whence the result follows. 
\end{proof}

The upper bound on the coarsening rate implied by Proposition 3.1 is independent of $\ve>0$ as $\ve\ra 0$.  We can see from (\ref{AR6}) that a lower bound on the rate of coarsening depends on $\ve$. In fact if we choose $v_0(z)=1/[\La_\ve(0)-z], \ 0<z<\La_\ve(0),$ then $v'_0(z)=v_0(z)^2$, and so at $t=0$ the RHS of (\ref{AR6}) is zero if $\ve=0$.  The random variable $X_0$ corresponding to this initial data is simply $X_0\equiv {\rm constant}$, and it is easy to see that if $\ve =0$ then $X_t\equiv X_0$ for all $t>0$. For $\ve>0$ however, we have the following:
\begin{lem}
Assume the initial data for  (\ref{BI7}), (\ref{BJ7}) satisfies the conditions of Lemma 3.1 and let $z(t)$ be as in Lemma 3.2. Then $\lim_{t\ra\infty} z(t)=x_\infty$ and if   $\ve>0$ one has $\lim_{t\ra\infty}\La_\ve(t)=\infty$.
\end{lem} 
\begin{proof}
Let  $\lim_{t\ra\infty}\La_\ve(t)=\La_{\ve,\infty}$ and assume first that $\La_{\ve,\infty}<\infty$. In that case  $\lim_{t\ra\infty} m_{1,\La_\ve}(t)=\infty$, and hence (\ref{BJ7}), (\ref{Z6}) imply that $\lim_{t\ra\infty}v_0(z(t))=\infty$. We conclude from (\ref{AA6}), (\ref{AB6}) that if $\La_{\ve,\infty}<\infty$ then $\lim_{t\ra\infty} z(t)=x_\infty$ . We also have from (\ref{B7}) that
\be \label{AW6}
\limsup_{t\ra\infty} \frac{m_{2,1/\La_\ve}(t)}{m_{1,1/\La_\ve}(t)} \ \le \  \La_{\ve,\infty} \ , \quad \liminf_{t\ra\infty} \frac{\sig_{1/\La_\ve}^2(t)}{m_{1,1/\La_\ve}(t)^2} \ \ge \ \frac{\La_\ve(0)}{2} \ .
\ee
It follows now from (\ref{Y6}), (\ref{AW6}) that if $\ve>0$ then $\limsup_{t\ra\infty} v_0(z(t))\le 2\La_{\ve,\infty}/\ve\La_\ve(0)<\infty$, which yields a contradiction. 

Next we assume that $\La_{\ve,\infty}=\infty$, which we have just shown always holds if $\ve>0$.   The function $t\ra z(t)$ is increasing, and let us suppose that $\lim_{t\ra\infty} z(t)=z_\infty<x_\infty$. Then from  (\ref{BJ7}), (\ref{Z6}) we have that $\lim_{t\ra\infty} m_{1,1/\La_\ve}(t)=\infty$.  We use the fact that for any $T\ge 0,$ there exists  a constant $K_T$ such that
\be \label{AX6}
\frac{\sig_{1/\La_\ve}^2(t)}{m_{1,1/\La_\ve}(t)^2} \ \le \ \frac{m_{2,1/\La_\ve}(t)}{m_{1,1/\La_\ve}(T)m_{1,1/\La_\ve}(t)}+K_T, \quad {\rm for \ } t\ge T.
\ee
From (\ref{Y6}),  (\ref{AX6}) we obtain the inequality
\be \label{AY6}
\frac{m_{2,1/\La_\ve}(t)}{m_{1,1/\La_\ve}(t)} \ \le \ z_\infty+\ve\left[  \frac{m_{2,1/\La_\ve}(t)}{m_{1,1/\La_\ve}(T)m_{1,1/\La_\ve}(t)}+K_T \right]v_0(z_\infty) \quad {\rm if \ }t\ge T. 
\ee
Choosing $T$ sufficiently large so that $m_{1,1/\La_\ve}(T)\ge 2\ve$, we  conclude from (\ref{AX6}), (\ref{AY6}) that there is a constant $C_1$ such that  $\sig_{1/\La_\ve}^2(t)/m_{1,1/\La_\ve}(t)^2\le C_1$ for $t\ge T$.  If we also choose $T$ such that $m_{1,1/\La_\ve}(t)\ge \ve v_0(z(t))/2$ for $t\ge T$ we have from (\ref{AB6}), (\ref{AQ6})  that 
\be \label{AZ6}
\left[1+C_1\ve v_0(z_\infty)^2\right]\frac{dz(t)}{dt} \ \ge \  \frac{1}{2m_{1,1/\La_\ve}(t)} \  \quad {\rm for \ } t\ge T. 
\ee
Since the function $t\ra z(t)$ is increasing and $\lim_{t\ra\infty} z(t)=z_\infty<\infty$, it follows from (\ref{BJ7}), (\ref{Z6}), (\ref{AZ6}) that there is a constant $C_2$ such that 
\be \label{BA6}
\int_0^t \frac{ds}{\La_\ve(s)} \ \le \ v_0(z_\infty)\int_0^t \ \frac{ds}{m_{1,1/\La_\ve}(s)} \ \le C_2 \quad {\rm for \ } t  \ge 0. 
\ee
However (\ref{BA6}) implies that $\lim_{t\ra\infty} m_{1,1/\La_\ve}(t)\le \exp[C_2]$ and so we have again a contradiction.  We conclude that $\lim_{t\ra \infty} z(t)=x_\infty$. 
\end{proof}
\begin{lem}
Assume the initial data for  (\ref{BI7}), (\ref{BJ7}) satisfies the conditions of Lemma 3.1  with $x_\infty=\infty$, and  that for $0<\del\le 1$, one has
\be \label{BB6}
\limsup_{x\ra\infty} \frac{v_\ve\left(x+\del/v_\ve(x,0),0\right)}{v_\ve(x,0)} \ \le \ 1+\ga(\del), \quad {\rm where \ } 
\lim_{\del\ra 0}\frac{\ga(\del)}{\del}=0.
\ee
Then $\lim_{t\ra\infty} \La_\ve(t)/t=1$ for any $\ve\ge 0$. 
\end{lem}
\begin{proof}
The main point about the condition (\ref{BB6}) is that it is invariant under the dynamics determined by (\ref{BI7}), (\ref{BJ7}). It is easy to see this in the case $\ve =0$ since we have, on using the notation of (\ref{Y6}), that
\be \label{BC6}
\frac{v_0\left(x+\del/v_0(x,t),t\right)}{v_0(x,t)} \ =  \ \frac{v_0\left(F_{1/\La_0}(x,t)+\del/v_0(F_{1/\La_0}(x,t),0),0\right)}{v_0(F_{1/\La_0}(x,t),0)} \ .
\ee
 
For $\ve>0$ we have
 \be \label{BD6}
\frac{v_\ve\left(x+\del/v_\ve(x,t),t\right)}{v_\ve(x,t)} \ =   \ \frac{v_\ve\left(z(x+\del/v_\ve(x,t),t),0\right)}{v_\ve(z(x,t),0)}  \  .
\ee
Since the function $x\ra v_\ve(x,0)$ is increasing it follows from (\ref{Y6}) that
\be \label{BE6}
z(x+\del/v_\ve(x,t),t) \ \le \ z(x,t)+\frac{\del}{m_{1,1/\La_\ve}(t)v_\ve(x,t)} \ .
\ee 
We conclude now from (\ref{Z6}), (\ref{BD6}), (\ref{BE6}) that
\be \label{BF6}
\frac{v_\ve\left(x+\del/v_\ve(x,t),t\right)}{v_\ve(x,t)} \ \le  \ \frac{v_\ve\left(z(x,t)+\del/v_\ve(z(x,t),0),0\right)}{v_\ve(z(x,t),0)}  \  .
\ee

From Lemma 3.3 and (\ref{BF6}) there exists $T_0\ge 0$ such that
\be \label{BG6}
\frac{v_\ve\left(x+\del/v_\ve(x,t),t\right)}{v_\ve(x,t)} \ \le  \ 1+2\ga(\del) \quad {\rm for \ } x\ge 0, \ t\ge T_0  \  .
\ee
We use (\ref{BG6}) to estimate $\La_\ve(T_0+t)/\La_\ve(T_0)$ in the interval $0\le t\le \del/v_\ve(0,T_0)$.  Thus we have
\be \label{BH6}
\frac{\La_\ve(T_0)}{\La_\ve(T_0+t)} \ = \ \frac{v_\ve(z(t),T_0)m_{1,1/\La_\ve}(T_0)}{v_\ve(0,T_0)m_{1,1/\La_\ve}(T_0+t)} \ , \quad {\rm where \ } 0\le z(t)\le t. 
\ee
We conclude from (\ref{BG6}), (\ref{BH6}) that
\be \label{BI6}
\frac{m_{1,1/\La_\ve}(T_0)}{\La_\ve(T_0)} \ \le \ \frac{dm_{1,1/\La_\ve}(T_0+t)}{dt} \ \le \ [1+2\ga(\del)]\frac{m_{1,1/\La_\ve}(T_0)}{\La_\ve(T_0)}  \quad {\rm for \ } 0\le t\le \del\La_\ve(T_0).
\ee
On integrating (\ref{BI6})  we have
\be \label{BJ6}
1+\frac{t}{\La_\ve(T_0)} \ \le \ \frac{m_{1,1/\La_\ve}(T_0+t)}{m_{1,1/\La_\ve}(T_0)} \ \le \ 1+\frac{[1+2\ga(\del)]t}{\La_\ve(T_0)} 
\quad {\rm for \ } 0\le t\le \del\La_\ve(T_0).
\ee
Hence (\ref{BH6}), (\ref{BJ6}) imply that
\be \label{BK6}
\frac{1}{1+2\ga(\del)}\left[1+\frac{t}{\La_\ve(T_0)}\right]  \ \le \ \frac{\La_\ve(T_0+t)}{\La_\ve(T_0)} \ \le \ 1+\frac{[1+2\ga(\del)]t}{\La_\ve(T_0)} \quad {\rm for \ } 0\le t\le \del\La_\ve(T_0).
\ee

We define now $T_1>T_0$ as the minimum time $T_1=T_0+t$ such that $\La_\ve(T_0+t)\ge (1+\del/2)\La_\ve(T_0)$.  The inequality (\ref{BK6}) now yields bounds on $T_1-T_0$ as
\be \label{BL6}
\frac{\del/2}{1+2\ga(\del)} \  \le \ \frac{T_1-T_0}{\La_\ve(T_0)}  \ \le \ [1+2\ga(\del)][1+\del/2]-1  \ ,
\ee
provided the RHS of (\ref{BL6}) is less than $\del$.  In view of (\ref{BB6}) this will be the case if $\del>0$ is sufficiently small.  
We can iterate the inequality (\ref{BL6}) by defining $T_k, \ k=1,2,..,$, as the minimum time such that $\La_\ve(T_k)\ge (1+\del/2)\La_\ve(T_{k-1})$. Thus we have that
\be \label{BM6}
\frac{\del/2}{1+2\ga(\del)} \  \le \ \frac{T_k-T_{k-1}}{(1+\del/2)^{k-1}\La_\ve(T_0)}  \ \le \ [1+2\ga(\del)][1+\del/2]-1  \ , \quad k=1,2,...
\ee
On summing (\ref{BM6}) over $k=1,...,N$ we conclude that
\be \label{BN6}
\left[1-\frac{1}{(1+\del/2)^N}\right]\frac{1}{1+2\ga(\del)} \  \le \ \frac{T_N-T_{0}}{\La_\ve(T_N)}  \ \le \ \left[1-\frac{1}{(1+\del/2)^N}\right]\frac{[1+2\ga(\del)][1+\del/2]-1}{\del/2} \ .
\ee
It follows from (\ref{BN6}) that
\be \label{BO6}
\frac{1}{(1+\del/2)[1+2\ga(\del)]} \ \le \ \liminf_{t\ra\infty} \frac{t}{\La_\ve(t)}  \ \le \ \limsup_{t\ra\infty} \frac{t}{\La_\ve(t)} \ \le \ (1+\del/2)\frac{[1+2\ga(\del)][1+\del/2]-1}{\del/2}  \ .
\ee
Now using the fact that $\lim_{\del\ra 0}\ga(\del)/\del=0$, we conclude from (\ref{BO6}) that $\lim_{t\ra\infty}\La_\ve(t)/t=1$. 
\end{proof}
\begin{rem}
Theorem 5.4 of \cite{cp} implies in the case $\ve=0$  convergence to the exponential self similar solution for initial data $v_\ve(x,0), \ x\ge 0,$ which has the property $\lim_{x\ra\infty} v_\ve(x,0)/x^\al=v_{\ve,\infty}$ with $0<v_{\ve,\infty}<\infty$ provided $\al>-1$. It is easy to see that if $\al\ge 0$ then such initial data satisfies the condition (\ref{BB6}) of Lemma 3.4. 

In \cite{carr} necessary and sufficient conditions -(5.18) and (5.19) of \cite{carr}- for convergence to the exponential self-similar solution are obtained in the case $\ve=0$. Note that (5.19) of \cite{carr} implies the condition (\ref{BB6}) of Lemma 3.4.  
\end{rem}

Next we  obtain a rate of convergence theorem for $\lim_{t\ra\infty}\La_\ve(t)/t$ which generalizes Proposition 2.1 to the system (\ref{BI7}), (\ref{BJ7}).  We assume that the function $x\ra v_\ve(x,0)$ is $C^1$ for large $x$, in which case the condition (\ref{BB6}) becomes $\lim_{x\ra\infty} v_\ve(x,0)^{-2}\pa v_\ve(x,0)/\pa x=0$, or equivalently $\lim_{x\ra\infty}\beta_{X_0}(x)=1$ for the initial condition random variable $X_0$. More precisely, we have that if  $v_\ve(x,0)^{-2}\pa v_\ve(x,0)/\pa x\le\eta$ for $x\ge x_\eta$ then
\be \label{BP6}
\frac{1}{v_\ve(x,0)} - \frac{1}{v_\ve(x+\del/v_\ve(x,0),0)} \ \le \  \frac{\eta\del}{v_\ve(x,0)} \quad  {\rm for \ } x\ge x_\eta \ .
\ee 
The inequality (\ref{BP6}) implies (\ref{BB6}) holds with $\ga(\del)\le\eta\del/(1-\eta\del)$. We conclude that (\ref{BB6}) holds if $\lim_{x\ra\infty} v_\ve(x,0)^{-2}\pa v_\ve(x,0)/\pa x=0$.

The condition on the initial data to guarantee a logarithmic rate of convergence for $\La_\ve(t)/t$ is similar to (\ref{BB6}). We require that there exists $\del,\ga(\del),x_\del>0$  such that
\be \label{BQ6}
\frac{v_\ve\left(y,0\right)^2}{\pa v_\ve(y,0)/\pa y}  \ \ge \ \frac{v_\ve\left(x,0\right)^2}{\pa v_\ve(x,0)/\pa x} +\del\ga(\del)\quad {\rm for \ } y=x+\del/v_\ve(x,0) , \ x\ge x_\del \ .
\ee
Observe that if (\ref{BQ6}) holds for arbitrarily small $\del>0$ and $\liminf_{\del\ra0}x_\del=x_0$ then the function $x\ra 1/v_\ve(x,0)$ is convex for $x\ge x_0$. Furthermore if the function $x\ra v_\ve(x,0)$ is $C^2$, then on taking $\del\ra 0$ in (\ref{BQ6}) we obtain the second order differential inequality
\be \label{BR6}
\frac{v_\ve(x,0)\pa^2 v_\ve(x,0)/\pa x^2}{[\pa v_\ve(x,0)/\pa x]^2} \ \le \ 2-\ga(0) \ .
\ee
Suppose now that (\ref{BR6})  holds with $\ga(0) = \eta>0$ for $x\ge x_0$. Then we have that
\be \label{BS6}
\frac{\pa}{\pa x} \frac{v_\ve(x,0)^2}{\pa v_\ve(x,0)/\pa x}  \ge \eta v_\ve(x,0) \quad {\rm for \ } x\ge x_0 \ .
\ee
On integrating (\ref{BS6}) and using the fact that the function $x\ra v_\ve(x,0)$ is increasing, we conclude that (\ref{BQ6}) holds for all $\del>0$ with $\ga(\del)=\eta$ and $x_\del=x_0$. 

It is easy to see that (\ref{BR6}) is invariant under affine transformations. That is if the function $x\ra v_\ve(x,0)$ satisfies (\ref{BR6}) for all $x>0$, then given any $\la,k>0$ so also does the function $x\ra\la v_\ve(\la x+k,0)$. We can solve the differential equation determined by equality in (\ref{BR6}). The solution is given by the formula
\be \label{BT6}
v_\ve(x,0) \ = \  a[1+\la x]^\al \ , \quad {\rm where \ } \al \ = \ 1/[\ga(0)-1] \ .
\ee
Since we require $\ga(0)>0$ it follows from (\ref{BR6}) that $\al$ must satisfy either $\al> 0$ or $\al<-1$. Note that the function $x\ra 1/v_\ve(x,0)$ of (\ref{BT6}) is convex precisely for this range of $\al$ values. 
\begin{lem}
Assume the initial data $x\ra v_\ve(x,0)$ for  (\ref{BI7}), (\ref{BJ7}) is $C^1$ increasing and that the function $x\ra1/v_\ve(x,0)$ is convex for sufficiently large $x$.  Assume further that there exists $\del,\ga(\del),x_\del>0$ such that (\ref{BQ6}) holds. Then there exists constants $C_0,t_0>0$ such that
\be \label{BU6}
1-\frac{C_0}{\log t} \ \le \ \frac{d\La_\ve(t)}{dt} \ \le 1 \  \quad {\rm for \ } t\ge t_0 \ .
\ee 
 \end{lem}
\begin{proof}
 Since the inequality (\ref{BQ6}) is invariant under affine transformations we see as in Lemma 3.4 that in the case $\ve=0$ there exists  $T_0>0$ such that if $t\ge T_0$ the function $x\ra 1/v_\ve(x,t)$ is convex for $x\ge 0$, and 
\be \label{BV6}
\frac{v_\ve\left(y,t\right)^2}{\pa v_\ve(y,t)/\pa y}  \ \ge \ \frac{v_\ve\left(x,t\right)^2}{\pa v_\ve(x,t)/\pa x} +\del\ga(\del)\quad {\rm for \ } y=x+\del/v_\ve(x,t) , \ x\ge 0, \ t\ge T_0 \ .
\ee

Next observe that since $\lim_{x\ra\infty} v_\ve(x,0)^{-2}\pa v_\ve(x,0)/\pa x=0$, we may for any $\nu>0$ choose $T_0$ such that   $v_\ve(x,T_0)^{-2}\pa v_\ve(x,T_0)/\pa x\le \nu$  for $x\ge 0$.  It follows then from (\ref{BP6}) that
\be \label{BW6}
\frac{v_\ve(y,T_0)}{v_\ve(0,T_0)} \ \le \ \frac{1}{1-\nu v_\ve(0,T_0)y} \ \quad {\rm for \ } 0\le y< \frac{1}{\nu v_\ve(0,T_0)} \ = \ \frac{\La_\ve(T_0)}{\nu} \ .
\ee
Hence as in (\ref{BI6}) we see from (\ref{BH6}), (\ref{BW6})  that 
\be \label{BX6}
\frac{dm_{1,1/\La_\ve}(T_0+t)}{dt} \ \le \ \frac{m_{1,1/\La_\ve}(T_0)}{[1-\nu v_\ve(0,T_0)t]\La_\ve(T_0)}  \quad {\rm for \ } 0\le t< \frac{\La_\ve(T_0)}{\nu} \ .
\ee
Integrating (\ref{BX6}) we conclude that for $0\le t< \La_\ve(T_0)/\nu$,
\be \label{BY6}
\frac{m_{1,1/\La_\ve}(T_0)}{m_{1,1/\La_\ve}(T_0+t)} \ \ge \ \left[1-\frac{1}{\nu}\log\left\{1-\frac{\nu t}{\La_\ve(T_0)}\right\}\right]^{-1} \ .
\ee
Using the inequality $-\log(1-z)\le 3z/2$ when $0\le z\le 1/3$, we conclude from (\ref{BY6}) that
\be \label{BZ6}
\frac{m_{1,1/\La_\ve}(T_0)}{m_{1,1/\La_\ve}(T_0+t)} \ \ge \ \left[1+\frac{3 t}{2\La_\ve(T_0)}\right]^{-1}  \quad {\rm for \ } 0\le t\le \frac{\La_\ve(T_0)}{3\nu} \ .
\ee
Similarly to (\ref{BK6}) we have from (\ref{BZ6})  that
 \be \label{CC6}
 \frac{\La_\ve(T_0+t)}{\La_\ve(T_0)}  \ \le \ 1+\frac{3 t}{2\La_\ve(T_0)}  \quad {\rm for \ } 0\le t\le \frac{\La_\ve(T_0)}{3\nu} \ .
 \ee

 In the case $\ve =0$ the LHS of (\ref{BZ6}) is $dz(t)/dt$, so on integration we have that 
 \be \label{CA6}
 z(t) \ \ge \ \frac{2\La_\ve(T_0)}{3}\log\left[1+\frac{3 t}{2\La_\ve(T_0)}\right] \quad {\rm for \ } 0\le t\le \frac{\La_\ve(T_0)}{3\nu}  \ .
 \ee
 We choose now $\nu$ sufficiently small so that $2\log[1+1/2\nu]/3>\del$ and let $T_1$ be the minimum $T_0+t$ such that $z(t)\ge \del/v_\ve(0,T_0)$.  Then we have that
 \be \label{CB6}
 T_1-T_0 \ \le \ \frac{\La_\ve(T_0)}{3\nu} \ , \quad \frac{v_\ve\left(0,T_1\right)^2}{\pa v_\ve(0,T_1)/\pa x}  \ \ge \ \frac{v_\ve\left(0,T_0\right)^2}{\pa v_\ve(0,T_0)/\pa x} +\del\ga(\del) \ .
 \ee
 Furthermore (\ref{CC6}) implies that $\La_\ve(T_1)/\La_\ve(T_0)\le 1+1/2\nu$.  We now iterate the foregoing to yield a sequence of times $T_k, \ k=1,2,..,$ with the properties that
 \be \label{CD6}
 T_k-T_{k-1} \ \le \ \frac{\La_\ve(T_{k-1})}{3\nu}, \quad  \frac{\La_\ve(T_k)}{\La_\ve(T_{k-1})} \ \le 1+\frac{1}{2\nu}, \quad \frac{v_\ve\left(0,T_k\right)^2}{\pa v_\ve(0,T_k)/\pa x} \ \ge  \ k\del\ga(\del).
 \ee
 It follows from (\ref{CD6}) that
 \be \label{CE6}
 T_N-T_0 \ \le \ \frac{2}{3} \left(1+\frac{1}{2\nu}\right)^N\La_\ve(T_0), \quad \frac{d\La_\ve(T_N)}{dt} \ \ge 1-\frac{1}{N\del\ga(\del)}, \quad N=1,2,...
 \ee
 The inequality (\ref{CE6}) implies the lower bound in (\ref{BU6}) since the function $t\ra d\La_\ve(t)/dt$ is increasing for $t\ge T_0$.  
 
 To deal with $\ve>0$ we first assume that the function $x\ra v_\ve(x,0)$ is $C^2$ for $x>0$. Letting $z(x,t)$ be the solution to (\ref{Y6}) we have that 
 \be \label{CF6}
 \frac{\pa z(x,t)}{\pa x} \ = \ \frac{1}{m_{1,1/\La_\ve}(t)} \left[1+\ve \frac{\sig_{1/\La_\ve}^2(t)}{m_{1,1/\La_\ve}(t)^2} \frac{\pa v_\ve(z(x,t),0)}{\pa z}\right]^{-1} \ ,
 \ee
 \be \label{CG6}
 \frac{\pa^2 z(x,t)}{\pa x^2} \ = \ -\ve\frac{\sig_{1/\La_\ve}^2(t)}{m_{1,1/\La_\ve}(t)^2} \frac{\pa^2 v_\ve(z(x,t),0)}{\pa z^2}  \left[1+\ve \frac{\sig_{1/\La_\ve}^2(t)}{m_{1,1/\La_\ve}(t)^2} \frac{\pa v_\ve(z(x,t),0)}{\pa z}\right]^{-1}\left( \frac{\pa z(x,t)}{\pa x}\right)^2 \ .
 \ee
 We also have that
 \be \label{CH6}
  \frac{\pa v_\ve(x,t)}{\pa x} \ = \ \frac{1}{m_{1,1/\La_\ve}(t)}  \frac{\pa v_\ve(z(x,t),0)}{\pa z}  \frac{\pa z(x,t)}{\pa x} ,
 \ee
\be \label{CI6}
  \frac{\pa^2 v_\ve(x,t)}{\pa x^2} \ = \ \frac{1}{m_{1,1/\La_\ve}(t)} \left[ \frac{\pa^2 v_\ve(z(x,t),0)}{\pa z^2} \left( \frac{\pa z(x,t)}{\pa x}\right)^2+
 \frac{\pa v_\ve(z(x,t),0)}{\pa z}  \frac{\pa^2 z(x,t)}{\pa x^2}\right]  .
 \ee
 It follows from (\ref{CF6})-(\ref{CI6}) that the ratio (\ref{BR6}) for the function $x\ra v_\ve(x,t)$ is given by 
 \be \label{CJ6}
\frac{v_\ve(x,t)\pa^2 v_\ve(x,t)/\pa x^2}{[\pa v_\ve(x,t)/\pa x]^2} \ = \  \frac{v_\ve(z(x,t),0)\pa^2 v_\ve(z(x,t),0)/\pa z^2}{[\pa v_\ve(z(x,t),0)/\pa z]^2}  \left[1+\ve \frac{\sig_{1/\La_\ve}^2(t)}{m_{1,1/\La_\ve}(t)^2} \frac{\pa v_\ve(z(x,t),0)}{\pa z}\right]^{-1} \ .
\ee
Hence if (\ref{BR6}) holds for all sufficiently large $x$, then Lemma 3.3 implies that there exists $T_0>0$ such that the RHS of (\ref{CJ6}) is bounded above by $2$ for $x\ge0,t\ge T_0$.  It follows  that the function $x\ra 1/v_\ve(x,t)$ is convex for $x\ge 0$ provided $t\ge T_0$. Observe that if $0<\ga(0)\le 2$ in (\ref{BR6}) then the RHS of (\ref{CJ6}) is bounded above by $2-\ga(0)$. However if $\ga(0)>2$ then we can only bound the RHS above by $0$. It is easy to see from the example (\ref{BT6}) that this is the best bound we can obtain.  In fact if $\al>1$ in (\ref{BT6}) then $\lim_{z\ra\infty}\pa v_\ve(z,0)/\pa z=\infty$, in which case the RHS of (\ref{CJ6}) converges to $0$ as $x\ra\infty$.  We have shown that if 
(\ref{BR6}) holds for all sufficiently large $x$ then there exists $T_0>0$ such that
 \be \label{CK6}
\frac{v_\ve(x,t)\pa^2 v_\ve(x,t)/\pa x^2}{[\pa v_\ve(x,t)/\pa x]^2} \ \le \ \max[2-\ga(0),0] \quad {\rm for \ } x\ge 0, \ t\ge T_0 \ .
\ee

In the case when we only assume that the function $x\ra v_\ve(x,0)$ is $C^1$ for $x>0$ we can make a more careful version of the argument of the previous paragraph.  We have now from (\ref{CF6}), (\ref{CH6}) that
\be \label{CL6}
\frac{v_\ve\left(x,t\right)^2}{\pa v_\ve(x,t)/\pa x} \ = \ \frac{v_\ve\left(z(x,t),0\right)^2}{\pa v_\ve(z(x,t),0)/\pa z} +\ve \frac{\sig_{1/\La_\ve}^2(t)}{m_{1,1/\La_\ve}(t)^2}v_\ve(z(x,t),0)^2 \ .
\ee
Since the function $x\ra z(x,t)$ is increasing  for $x\ge 0$, it follows that the second function on the RHS of (\ref{CL6}) is increasing for  $x\ge 0$.  Since we are assuming that the function $z\ra 1/v_\ve(z,0)$ is convex for all large $z$, it follows that the first function on the RHS of (\ref{CL6}) is also increasing for $x\ge 0$ provided $t\ge T_0$ and $T_0$ is sufficiently large.  We conclude that the function $x\ra 1/v_\ve(x,t)$ is convex for $x\ge 0$ provided $t\ge T_0$. 

We can also obtain an inequality (\ref{BV6}) for  a $\del$ which is  twice the $\del$ which occurs in  (\ref{BQ6}).  To show this we consider two possibilities. In the first of these (\ref{BV6}) follows from the monotonicity of the second function on the RHS of (\ref{CL6}).   We use the inequality
\begin{multline} \label{CM6}
\ve \frac{\sig_{1/\La_\ve}^2(t)}{m_{1,1/\La_\ve}(t)^2}v_\ve\left(z\left(x+\frac{\eta}{v_\ve(x,t)},t\right),0\right)^2 \  \ \ge \ \ve \frac{\sig_{1/\La_\ve}^2(t)}{m_{1,1/\La_\ve}(t)^2}v_\ve(z(x,t),0)^2+ \\
2\ve\eta \frac{\sig_{1/\La_\ve}^2(t)}{m_{1,1/\La_\ve}(t)^2}\int_0^1d\rho \ \frac{\pa v_\ve}{\pa z}\left(z\left(x+\frac{\rho\eta}{v_\ve(x,t)},t\right),0\right)\left[1+
\ve \frac{\sig_{1/\La_\ve}^2(t)}{m_{1,1/\La_\ve}(t)^2}\frac{\pa v_\ve}{\pa z}\left(z\left(x+\frac{\rho\eta}{v_\ve(x,t)},t\right),0\right) \ \right]^{-1} \ ,
\end{multline}
which follows from (\ref{Z6}), (\ref{CF6}), the monotonicity of the function $z\ra v_\ve(z,0)$,  and Taylor's formula.  Observe next from the convexity of the function $z\ra 1/v_\ve(z,0)$  that   for $0<\rho'<\rho$, 
\be \label{CN6}
\pa v_\ve\left(z\left(x+\frac{\rho\eta}{v_\ve(x,t)},t\right),0\right)\Big/\pa z \ \le \ 
\left[\frac{v_\ve(x+\rho\eta/v_\ve(x,t),t)}{v_\ve(x+\rho'\eta/v_\ve(x,t),t)}\right]^2\frac{\pa v_\ve}{\pa z}\left(z\left(x+\frac{\rho'\eta}{v_\ve(x,t)},t\right),0\right) \ .
\ee
Furthermore we have similarly to (\ref{BW6}) that for $\nu>0$  there exists $T_0>0$ such that for $0<\rho'<\rho$ and  $t\ge T_0$, 
\be \label{CO6}
1 \ \le \ \frac{v_\ve(x+\rho\eta/v_\ve(x,t),t)}{v_\ve(x+\rho'\eta/v_\ve(x,t),t)} \ \le  \  \frac{1-\nu \rho'\eta}{1-\nu\rho\eta} \quad {\rm provided \ } \rho<\frac{1}{\nu\eta} \ .
\ee
Now let us assume that $t\ge T_0$ and 
\be \label{CP6}
\ve \frac{\sig_{1/\La_\ve}^2(t)}{m_{1,1/\La_\ve}(t)^2}\frac{\pa v_\ve}{\pa z}\left(z\left(x+\frac{\del}{2v_\ve(x,t)},t\right),0\right) \ \ge \  \frac{1}{4} \ .
\ee
Then (\ref{CN6}), (\ref{CO6}) imply upon setting $\eta=\del/2$ in (\ref{CM6}) and choosing $\nu$ less than some constant depending only on $\del$,  that the second term on the RHS  is bounded below by  $\del/6$ for $t\ge T_0$.   This implies that (\ref{BV6}) holds with $\ga(\del)=1/6$. 

Alternatively we assume that (\ref{CP6}) does not hold.  Then on choosing $\nu$ sufficiently small, depending only on $\del$, we see that  (\ref{CF6}), (\ref{CN6}), (\ref{CO6}) implies
\be \label{CQ6}
z\left(x+\frac{2\del}{v_\ve(x,t)},t\right) \ge  \ z\left(x+\frac{\del}{2v_\ve(x,t)},t\right)+\frac{11\del}{10v_\ve(z(x,t),0)} \ .
\ee
Then we use the first term on the RHS of (\ref{CL6}) and (\ref{BQ6}), (\ref{CQ6}) to establish (\ref{BV6}) with $2\del$ in place of $\del$.   We have proved then that there exists $T_0>0$ such that (\ref{BV6}) holds  (with $\del$ replaced by $2\del$).  

We wish next to establish an inequality like (\ref{CB6}) in the case $\ve>0$, in which case we need to examine the terms  of (\ref{AQ6}) that depend on $\ve$.   Using the notation of (\ref{Y2}), (\ref{Z2}) the $\ve$ dependent coefficient on the LHS of (\ref{AQ6}) is given by
\be \label{CR6}
\ve \frac{\sig_{1/\La_\ve}^2(T_0,T_0+t)}{m_{1,1/\La_\ve}(T_0,T_0+t)^2}\frac{\pa v_\ve(z(t),T_0)}{\pa z} \ = \  \ve \frac{\sig_{1/\La_\ve}^2(T_0,T_0+t)}{\La_\ve(T_0+t)^2}\left[v_\ve(z(t),T_0)^{-2}\frac{\pa v_\ve(z(t),T_0)}{\pa z}\right] \ .
\ee
The $\ve$ dependent coefficient on the RHS of (\ref{AQ6}) is given by
\be \label{CS6}
\ve \frac{v_\ve(z(t),T_0)}{m_{1,1/\La_\ve}(T_0,T_0+t)} \ = \ \frac{\ve}{\La_\ve(T_0+t)} \ .
\ee
We choose now $T_0$ large enough so that $\ve/\La_\ve(T_0)<1/2$, whence (\ref{CS6}) implies that the term in brackets on the RHS of (\ref{AQ6}) is at least $1/2$.  We also have from (\ref{CR6}) that
\be \label{CT6}
\ve \frac{\sig_{1/\La_\ve}^2(T_0,T_0+t)}{m_{1,1/\La_\ve}(T_0,T_0+t)^2}\frac{\pa v_\ve(z(t),T_0)}{\pa z} \ \le \ \frac{\nu}{2}\frac{\sig_{1/\La_\ve}^2(T_0,T_0+t)}{\La_\ve(T_0+t)} \ .
\ee
Now using (\ref{Y2}), we conclude from (\ref{CT6}) that for any $K>0$, 
\be \label{CU6}
\ve \frac{\sig_{1/\La_\ve}^2(T_0,T_0+t)}{m_{1,1/\La_\ve}(T_0,T_0+t)^2}\frac{\pa v_\ve(z(t),T_0)}{\pa z} \ \le \  \frac{\nu K\exp(2K)}{2} \quad {\rm for \ } 0\le t\le K\La_\ve(T_0)  \ .
\ee
It follows from (\ref{BV6}), (\ref{CF6}),  (\ref{CU6}) that there exists $T_1>T_0$ such that (\ref{CB6}) holds. Therefore we can define a sequence $ T_k, \ k=1,2,..,$ of times having the properties (\ref{CD6}). 

In order to estimate $d\La_\ve(T_{k-1}+t)/dt$ for $0\le t\le T_k-T_{k-1}$,  we need to examine the terms  of (\ref{AR6}) that depend on $\ve$.  Similarly to (\ref{CU6}) we have from (\ref{BJ7}), (\ref{AR6}) and the convexity of the function $x\ra 1/v_\ve(x,T_{k-1})$  that
\be \label{CV6}
\left[1+\frac{\ve\nu K\exp(2K)}{\La_\ve(T_{k-1})} \right]\frac{d\La_\ve(T_{k-1}+t)}{dt} \ \ge 1-  \frac{\pa v_\ve(0,T_{k-1})/\pa x}{v_\ve(0,T_{k-1})^2} \quad {\rm for \ } 0\le t\le T_k-T_{k-1} \ .
\ee
Noting from (\ref{BK6}) that $\La_\ve(T_k)$ grows exponentially in $k$, we conclude  from (\ref{CD6}), (\ref{CE6}) and  (\ref{CV6}) that the lower bound in (\ref{BU6}) holds.  To obtain the upper bound in (\ref{BU6}) we use the identity
\be \label{CY6}
\frac{d\La_\ve(t)}{dt} \ = \ 1-\left[1-\ve v_\ve(0,t)\right]\frac{1}{v_\ve(0,t)^2}\frac{\pa v_\ve(0,t)}{\pa x} \ ,
\ee
obtained from (\ref{BI7}), (\ref{BJ7}).  Evidently the RHS of (\ref{CY6}) does not exceed $1$.
\end{proof}
\begin{lem}
Assume $\ve>0$ and the initial data for  (\ref{BI7}), (\ref{BJ7}) satisfies the conditions of Lemma 3.1  with $x_\infty<\infty$. Then for any $t>0$ the function $x\ra v_\ve(x,t)$ satisfies $\lim_{x\ra\infty} v_\ve(x,t)/x=v_\infty(t)$ for some $v_\infty(t)>0$. 

 Assume in addition that the initial data is $C^1$, the function $x\ra 1/v_\ve(x,0)$ is convex for $x$ sufficiently close to $x_\infty$, and $\liminf_{x\ra x_\infty} \pa v_\ve(x,0)/\pa x>0$. Then for any $t>0$ the function $x\ra 1/v_\ve(x,t)$ is convex for $x$ sufficiently large, and  the inequality (\ref{BQ6}) holds for all $\del>0$. 
\end{lem}
\begin{proof}
From (\ref{Y6}), (\ref{Z6}) we have that 
\be \label{CW6}
\frac{x+m_{2,1/\La_\ve}(t)-m_{1,1/\La_\ve}(t)x_\infty}{\ve\sig_{1/\La_\ve}^2(t)} \ \le \ v_\ve(x,t) \ \le \ \frac{x+m_{2,1/\La_\ve}(t)}{\ve\sig_{1/\La_\ve}^2(t)} \quad {\rm for \ } x\ge 0. 
\ee
Hence $\lim_{x\ra\infty} v_\ve(x,t)/x=1/\ve \sig_{1/\La_\ve}^2(t)$. 

It is easy to see  from our assumptions that (\ref{BQ6}) is satisfied if $v_\ve(x,0)$ is $C^2$ for $x$ sufficiently close to $x_\infty$. In that case it follows from  the convexity of the function $x\ra 1/v_\ve(x,0)$ close to $x_\infty$ and (\ref{CJ6}) that there exists $\eta(t)>0$ and $x_\eta(t)$ with
 \be \label{CX6}
\frac{v_\ve(x,t)\pa^2 v_\ve(x,t)/\pa x^2}{[\pa v_\ve(x,t)/\pa x]^2} \ \le \   2\left[1+\ve \frac{\sig_{1/\La_\ve}^2(t)}{m_{1,1/\La_\ve}(t)^2} \frac{\pa v_\ve(z(x,t),0)}{\pa z}\right]^{-1} \ \le 2-\eta(t) \ 
\ee
for $x>x_\eta(t)$.  If we only assume the function $x\ra v_\ve(x,t)$ is $C^1$, then we use  (\ref{CL6}), whence (\ref{BQ6}) follows from (\ref{CM6}). 
\end{proof}
\begin{proof}[Proof of Theorem 1.1:]  Note the assumption that the function $x\ra E[X_0-x \ | \ X_0>x]$ is decreasing implies that the initial data $v_\ve(\cdot,0)$ for (\ref{BI7}), (\ref{BJ7}) is continuous and increasing. Now $\lim_{t\ra\infty}\langle X_t\rangle/t=1$ follows from Lemma 3.4, the remark following it and Lemma 3.6. The inequality (\ref{L1}) follows from Lemma 3.5 and Lemma 3.6.
\end{proof}

\vspace{.1in}

\section{Representations of Green's functions}
Let $b:\R\times\R\ra\R$ be a continuous function which satisfies the uniform Lipschitz condition
\be \label{A2}
\sup\left\{|\pa b(y,t)/\pa y| \ : \ y,t\in\R\right\} \ \le \ A_\infty
\ee
for some constant $A_\infty$. 
Then the terminal value problem
\be \label{B2}
\frac{\pa u_\ve(y,t)}{\pa t}+b(y,t)\frac{\pa u_\ve(y,t)}{\pa y}+\frac{\ve}{2}\frac{\pa^2 u_\ve(y,t)}{\pa y^2} \ = \ 0, \quad y\in\R, \ t<T,
\ee
\be \label{C2}
u_\ve(y,T) \ = \ u_T(y), \quad y\in\R,
\ee
has a unique solution $u_\ve$ which has the representation
\be \label{D2}
u_\ve(y,t) \ = \ \int_{-\infty}^\infty G_\ve(x,y,t,T) u_T(x) \ dx, \quad y\in\R, \ t<T,
\ee
where $G_\ve$ is the Green's function for the problem.  
The adjoint problem to (\ref{B2}), (\ref{C2}) is the initial value problem
\be \label{E2}
\frac{\pa v_\ve(x,t)}{\pa t}+\frac{\pa}{\pa x}\left[b(x,t)v_\ve(x,t)\right] \ = \ \frac{\ve}{2}\frac{\pa^2 v_\ve(x,t)}{\pa x^2} \ , \quad x\in\R, \ t>0,
\ee
\be \label{F2}
v_\ve(x,0) \ = \ v_0(x), \quad y\in\R.
\ee
The solution to (\ref{E2}), (\ref{F2}) is given by the formula
\be \label{G2}
v_\ve(x,T) \ = \ \int_{-\infty}^\infty G_\ve(x,y,0,T) v_0(y) \ dy, \quad x\in\R, \  T>0.
\ee

For any $t<T$ let $Y_\ve(s), \ s>t,$ be the solution to the initial value problem for the SDE
\be \label{H2}
dY_\ve(s) \ = \ b(Y_\ve(s),s) ds+\sqrt{\ve} \ dB(s), \quad Y_\ve(t)=y,
\ee
where $B(\cdot)$ is Brownian motion.  Then $G_\ve(\cdot,y,t,T)$ is the probability density for the random variable $Y_\ve(T)$.  In the case when the function $b(y,t)$ is linear in $y$ it is easy to see that (\ref{H2}) can be explicitly solved. Thus let $A:\R\ra\R$ be a continuous function and $b:\R\times\R\ra\R$ the function $b(y,t)=A(t)y-1$. The solution to (\ref{H2}) is then given by 
\be \label{I2}
Y_\ve(s) \ = \ \exp\left[\int_t^s A(s') ds'\right] y-\int_t^s \exp\left[\int_{s'}^s A(s'') ds''\right] ds'
+\sqrt{\ve}\int_t^s \exp\left[\int_{s'}^s A(s'') ds''\right] dB(s') \ . 
\ee
Hence the random variable $Y_\ve(T)$ conditioned on $Y_\ve(0)=y$ is Gaussian with mean
$m_{1,A}(T)y-m_{2,A}(T)$ and variance $\ve\sig_A^2(T)$, where $m_{1,A},m_{2,A}$ are given by (\ref{B7}) and $\sig^2_A$ by
\be \label{J2}
\sig_A^2(T)= \int_0^T \exp\left[2\int_{s}^T A(s') ds'\right] ds \ .  
\ee
The Green's function $G_\ve(x,y,0,T)$ is therefore explicitly given by the formula
\be \label{K2}
G_\ve(x,y,0,T)=\frac{1}{\sqrt{2\pi\ve\sig_A^2(T)}}\exp\left[-\frac{\{x+m_{2,A}(T)-m_{1,A}(T)y\}^2}{2\ve\sig_A^2(T)}\right] \ .
\ee

To obtain the formula (\ref{K2}) we have used the fact that the solution to the terminal value problem (\ref{B2}), (\ref{C2}) has a representation as an expectation value $u_\ve(y,t)=E[u_0(Y_\ve(T)) \ |  \ Y(t)=y \ ]$,  where $Y_\ve(\cdot)$ is the solution to the SDE (\ref{H2}).  The initial value problem (\ref{E2}), (\ref{F2}) also has a representation as an expectation value in terms of the  solution to the SDE 
\be \label{AZ2}
dX_\ve(s) \ = \ b(X_\ve(s),s) ds+\sqrt{\ve} \ dB(s), \quad X_\ve(T)=x, \ s<T.
\ee
run {\it backwards} in time. Thus in (\ref{AZ2}) $B(s),  \ s<T,$ is Brownian motion run backwards in time.   The  solution $v_\ve$ of (\ref{E2}), (\ref{F2})  has the representation
\be \label{BA2}
v_\ve(x,T) \ = \  E\left[\exp\left\{-\int^T_0\frac{\pa b(X_\ve(s),s)}{\pa x} \ ds\right\} v_0(X_\ve(0)) \ \ \Bigg| \ X_\ve(T)=x \  \right]  \ .
\ee

Next we consider the terminal value problem (\ref{B2}), (\ref{C2}) in the half space $y>0$ with Dirichlet boundary condition $u_\ve(0,t)=0, \ t<T$. In that case the solution $u_\ve(y,t)$  has the representation
\be \label{L2}
u_\ve(y,t) \ = \ \int_0^\infty G_{\ve,D}(x,y,t,T) u_T(x) \ dx, \quad y>0, \ t<T,
\ee
in terms of the Dirichlet Green's function $G_{\ve,D}$ for the half space.  Similarly the solution to (\ref{E2}), (\ref{F2}) in the half space $x>0$ with Dirichlet condition $v_\ve(0,t)=0, \ t>0,$ has the representation  
\be \label{M2}
v_\ve(x,T) \ = \ \int_{0}^\infty G_{\ve,D}(x,y,0,T) v_0(y) \ dy, \quad x>0, \  T>0.
\ee
The function $G_{\ve,D}(\cdot,y,t,T)$ is the probability density of the random variable $Y_\ve(T)$ for solutions $Y_\ve(s),  \ s>t,$ to (\ref{H2}) which have the property that $\inf_{t\le s\le T} Y_\ve(s)>0$.
No explicit formula for  $G_{\ve,D}(x,y,0,T)$ in the case of  linear $b(y,t)=A(t)y-1$ is known except when $A(\cdot)\equiv 0$.  In that case the method of images yields the formula
 \be \label{N2}
 G_{\ve,D}(x,y,0,T)=\frac{1}{\sqrt{2\pi\ve T}} \left\{ \exp\left[-\frac{(x-y+T)^2}{2\ve T}\right]-  
 \exp\left[-\frac{2x}{\ve}-\frac{(x+y-T)^2}{2\ve T}\right]\right\} \ .
 \ee
It follows from (\ref{K2}), (\ref{N2}) that 
\be \label{O2}
G_{\ve,D}(x,y,0,T)/G_{\ve}(x,y,0,T) \ = \  1- \exp[-2xy/\ve T] \ .
\ee
 We may interpret the formula (\ref{O2}) in terms of conditional probability for solutions $Y_\ve(s), \ s\ge 0,$ of (\ref{H2}) with $b(\cdot,\cdot)\equiv-1$. Thus we have that
\be \label{P2}
P(\inf_{0\le s\le T} Y_\ve(s)>0 \ | \ Y_\ve(0)=y, \ Y_\ve(T)=x) \ = \ 1-\exp[-2xy/\ve T]  \ .
\ee

We wish to generalize (\ref{P2}) to the case of linear $b(y,t)=A(t)y-1$ in a way that is uniform as $\ve\ra 0$. To see what conditions on the function $A(\cdot)$ are needed we consider for $x,y\in\R,t<T,$ the function $q(x,y,t)$ defined by the variational formula
\be \label{Q2}
q(x,y,t,T) = \min_{y(\cdot)} \left\{ \frac 1 2 \int^T_t \left[\frac{dy(s)}{ds} - b(y(s),s)\right]^2ds \ \Big| \ y(t) = y, \ y(T) = x \right\}.
\ee
The Euler-Lagrange equation for the minimizing trajectory $y(\cdot)$ of (\ref{Q2}) is  
\be \label{R2}
\frac d{ds} \left[ \frac{dy(s)}{ds} - b(y(s),s) \right] + \frac {\pa b}{\pa y}(y(s),s) \left[ \frac{dy(s)}{ds} - b(y(s),s) \right] = 0,  \quad t \le s \le T,
\ee
and we need to solve (\ref{R2}) for  the function $y(\cdot)$ satisfying the boundary conditions $y(t)=y, \ y(T)=x$. In the case $b(y,t)=A(t)y-1$ equation (\ref{R2}) becomes
\be \label{S2}
\left[ -\frac{d^2}{ds^2}+A'(s)+A(s)^2\right] y(s) \ = \ A(s), \quad t \le s \le T.
\ee
It is easy to solve (\ref{S2}) with the given boundary conditions explicitly. In fact taking $t=0$ we see from (\ref{R2})  that
\be \label{T2}
 \frac{dy(s)}{ds} - b(y(s),s)  \ = \  C(x,y,T)\exp\left[\int_s^T A(s')ds'\right] \ , \quad 0\le s\le T,
\ee
where the constant  $C(x,y,T)$ is given by the formula
\be \label{U2}
C(x,y,T) \ = \  [x+m_{2,A}(T)-m_{1,A}(T)y]/\sig_A^2(T) \ ,
\ee
with $m_{1,A}(T),m_{2,A}(T)$ as in (\ref{B7}) and $\sig_A^2(T)$ as in (\ref{J2}). It follows from (\ref{K2}), (\ref{Q2}), (\ref{T2}), (\ref{U2}) that the Green's function $G_\ve(x,y,0,T)$ is given by the formula
\be \label{V2}
G_\ve(x,y,0,T)=\frac{1}{\sqrt{2\pi\ve\sig_A^2(T)}}\exp\left[-q(x,y,0,T)/\ve\right] \ .
\ee

The minimizing trajectory $y(\cdot)$  for (\ref{Q2}) has probabilistic significance as well as the function $q(x,y,t,T)$. One can easily see that  for solutions $Y_\ve(s), \ 0\le s\le T,$ of (\ref{H2}) the random variable $Y_\ve(s)$ conditioned on $Y_\ve(0)=y, \ Y_\ve(T)=x,$ is Gaussian with mean and variance given by 
\be \label{W2}
E[ Y_\ve(s) \ | \ Y_\ve(0)=y, \ Y_\ve(T)=x] \ = \ y(s), \quad 0\le s\le T,
\ee
\be \label{X2}
{\rm Var}[ Y_\ve(s) \ | \ Y_\ve(0)=y, \ Y_\ve(T)=x] \ = \  \ve \sig_A^2(0,s)\sig_A^2(s,T)/\sig_A^2(T) \ ,
\ee
where the function $\sig_A^2(s,t)$ is defined by
\be \label{Y2}
\sig_A^2(s,t) \ = \ \int_s^t \exp\left[2\int_{s'}^t A(s'') ds''\right] ds' \quad {\rm for  \ }  s\le t. 
\ee
Let $m_{1,A}(s,t), \ m_{2,A}(s,t)$ be defined by
\be \label{Z2}
m_{1,A}(s,t)=\exp\left[\int_{s}^t A(s') ds'\right], \quad m_{2,A}(s,t)=
\int_s^t \exp\left[\int_{s'}^t A(s'') ds''\right] ds' \quad {\rm for  \ }  s\le t. 
\ee
The minimizing trajectory $y(\cdot)$ for the variational problem (\ref{Q2}) is explicitly given by the formula
\begin{multline} \label{AA2}
\sig_A^2(T)y(s) \ = \  xm_{1,A}(s,T)\sig_A^2(0,s) +ym_{1,A}(0,s)\sig_A^2(s,T) \\
+ \ m_{1,A}(s,T)m_{2,A}(s,T)\sig_A^2(0,s)-m_{2,A}(0,s)\sig_A^2(s,T) \ .
\end{multline}

Now the process  $Y_\ve(s), \ 0\le s\le T,$ conditioned on $Y_\ve(0)=y, \ Y_\ve(T)=x,$ is in fact a Gaussian process with covariance independent of $x,y$, 
\be \label{AB2}
{\rm Covar}[ Y_\ve(s_1),Y_\ve(s_2) \ | \ Y_\ve(0)=y, \ Y_\ve(T)=x] \ = \  \ve \Gamma_A(s_1,s_2) \ , \quad 0\le s_1, s_2\le T,
\ee 
where the symmetric function $\Gamma:[0,T]\times[0,T]\ra\R$   is given by the formula
\be \label{AC2}
\Gamma_A(s_1,s_2) \ = \ \frac{m_{1,A}(s_1,s_2)\sig_A^2(0,s_1)\sig_A^2(s_2,T)}{\sig_A^2(T)} \ , \quad 0\le s_1\le s_2\le T. 
\ee
The function $\Gamma_A$ is the Dirichlet Green's function for the operator  on the LHS of (\ref{S2}).  Thus one has that
\be \label{AD2}
\left[ -\frac{d^2}{ds_1^2}+A'(s_1)+A(s_1)^2\right] \Gamma_A(s_1,s_2) \ = \ \del(s_1-s_2), \quad 0<s_1,s_2 <T,
\ee
and $\Gamma_A(0,s_2)=\Gamma_A(T,s_2)=0$ for all $0<s_2<T$.  

We can obtain a representation of the conditioned process $Y_\ve(\cdot)$ in terms of  the white noise process, which is the derivative $dB(\cdot)$ of Brownian motion, by obtaining a factorization of $\Gamma$ corresponding to the factorization
\be \label{AE2}
-\frac{d^2}{ds^2}+A'(s)+A(s)^2 \ = \  \left[-\frac{d}{ds}-A(s)\right]\left[\frac{d}{ds}-A(s)\right] \  .
\ee
To do this we note that the boundary value problem
\be \label{AF2}
\left[\frac{d}{ds}-A(s)\right]u(s) \ =  \ v(s),  \quad 0<s<T, \quad u(0)=u(T)=0,
\ee
has a solution if and only if  the function $v:[0,T]\ra\R$ satisfies the orthogonality condition
\be \label{AG2}
\int_0^T \frac{v(s)}{m_{1,A}(s)} \ ds \ = \ 0.
\ee
Hence it follows from (\ref{AE2}) that we can solve the boundary value problem 
\be \label{AH2}
\left[-\frac{d^2}{ds^2}+A'(s)+A(s)^2 \right]u(s) \ =  \ f(s), \quad 0<s<T, \quad u(0)=u(T)=0,
\ee
by first finding the solution $v:[0,T]\ra\R$ to 
\be \label{AI2}
\left[-\frac{d}{ds}-A(s)\right]v(s) \ =  \ f(s),  \quad 0<s<T, 
\ee
which satisfies the orthogonality condition (\ref{AG2}).  Then we solve the differential equation in (\ref{AF2}) subject to the condition $u(0)=0$. 

The solution to (\ref{AG2}), (\ref{AI2}) is given by an expression
\be \label{AJ2}
v(s) \ =  \ K^*f(s) \ = \ \int_0^T k(s',s)f(s') \ ds' \ , \quad 0\le s\le T,
\ee
where the kernel $k:[0,T]\times[0,T]\ra\R$ is defined by
\begin{multline} \label{AK2}
k(s',s) \ = \ \frac{m_{1,A}(s,s')\sig^2(s',T)}{\sig_A^2(T)} \quad {\rm if} \ s'>s, \\
k(s',s) \ = \ \frac{\sig_A^2(s',T)}{m_1(s',s)\sig_A^2(T)}-\frac{1}{m_{1,A}(s',s)} \quad {\rm if} \ s'<s.
\end{multline}
If $v:[0,T]\ra\R$ satisfies the condition (\ref{AG2}) then
\be \label{AL2}
u(s) \ =  \ Kv(s) \ = \ \int_0^T k(s,s')v(s') \ ds' \ , \quad 0\le s\le T,
\ee
is the solution to (\ref{AF2}). It follows that the kernel $\Gamma_A$ of (\ref{AC2}) has the factorization $\Gamma_A=KK^*$, and so  the conditioned process $Y_\ve(\cdot)$ has the representation
\be \label{AM2}
Y_\ve(s) \ = \ y(s)+ \sqrt{\ve}\int_0^T k(s,s') \ dB(s')  \ , \quad 0\le s\le T,
\ee
where $y(\cdot)$ is the function (\ref{AA2}). In the  case $A(\cdot)\equiv 0$ equation (\ref{AM2}) yields the familiar representation
\be \label{AN2}
Y_\ve(s) \ =  \ \frac{s}{T} x+\left(1-\frac{s}{T}\right)y+\sqrt{\ve}\left[B(s)-\frac{s}{T}B(T)\right] \ , \quad 0\le s\le T,
\ee
for the Brownian bridge process.

We can obtain an alternative representation of the conditioned process $Y_\ve(\cdot)$ in terms of Brownian motion by considering a stochastic control problem.  Let $Y_\ve(\cd)$ be the solution to the stochastic differential equation
\be \label{AO2}
dY_\ve(s) = \la_\ve(\cd,s)ds + \sqrt{\ve} \; dB(s),
\ee
where $\la_\ve(\cd,s)$ is a non-anticipating function.  We consider the problem of minimizing the cost function given by the formula
\be \label{AP2}
q_\ve(x,y,t,T) = \min_{\la_\ve} E \left[ \frac 1 2 \int^T_t \left[ \la_\ve(\cd,s) - b( Y_\ve(s), s) \right]^2 \; ds \  \Big| \  Y_\ve(t) = y,\; Y_\ve(T) = x\right].
\ee
The minimum in (\ref{AP2}) is to be taken over all non-anticipating $\la_\ve(\cd,s)$, $t \le s < T$, which have the property that the solutions of (\ref{AO2}) with initial condition $Y_\ve(t) = y$ satisfy the terminal condition $Y_\ve(T) = x$ with probability 1.    Formally the optimal controller $\la^*$ for the problem is given by the expression
\be \label{AQ2}
\la_\ve(\cd,s) = \la^*_\ve  ( x, Y_\ve(s), s) =  b( Y_\ve(s), s) - \frac{\pa q_\ve}{\pa y} \; ( x, Y_\ve(s), s).
\ee

Evidently in the classical control case $\ve=0$ the solution to (\ref{AO2}), (\ref{AP2}) is the solution to the variational problem (\ref{Q2}). If $b(y,t)=A(t)y-1$ is a linear function of $y$ then one expects as in the case of LQ problems that the difference between the cost functions for the classical and  stochastic control problems is independent of $y$. Therefore from (\ref{K2}), (\ref{V2}) we expect that
\be \label{AR2}
 \la^*_\ve  ( x, y, t) \ = \ b(y,t)-\frac{\pa q(x,y,t,T)}{\pa y} \ = \ A(t)y-1-\frac{\pa}{\pa y}\frac{\{x+m_{2,A}(t,T)-m_{1,A}(t,T)y\}^2}{2\sig_A^2(t,T)}  \ .
\ee
It is easy to see that if we solve the SDE (\ref{AO2}) with controller given by (\ref{AR2}) and conditioned on $Y_\ve(t)=y$ then $Y_\ve(T)=x$ with probability $1$ and in fact the process $Y_\ve(s), \ t\le s\le T,$ has the same distribution as the process $Y_\ve(s), \ t\le s\le T,$ satisfying the SDE (\ref{H2}) conditioned on $Y_\ve(t)=y, \ Y_\ve(T)=x$. Thus we have obtained the Markovian representation for the conditioned process of (\ref{H2}). Note however that the stochastic control problem with cost function (\ref{AP2}) does not have a solution since the integral in (\ref{AP2}) is logarithmically divergent at $s=T$ for the process (\ref{AO2}) with optimal controller (\ref{AR2}). 

Solving (\ref{AO2}) with drift (\ref{AR2})  and $Y_\ve(0)=y$, we see on taking $t=0$ that (\ref{AM2}) holds with kernel $k:[0,T]\times[0,T]\ra\R$ given by 
\be \label{AS2}
k(s,s') \ = \ \frac{m_{1,A}(s',s)\sig_A^2(s,T)}{\sig_A^2(s',T)} \quad {\rm if \ } s'<s, \qquad k(s,s') \ = \ 0 \quad {\rm if \ } s'>s.
\ee
Observe that the kernel (\ref{AS2}) corresponds to the Cholesky factorization $\Gamma_A=KK^*$ of the kernel $\Gamma_A$ \cite{ciarlet}.  In the case $A(\cdot)\equiv 0$ equation (\ref{AS2}) yields the Markovian representation
\be \label{AT2}
Y_\ve(s) \ =  \ \frac{s}{T} x+\left(1-\frac{s}{T}\right)y+\sqrt{\ve}(T-s)\int_0^s \frac{dB(s')}{T-s'} \ , \quad 0\le s\le T,
\ee
for the Brownian bridge process. 

We can also express the ratio (\ref{O2}) of Green's functions for the linear case $b(y,t)=A(t)y-1$ in terms of the solution to a  PDE. Thus we assume $x>0$ and define
\be \label{AU2}
u(y,t) \ =  \ P(\inf_{t\le s\le T} Y_\ve(s) > 0 \ | \ Y_\ve(t)=y) \ , \quad y>0,t<T,
\ee
where $Y_\ve(\cdot)$ is the solution to the SDE (\ref{AO2}) with drift (\ref{AR2}).  Then $u(y,t)$ is the solution to the PDE
\be \label{AV2}
\frac{\pa u(y,t)}{\pa t} +\la_\ve^*(x,y,t)\frac{\pa u(y,t)}{\pa y}+\frac{\ve}{2}\frac{\pa^2 u(y,t)}{\pa y^2} \ = \ 0, \quad y>0,t<T,
\ee
with boundary and terminal conditions given by
\be \label{AW2}
u(0,t) \ = \ 0 \ {\rm for} \  t<T, \quad \lim_{t\ra T} u(y,t) \ = \ 1 \ {\rm for} \  y>0.
\ee
In the case $A(\cdot)\equiv 0$ the PDE (\ref{AV2}) becomes 
\be \label{AX2}
\frac{\pa u(y,t)}{\pa t} +\left(\frac{x-y}{T-t}\right) \ \frac{\pa u(y,t)}{\pa y}+\frac{\ve}{2}\frac{\pa^2 u(y,t)}{\pa y^2} \ = \ 0, \quad y>0,t<T.
\ee
Evidently the function $u$ defined by
\be \label{AY2}
u(y,t) \ = \ 1- \ \exp\left[-\frac{2xy}{\ve(T-t)}\right] \ , \quad t<T, y>0,
\ee
is the solution to (\ref{AW2}), (\ref{AY2}).  Observe that the RHS of (\ref{AY2}) at $t=0$ is the same as the RHS of (\ref{O2}).

\vspace{.1in}

\section{Estimates on the Dirichlet Green's function}
In this section we shall obtain estimates on the ratio of the Dirichlet to the full space Green's function in the case of linear drift $b(y,t)=A(t)y-1$. In particular we shall prove a limit theorem which generalizes the formula (\ref{O2}):
\begin{proposition}
Assume $b(y,t)=A(t)y-1$ where (\ref{A2}) holds and the function $A(\cdot)$ is non-negative.
Then for $\la,y, T>0$ the ratio of the Dirichlet to full space Green's function satisfies the limit
\be \label{A3}
\lim_{\ve\ra 0} \frac{G_{\ve,D}(\la\ve,y,0,T)}{G_{\ve}(\la\ve,y,0,T)} \ = \  1- \exp\left[-2\la\left\{1-\frac{m_{2,A}(T)}{\sig_A^2(T)}+\frac{m_{1,A}(T)y}{\sig_A^2(T)}\right\} \ \right] \ ,
\ee
where $m_{1,A}(T),m_{2,A}(T)$ are given by (\ref{B7}) and $\sig_A^2(T)$  by (\ref{J2}). 
\end{proposition}
Note that since we are assuming $A(\cdot)$ is non-negative in the statement of the proposition, it follows from (\ref{J2}) that $m_{2,A}(T)/\sig_A^2(T)\le 1$. Hence the  RHS of (\ref{A3}) always lies between $0$ and $1$. 
We can see why (\ref{A3}) holds from the representation (\ref{AK2}), (\ref{AM2}) for the conditioned process $Y_\ve(s), \ 0\le s\le T$.  Thus we have that
\be \label{B3}
Y_\ve(s) \ = \ y(s)+\sqrt{\ve}\left[ \ \frac{m_{1,A}(s)\sig_A^2(s,T)}{\sig_A^2(T)}\int_0^T \  \frac{dB(s')}{m_{1,A}(s')}-m_{1,A}(s)\int_s^T  \ \frac{dB(s')}{m_{1,A}(s')} \ \right] \ .
\ee
Since $\sig_A^2(s,T)=O(T-s)$ the conditioned process $Y_\ve(s)$ close to $s=T$ is approximately the same as
\be \label{C3}
Y_\ve(s) \ = \ \la\ve -y'(T)(T-s)-\sqrt{\ve}\int_s^T \ dB(s') \ .
\ee
Observe now from (\ref{AA2})  that
\be \label{D3}
-y'(T) \ = \ O(\ve)+1-\frac{m_{2,A}(T)}{\sig_A^2(T)}+\frac{m_{1,A}(T)y}{\sig_A^2(T)} \ .
\ee
Hence for $s$ close to $T$ the process $Y_\ve(s), \ s<T,$ is approximately Brownian motion with a constant drift. Thus let $Z_\ve(t), \ t>0,$ be the solution to the initial value problem  for the SDE
\be \label{E3}
dZ_\ve(t) \ = \  \mu dt+\sqrt{\ve} \  dB(t), \quad Z_\ve(0)=\la\ve \ ,
\ee
where we assume the drift $\mu$ is positive. Then from (\ref{C3}), (\ref{D3}) we see that  $Y_\ve(T-t)\simeq Z_\ve(t)$ if $\mu$ is given by the formula
\be \label{F3}
\mu \ = \ 1-\frac{m_{2,A}(T)}{\sig_A^2(T)}+\frac{m_{1,A}(T)y}{\sig_A^2(T)} \ .
\ee
Observe now that $P(\inf_{t>0} Z_\ve(t)<0)=e^{-2\la\mu}$, whence  the RHS of (\ref{A3}) is simply 
$P(\inf_{t>0} Z_\ve(t)>0)$ when $\mu$ is given by (\ref{F3}).  Since the time for which $Z_\ve(t)$ is likely to become negative is $t\simeq O(\ve)$ the approximations above are justified and so we obtain (\ref{A3}). 
\begin{proof}[Proof of Proposition 5.1]  Let $Y_\ve(s), \ 0\le s\le T,$ be given by (\ref{B3}) where $y(T)=\la\ve$. Then we have that for $0<a\ve\le T$, 
\be \label{G3}
P\left(\inf_{0\le s\le T} Y_\ve(s)>0\right) \ \le \ P\left(\inf_{0\le t\le a\ve} Y_\ve(T-t)>0\right) \ = \ 
P\left(\inf_{0<t<a\ve} [Z_\ve(t)+\tilde{Z}_\ve(t)]>0\right) \ ,
\ee
where $Z_\ve(\cdot)$ is the solution to (\ref{E3}) with $\mu$ given by (\ref{F3}) and $\tilde{Z}_\ve(\cdot)$ is given by the formula 
\begin{multline} \label{H3}
\tilde{Z}_\ve(t) \ = \  y(T-t)-y(T)+y'(T)t+\\
\sqrt{\ve}\left[ \ \frac{m_{1,A}(T-t)\sig_A^2(T-t,T)}{\sig_A^2(T)}\int_0^T \  \frac{dB(s')}{m_{1,A}(s')} 
+\int_{T-t}^T\left [1-\frac{m_{1,A}(T-t)}{m_{1,A}(s')}\right]  \ dB(s') \ \right] \ .
\end{multline}
We use the inequality
\be \label{I3}
P\left(\inf_{0<t<a\ve} [Z_\ve(t)+\tilde{Z}_\ve(t)]>0\right)  \ \le \ P\left(\inf_{0<t<a\ve} Z_\ve(t)>-b\la\ve\right) +P\left(\sup_{0<t<a\ve} \tilde{Z}_\ve(t)>b\la\ve\right)  \ ,
\ee
which holds for any $a,b>0$ satisfying $a\ve\le T$.

To estimate the first term on the RHS of (\ref{I3}) we observe by the method of images  that 
\begin{multline} \label{J3}
P\left(\inf_{0<t<a\ve} Z_\ve(t)<-b\la\ve\right) \ = \\  e^{-2\mu(1+b)\la}\frac{1}{\sqrt{2\pi}}\int_{[(1+b)\la-\mu a]/\sqrt{a}}^\infty e^{-z^2/2} \ dz  \ 
+ \ \frac{1}{\sqrt{2\pi}}\int^{-[(1+b)\la+\mu a]/\sqrt{a}}_{-\infty} e^{-z^2/2} \ dz \ .
\end{multline}
To estimate the second term we write $\tilde{Z}_\ve(t)$ in (\ref{H3}) as a sum of three quantities. The first of these is bounded as
\be \label{K3}
\sup_{0\le t\le a\ve} |y(T-t)-y(T)+y'(T)t| \ \le \  C[\la\ve+y+1]a^2\ve^2, \quad 0<a\ve \le T,
\ee
for a constant $C$ depending only on $A_\infty,T$.  The second is bounded as
\be \label{L3}
\sup_{0\le t\le a\ve}\left|\sqrt{\ve} \ \frac{m_{1,A}(T-t)\sig_A^2(T-t,T)}{\sig_A^2(T)}\int_0^T \  \frac{dB(s')}{m_{1,A}(s')} \right| \ \le \ Ca\ve^{3/2}\left|\int_0^T \  \frac{dB(s')}{m_{1,A}(s')} \right|  \ ,
\ee
where $C$ depends only on $A_\infty,T$. Finally the third quantity is bounded as
\begin{multline} \label{M3}
\sup_{0\le t\le a\ve}\left| \ \int_{T-t}^T\left [1-\frac{m_{1,A}(T-t)}{m_{1,A}(s')}\right]  \ dB(s')  \ \right| \ \le \ 
\sup_{0\le t\le a\ve}\left| \ \int_{T-t}^T\left [1-\frac{m_{1,A}(T)}{m_{1,A}(s')}\right]  \ dB(s') \ \right| \\
+ Ca\ve \sup_{0\le t\le a\ve}\left|\int_{T-t}^T \  \frac{dB(s')}{m_{1,A}(s')} \right|  \ , \quad {\rm where \ } C \ {\rm depends \ only \ on \ } A_\infty,T.
\end{multline}

We can estimate probabilities for the terms on the RHS of (\ref{L3}), (\ref{M3}) by using Martingale properties. Thus if $g:(-\infty,T)\ra\R$ is a continuous function we define $X(t), \ t\ge 0,$ by
\be \label{N3}
X(t) \ = \ \int_{T-t}^T   g(s) \  dB(s) \ .
\ee
Then for $\theta\in\R$
\be \label{O3}
X_\theta(t) \ = \ \exp\left[ \theta X(t)-\frac{\theta^2}{2}\int_{T-t}^T ds \  g(s)^2 \ \right] \quad {\rm is \ a \ Martingale \ and \ }  E[X_\theta(t)]=1.
\ee
Using the inequality
\be \label{P3}
P\left( \ |X(0)|>M \ \right) \ \le \ 2\exp\left[-\theta M+\frac{\theta^2}{2}\int_0^T ds \  g(s)^2 \ \right] \quad {\rm for \ } M,\theta>0,
\ee
and optimizing the RHS of (\ref{P3}) with respect to $\theta>0$ we conclude that
\be \label{R3}
P\left( \ a\ve^{3/2}\left|\int_0^T \  \frac{dB(s')}{m_{1,A}(s')} \right| > b\la\ve/4 \ \right) \ \le  \ 
2\exp\left[-Cb^2\la^2/a^2\ve\right] \ ,
\ee
where the constant $C>0$ depends only on $A_\infty,T$.  We use Doob's inequality to estimate probabilities for the terms on the RHS of (\ref{M3}). Thus we have for $\theta>0$ that
\begin{multline} \label{S3}
P\left( \ \sup_{0\le t\le t_0} X(t)>M \ \right) \ \le \ P\left( \ \sup_{0\le t\le t_0} X_\theta(t)>\exp\left[\theta M- \frac{\theta^2}{2}\int_{T-t_0}^T ds \  g(s)^2  \ \right] \ \right)  \\
\le \ \exp\left[-\theta M+ \frac{\theta^2}{2}\int_{T-t_0}^T ds \  g(s)^2  \ \right] \ . 
\end{multline}
Optimizing the term on the RHS of (\ref{S3}) with respect to $\theta>0$ we conclude that
\be \label{T3}
P\left( \ \sup_{0\le t\le t_0} |X(t)|>M \ \right) \ \le \  2\exp\left[- M^2\bigg/2\int_{T-t_0}^T ds \  g(s)^2  \ \right]  \ .
\ee
Hence we have from (\ref{T3}) for the first term on the RHS of (\ref{M3}) that
\be \label{U3}
P\left( \  \sup_{0\le t\le a\ve}\left| \ \int_{T-t}^T\left [1-\frac{m_{1,A}(T)}{m_{1,A}(s')}\right]  \ dB(s') \ \right|  \ > \ b\la\ve/4 \ \right) \ \le \ 2\exp\left[-Cb^2\la^2/a^3\ve\right] \ ,
\ee
where the constant $C>0$ depends only on $A_\infty,T$.  Similarly we have that if $C_1$ depends only on $A_\infty,T$ then
\be \label{V3}
P\left( \   C_1a\ve \sup_{0\le t\le a\ve}\left|\int_{T-t}^T \  \frac{dB(s')}{m_{1,A}(s')} \right|  \ > \ b\la\ve/4 \ \right) \ \le \ 2\exp\left[-C_2b^2\la^2/a^3\ve\right] \ ,
\ee
where the constant $C_2>0$ also depends only on $A_\infty,T$. 

We choose now $a=\ve^{-\al}, b=\ve^\beta$ for some $\al,\beta>0$. Since $\mu>0$ it follows from  
(\ref{J3})  that the first term on the RHS of (\ref{I3}) converges to $1-e^{-2\la\mu}$ as $\ve\ra 0$. We also see from the estimates of the previous paragraph that the second term on the RHS of (\ref{I3}) converges to $0$ as $\ve\ra 0$ provided $3\al+2\beta<1$. We have therefore shown that $\limsup_{\ve\ra 0}P\left(\inf_{0\le s\le T} Y_\ve(s)>0\right)$ is bounded above by the RHS of (\ref{A3}). 

To obtain the corresponding lower bound we use the inequality
\be \label{W3}
P\left(\inf_{0\le s\le T} Y_\ve(s)>0\right) \ \ge \ P\left(\inf_{T-a\ve\le s\le T} Y_\ve(s)>0\right)-P\left(\inf_{0\le s\le T-a\ve} Y_\ve(s)<0\right) \ .
\ee 
Next we use the inequality similar to (\ref{I3}) that
\be \label{X3}
P\left(\inf_{T-a\ve\le s \le T} Y_\ve(s)>0\right) \ \ge \ 
P\left(\inf_{0<t<a\ve} Z_\ve(t)>b\la\ve\right) -P\left(\inf_{0<t<a\ve} \tilde{Z}_\ve(t)<-b\la\ve\right)  \ .
\ee
Arguing as previously we see from (\ref{X3}) on choosing $a=\ve^{-\al},b=\ve^\beta$ with $3\al+2\beta<1$ that $ \liminf_{\ve\ra 0}P\left(\inf_{T-\ve^{1-\al}\le s\le T} Y_\ve(s)>0\right)$ is bounded below by the RHS of (\ref{A3}).  Next we need to obtain a bound on the second term on the RHS of
(\ref{W3}) when $a=\ve^{-\al}$ which vanishes as $\ve\ra 0$. Since $A(\cdot)$ is non-negative there is a positive constant $C$ depending only on $A_\infty,T$ such that  the function $y(\cdot)$ of (\ref{AA2}) satisfies an inequality  $y(s)\ge C(T-s)y$ for $0\le s\le T$. 
Hence  there is a positive constant $c$ depending only on $A_\infty,T$ such that
\begin{multline} \label{Y3}
P\left(\inf_{0\le s\le T-\ve^{1-\al}} Y_\ve(s)<0\right)    \ \le \ 
 P\left( \ \left|\int_0^T \  \frac{dB(s')}{m_{1,A}(s')} \right| > \frac{cy}{\sqrt{\ve}} \ \right) \\
 + P\left( \ \sup_{\ve^{1-\al}\le t\le T}\left|\frac{1}{t}\int_{T-t}^T \  \frac{dB(s')}{m_{1,A}(s')} \right| > \frac{cy}{\sqrt{\ve}} \ \right)  \ .
\end{multline}
We can bound the first term on the RHS of (\ref{Y3}) similarly to  (\ref{R3}).  We bound the second term by using the inequality
\be \label{Z3}
P\left(\sup_{\ve^{1-\al}\le t\le T} |X(t)|>cy/\sqrt{\ve}\right)  \ \le \  \sum_{k\ge1} P\left(\sup_{k\ve^{1-\al}\le t\le (k+1)\ve^{1-\al}} |X(t)|>cy/\sqrt{\ve}\right) \ . 
\ee
From (\ref{T3}) we see that for $k\ge 1$,
\be \label{AA3}
 P\left(\sup_{k\ve^{1-\al}\le t\le (k+1)\ve^{1-\al}} \left|\frac{1}{t}\int_{T-t}^T \  \frac{dB(s')}{m_{1,A}(s')} \right| \ >cy/\sqrt{\ve}\right) \ \le \  \exp\left[-\frac{c_1ky^2}{\ve^\al} \ \right]  \ ,
\ee
where $c_1>0$ depends only on $A_\infty,T$.  We conclude that the second term on the RHS of (\ref{W3})  converges when $a=\ve^{-\al}$ with $\al>0$ to zero as $\ve\ra 0$.  Hence $\liminf_{\ve\ra 0}P\left(\inf_{0\le s\le T} Y_\ve(s)>0\right)$ is bounded below by the RHS of (\ref{A3}).
\end{proof}
Next we wish to obtain estimates on the LHS of (\ref{A3}) which are uniform as $\la\ra 0$.
\begin{lem}
Assume the function $A(\cdot)$ is non-negative and that $0<\la\le 1, \ 0<\ve\le T, \ y>0$. Let $\Ga:\R^+\times\R^+\ra\R^+$ be the function $\Ga(a,b)=1$ if $b>a^{-1/4}$ and otherwise $\Ga(a,b)=a^{1/8}$. Then there is a constant $C$ depending only on $A_\infty T$ such that 
\begin{multline} \label{AB3}
 \frac{G_{\ve,D}(\la\ve,y,0,T)}{G_{\ve}(\la\ve,y,0,T)} \ \le \\
   1- \exp\left[-2\la\left\{1-\frac{m_{2,A}(T)}{\sig_A^2(T)}+\frac{m_{1,A}(T)y}{\sig_A^2(T)}\right\} \ \right] +C\la\Ga\left(\frac{\ve}{T},\frac{y}{T}\right) \left[1+\frac{y}{T}\right] \ .
\end{multline}
\end{lem}
\begin{proof}
We make the change of variable $s\leftrightarrow t$ in which
 \be \label{AC3}
 \frac{ds}{dt} \ =  \  -\left[ \ \frac{m_{1,A}(s)}{m_{1,A}(T)} \ \right]^2 \ , \quad s(0) \ = \  T.
 \ee
 Hence $s\simeq T-t$ if $t$ is small and  
 \be \label{AD3}
m_{1,A}(T) \int_s^T  \ \frac{dB(s')}{m_{1,A}(s')}  \ = \  \int_0^t d\tilde{B}(t') \quad {\rm where \ } \tilde{B}(\cdot) \ {\rm is \  a \ Brownian \  motion.}
 \ee
Letting $s(\tilde{T})=0$, we see from (\ref{B3}), (\ref{AD3})  that $Y_\ve(s)=\tilde{Y}_\ve(t)$ where
\be \label{AE3}
\tilde{Y}_\ve(t) \ = \  \tilde{y}(t)+ \sqrt{\ve}\left[ \ \frac{m_{1,A}(s)\sig^2(s,T)}{m_{1,A}(T)\sig^2(T)}\int_0^{\tilde{T}} d\tilde{B}(t') -\frac{m_{1,A}(s)}{m_{1,A}(T)}\int_0^t d\tilde{B}(t')  \ \right] \ ,
\ee
and $\tilde{y}(t)=y(s)$, where $y(\cdot)$ is the function (\ref{AA2}). 
We consider any  $a$ for which  $0<a\ve\le\tilde{T}$ and observe as in (\ref{G3}) that if $M>0$ then
\begin{multline} \label{AF3}
P\left(\inf_{0\le s\le T} Y_\ve(s)>0\right) \ \le \ P\left(\inf_{0\le t\le a\ve} \tilde{Y}_\ve(t)>0\right) \\ 
\le \  
 P\left(\inf_{0\le t\le a\ve} \tilde{Y}_\ve(t)>0; \  \sup_{0\le t\le a\ve}\left| \int_0^t d\tilde{B}(t')  \right|\le M  \right) \\
  +   P\left(\inf_{0\le t\le a\ve} \tilde{Y}_\ve(t)>0; \  \sup_{0\le t\le a\ve}\left| \int_0^t d\tilde{B}(t')  \right|> M  \right)   \ .
\end{multline}
The first term on the RHS of (\ref{AF3}) is bounded above by $  P\left(\inf_{0\le t\le a\ve} \tilde{Y}_{0,\ve}(t)>0\right)$
where $\tilde{Y}_{0,\ve}(t)$ is given from (\ref{AE3})   by the formula
\be \label{AG3}
\tilde{Y}_{0,\ve}(t) \ = \ \tilde{y}(t)+\frac{C\sqrt{\ve}Mt}{T}+\sqrt{\ve} \ \frac{m_{1,A}(s)\sig_A^2(s,T)}{m_{1,A}(T)\sig_A^2(T)} \int_{a\ve}^{\tilde{T}} d\tilde{B}(t') -\sqrt{\ve}\frac{m_{1,A}(s)}{m_{1,A}(T)}  \int_0^t d\tilde{B}(t')   \ ,
\ee
with $C$ in (\ref{AG3}) depending only on $AT$.  To estimate the second term on the RHS of (\ref{AF3}) we introduce the stopping time $\tau$ defined by
\be \label{AH3}
\tau = \inf\left\{ \ t<\tilde{T}: \  \left|  \int_0^t d\tilde{B}(t')\right| > M    \right\} \ .
\ee
Hence the second term is bounded above by $ P\left(\inf_{ 0\le t\le\tau} \tilde{Y}_\ve(t)>0; \  \tau<a\ve \ \right)$. Observe now that for any $M_1>0$, 
\begin{multline} \label{AI3}
 P\left(\inf_{ 0\le t\le\tau} \tilde{Y}_\ve(t)>0; \  \tau<a\ve \ \right) \ = \\
  \sum_{n=1}^\infty  P\left(\inf_{ 0\le t\le\tau} \tilde{Y}_\ve(t)>0; \  \tau<a\ve, \ (n-1)M_1\le\sup_{\tau\le t\le \tau+\tilde{T}}\left|\int_\tau^t d\tilde{B}(t')\right| < nM_1    \ \right) \ \le \\
   \sum_{n=1}^\infty  P\left(\inf_{0\le t\le \tau} \tilde{Y}_{n,\ve}(t)>0; \  \tau<a\ve  \ \right)P\left( \ (n-1)M_1\le \left|\sup_{\tau\le t\le \tau+\tilde{T}}\int_\tau^t d\tilde{B}(t')\ \right| < nM_1    \ \right) \\
  \ = \    \sum_{n=1}^\infty  P\left(\inf_{0\le t\le \tau} \tilde{Y}_{n,\ve}(t)>0; \  \tau<a\ve  \ \right)P\left( \ (n-1)M_1\le \left|\sup_{0\le t\le \tilde{T}}\int_0^t d\tilde{B}(t')\ \right| < nM_1    \ \right)
    \ ,
\end{multline}
where $\tilde{Y}_{n,\ve}$ is given by the formula
\be \label{AJ3}
\tilde{Y}_{n,\ve}(t) \ = \ \tilde{y}(t)+\frac{C\sqrt{\ve}(M+nM_1)t}{T}- \sqrt{\ve}\frac{m_{1,A}(s)}{m_{1,A}(T)}  \int_0^t d\tilde{B}(t')  \ ,
\ee
and the constant $C$ depends only on $A_\infty T$. Note that in (\ref{AI3}) we are using the fact that the variables
\be \label{AK3}
\tau \   \ {\rm and \ } \{\tilde{B}(t): \   0<t\le \tau\} \quad {\rm are \ independent \ of \ the \ variable \ }    \left|\sup_{\tau\le t\le \tau+\tilde{T}}\int_\tau^t d\tilde{B}(t')\ \right|    \ .
\ee

To estimate $  P\left(\inf_{0\le t\le a\ve} \tilde{Y}_{0,\ve}(t)>0\right)$
 we compare $\tilde{Y}_{0,\ve}(\cdot)$ to Brownian motion with constant drift as in (\ref{E3}).   It follows  from  (\ref{AG3})   that
 \be \label{AL3}
 P\left(\inf_{0\le t\le a\ve} \tilde{Y}_{0,\ve}(t)>0\right)  \ \le \  E\left[ \ P\left(\inf_{0\le t\le  a\ve} Z_\ve(t)>0 \   \Big| \ \mu=\mu_{\rm rand}, \ Z_\ve(0)=\la \ve[1+Ca\ve/T]  \ \right) \ \right] \ ,
 \ee
 where $\mu_{\rm rand}$ is the random variable
 \be \label{AM3}
 \mu_{\rm rand} \ = \  1-\frac{m_{2,A}(T)}{\sig_A^2(T)}+\frac{m_{1,A}(T)y}{\sig_A^2(T)}+\frac{Ca\ve}{T}\left[1+\frac{y}{T}\right] + \frac{C\la\ve}{T}+ \frac{C\sqrt{\ve}}{T}\left[M+ \left|\int_{a\ve}^{\tilde{T}} d\tilde{B}(t') \right|\right] \ ,
 \ee
 and $C>0$ is a constant depending only on $A_\infty T$. To bound the RHS of (\ref{AL3}) we use an identity similar to  (\ref{J3}), 
 \begin{multline} \label{AN3}
P\left(\inf_{0<t<a'\ve} Z_\ve(t)>0 \ \big| \ Z_\ve(0)=\la'\ve \ \right) \ = \\ \left
\{1- e^{-2\mu\la'}\right\}\frac{1}{\sqrt{2\pi}}\int_{[\la'-\mu a']/\sqrt{a'}}^\infty e^{-z^2/2} \ dz  \ 
+ \ \frac{1}{\sqrt{2\pi}}\int^{[\la'-\mu a']/\sqrt{a'}}_{[-\la'-\mu a']/\sqrt{a'}} e^{-z^2/2} \ dz \ .
\end{multline}
From (\ref{AN3}) we obtain the upper bound 
\be \label{AO3}
P\left(\inf_{0<t<a'\ve} Z_\ve(t)>0 \ \big| \ Z_\ve(0)=\la'\ve \ \right) \ \le \ 
1- e^{-2\mu\la'} + \frac{2\la'}{\sqrt{2\pi a'}}  \ .
\ee
Using (\ref{AO3}) we estimate the RHS of (\ref{AL3})  when $a=\min\left[(T/\ve)^\al, \ \tilde{T}/\ve\right]$  for some $\al$ satisfying $0<\al<1$.  In that case $\la'=\la[1+Ca\ve/T]\le \la[1+C]$ for some constant $C$ depending only on $A_\infty T$.    Taking $M=C_1\sqrt{T}$ in (\ref{AM3}) where $C_1$ depends only on $A_\infty T$ we conclude from (\ref{AL3}), (\ref{AO3}) that for $0<\la\le 1, \ 0<\ve\le T, $
\begin{multline} \label{AP3}
P\left(\inf_{0\le t\le a\ve} \tilde{Y}_{0,\ve}(t)>0\right)  \ \le \\ 
 1- \exp\left[-2\la\left\{1-\frac{m_{2,A}(T)}{\sig_A^2(T)}+\frac{m_{1,A}(T)y}{\sig_A^2(T)}\right\} \ \right]+
 C_2\la \left[ \left(\frac{\ve}{T}\right)^{1-\al}\left\{1+\frac{y}{T}\right\}+\left(\frac{\ve}{T}\right)^{\al/2}\right]  \ ,
\end{multline}
where $C_2$ in (\ref{AP3}) depends only on $A_\infty T$. 

Next we estimate the probabilities on the RHS of (\ref{AI3}). Evidently we have from (\ref{AJ3}) that
\be \label{AQ3}
\tilde{Y}_{n,\ve}(\tau) \ = \ \tilde{y}(\tau)+\frac{C\sqrt{\ve}(M+nM_1)\tau}{T}\pm M\sqrt{\ve}\frac{m_{1,A}(s(\tau))}{m_{1,A}(T)}  \ .
\ee
We choose $M_1=\sqrt{T}$ in (\ref{AQ3}) and $M=C_1\sqrt{T}$ for a constant $C_1$ depending only on $A_\infty T$ so that $M\sqrt{\ve}/m_1(T)>2\ve$.  Since $Y_{n,\ve}(\tau)>0$, it follows that if (\ref{AQ3}) holds with the $-$ sign then there is a constant $c>0$ depending only on $A_\infty T$ such that
\be \label{AR3}
\tau>  \ \tau_n \ = \ cT\sqrt{\frac{\ve}{T}}\left[1+\frac{y}{T}+n \sqrt{\frac{\ve}{T}} \  \right]^{-1} \ .
\ee
Observe now from (\ref{AR3}) that if $\al<1/2$ then $\tau_n>a\ve$ provided
\be \label{BB3}
1+\frac{y}{T} +n\sqrt{\frac{\ve}{T}} \ \le \  2c_1\left(\frac{T}{\ve}\right)^{1/2-\al} \quad {\rm for \ } c_1>0 \ {\rm depending  \ only \ on \ } A_\infty T. 
\ee
Since $\tau<a\ve$ it follows that (\ref{AQ3}) can hold with the minus sign only if 
\be \label{BC3}
1+\frac{y}{T}  \ \ge \  c_1\left(\frac{T}{\ve}\right)^{1/2-\al}  \quad {\rm or \ } n \ \ge \ c_1\left(\frac{T}{\ve}\right)^{1-\al}  \ .
\ee
In the case when $\tau_n<a\ve$ we see from (\ref{AJ3}) that there is a constant $C$ depending only on $A_\infty T$ and
\be \label{AS3}
 P\left(\inf_{ 0\le t\le\tau_n} \tilde{Y}_{n,\ve}(t)>0  \ \right) \ \le \ 
 P\left(\inf_{0\le t\le  \tau_n} Z_\ve(t)>0 \   \Big| \ \mu=\mu_n, \ Z_\ve(0)=\la \ve[1+Ca\ve/T]  \ \right) \ , 
\ee
where $Z_\ve(\cdot)$ is the solution to the SDE (\ref{E3}). The drift  $\mu_n$ is given by the formula
\be \label{AT3}
\mu_n \ = \  C\left[ 1+\frac{y}{T} +n\sqrt{\frac{\ve}{T}} \  \right] \quad {\rm where  \ }C \ {\rm depends \ only \ on \ } A_\infty T.
\ee
It follows then from (\ref{AO3}), (\ref{AS3}), (\ref{AT3}) that
\be \label{AU3}
 P\left(\inf_{ 0\le t\le\tau_n} \tilde{Y}_{n,\ve}(t)>0  \ \right) \ \le \  C_1\la\left[ 1+\frac{y}{T} +n\sqrt{\frac{\ve}{T}} + \left(\frac{\ve}{\tau_n}\right)^{1/2}\  \right] \ \le \ C_2\la\left[ 1+\frac{y}{T} +n\sqrt{\frac{\ve}{T}} \  \right] 
\ee
for some constants $C_1,C_2$ depending only on $A_\infty T$.  We conclude from (\ref{AU3}) that
\begin{multline} \label{AV3}
 \sum_{n\ge c(T/\ve)^{1-\al}}  P\left(\inf_{0\le t\le \tau_n} \tilde{Y}_{n,\ve}(t)>0   \ \right)P\left( \ (n-1)M_1\le \left|\sup_{0\le t\le \tilde{T}}\int_0^t d\tilde{B}(t')\ \right| < nM_1    \ \right) \\
 \le \  C\la\sum_{n\ge c(T/\ve)^{1-\al}} \left[ 1+\frac{y}{T} +n\sqrt{\frac{\ve}{T}} \  \right]  e^{-n^2/2}
  \ \le \ C_1\la\left(1+\frac{y}{T}\right) \exp\left[ \ -c_1\left(\frac{T}{\ve}\right)^{2(1-\al)} \right] \ ,
 \end{multline}
 where the constants $C_1,c_1$ depend only on $A_\infty T$.  
 
 We consider next the situation where (\ref{AQ3}) holds with the plus sign.  One sees that
 \begin{multline} \label{AW3}
  P\left(\inf_{ 0\le t\le\tau} \tilde{Y}_{n,\ve}(t)>0; \ \tau<a\ve,  \ \int_0^\tau d\tilde{B}(t') \ dt'=-M \ \right) \\ \le \  
P\left(\inf_{0\le t\le  \tau} Z_\ve(t)>0, \ \tau<a\ve, \  \ Z_\ve(\tau)\ge M\sqrt{\ve} \   \Big| \ \mu=\mu_n, \ Z_\ve(0)=\la \ve[1+Ca\ve/T]  \ \right)   \ ,
 \end{multline}
 where $ \mu_n$ is given by (\ref{AT3}). Observe that the RHS of (\ref{AW3}) is bounded by the probability that the diffusion $Z_\ve(\cdot)$ started at $\la\ve[1+O(\ve^{1-\al})]$ exits the interval $[0,C_1T(\ve/T)^{1/2}]$ through the rightmost boundary in time less than $T(\ve/T)^{1-\al}$. This probability is bounded by $K(\ve/T, n,y/T)\la$ for some function $K$ which has the property that $\lim_{\ve\ra 0} K(\ve/T, n,y/T)=0$ provided $\al<1/2$. To find an expression for $K$ we first choose $C_1$ depending only on $A_\infty T$ large enough so that $Z_\ve(0)<C_1T(\ve/T)^{1/2}/2$ for any $\ve $ satisfying $0<\ve\le T$. It is easy to see that for $0<\la'<\La'$, 
 \be \label{AX3}
 P\left( 0<Z_\ve(t)<\La'\ve, \ t<\tau, \ Z_\ve(\tau)=\La'\ve \ \big| \ Z_\ve(0)=\la'\ve \ \right) \ = \ 
 \frac{1-e^{-2\mu\la'}}{1-e^{-2\mu\La'}} \ .
 \ee
 We apply (\ref{AX3}) with $\La'=C_1(T/\ve)^{1/2}, \ \mu=\mu_n$ and $\la'=\la[1+Ca\ve/T]$, whence $\mu\La'\ge c$ for some positive constant $c$ depending only on $A_\infty T$.  
We conclude from (\ref{AX3}) that the function $K$ satisfies the inequality
\be \label{AY3}
K(\ve/T, n,y/T) \ \le  \  C\left[ 1+\frac{y}{T} +n\sqrt{\frac{\ve}{T}} \  \right]
\ee
for some constant $C$ depending only on $A_\infty T$. 

To show that $\lim_{\ve\ra 0}K(\ve/T, n,y/T)=0$ we assume $\la'<\La'/2$ and use the inequality
\begin{multline} \label{AZ3}
 P\left( \inf_{0<t<a\ve} Z_\ve(t)>0, \ \sup_{0<t<a\ve}  Z_\ve(t)\ge\La'\ve \ \big| \ Z_\ve(0)=\la'\ve \ \right) \ \le \\
   P\left( 0<Z_\ve(t)<\La'\ve/2, \ t<\tau, \ Z_\ve(\tau)=\La'\ve/2 \ \big| \ Z_\ve(0)=\la'\ve \ \right) \ 
   P\left(  \sup_{0<t<a\ve}  Z_\ve(t)\ge\La'\ve \ \big| \ Z_\ve(0)=\La'\ve/2 \ \right) \ . 
\end{multline}
The second probability on the RHS of (\ref{AZ3}) can be bounded as
\be \label{BA3}
 P\left(  \sup_{0<t<a\ve}  Z_\ve(t)\ge\La'\ve \ \big| \ Z_\ve(0)=\La'\ve/2 \ \right) \ \le \  C\exp\left[-\frac{\La'^2}{32a}\right] \ 
\ee
for some universal constant $C$ provided $\mu a<\La'/4$. Observe now from (\ref{AT3}) that the condition $\mu_n a<\La'/4$  is implied by (\ref{BB3}). 
We conclude from (\ref{AY3}), (\ref{BA3}) that if (\ref{BB3}) holds then
\be \label{BD3}
K(\ve/T, n,y/T) \ \le  \  C_2\left[ 1+\frac{y}{T} +n\sqrt{\frac{\ve}{T}} \  \right]\exp\left[-c_2\left(\frac{T}{\ve}\right)^{1-\al} \right]
\ee
for some positive constants $C_2,c_2$ depending only on $A_\infty T$. If (\ref{BB3}) does  not hold we can argue as before using (\ref{BC3}), (\ref{AY3}) to obtain an inequality similar to (\ref{AV3}).  The inequality (\ref{AB3}) follows now from (\ref{AP3}),  (\ref{AV3}), (\ref{AY3})  on choosing $\al=1/4$.  
\end{proof}
\begin{lem}
Assume the function $A(\cdot)$ is non-negative and that $0<\la\le 1, \ 0<\ve\le T, \ y>0$. Then there are positive constants $C,c$ depending only on $A_\infty T$ such that   if $\ga=c(T/\ve)^{1/8}(y/T)\ge 5$ then
\be \label{BF3}
 \frac{G_{\ve,D}(\la\ve,y,0,T)}{G_{\ve}(\la\ve,y,0,T)} \ \ge \ 
  [1+e^{-\ga^2/4}]^{-2} \left(1- \exp\left[-\frac{2\la}{1+C(\ve/T)^{1/8}}\left\{1-\frac{m_{2,A}(T)}{\sig_A^2(T)}+\frac{m_{1,A}(T)y}{\sig_A^2(T)}\right\} \ \right]  \ \right) \ .
  \ee
\end{lem}
\begin{proof}
We choose $a=\min\left[(T/\ve)^\al, \tilde{T}/\ve\right]$ with $0<\al<1$ as in Lemma 5.1 and observe from (\ref{AA2}), (\ref{AC3}) that there is a constant $c>0$ depending only on $A_\infty T$ such that  $\tilde{y}(t)\ge cty/T$ for $0\le  t\le \tilde{T}$.  Hence there exists a constant $c_1>0$ depending only on $A_\infty T$ such that the process $\tilde{Y}_\ve(\cdot)$ of (\ref{AE3}) satisfies:
\begin{multline} \label{BG3}
\tilde{Y}_\ve(t)>0 \quad {\rm for \ } a\ve\le t\le \tilde{T} \ {\rm \ if \ for \ } k=1,2,.., \\
 \left|\int_0^{a\ve} d\tilde{B}(t')\right| \ < \ c_1\sqrt{T}\left(\frac{\ve}{T}\right)^{1/2-\al} \frac{y}{T} \quad {\rm and \ }  \sup_{a\ve\le t\le (k+1)a\ve}\left|\int_{a\ve}^t d\tilde{B}(t')\right| \ \le \
  c_1k\sqrt{T}\left(\frac{\ve}{T}\right)^{1/2-\al} \frac{y}{T} \ .
\end{multline}
It follows from (\ref{BG3}) that
\be \label{BH3}
P\left(\inf_{0\le s\le T} Y_\ve(s)>0\right) \ = \ P\left(\inf_{0\le t\le \tilde{T}} \tilde{Y}_\ve(t)>0\right)  \ \ge 
 \  P\left(\inf_{0\le t\le a\ve} \tilde{Y}_\ve(t)>0 \ ; \mathcal{E}\right) \ ,
\ee
where $\mathcal{E}$ is the event defined by the second line of (\ref{BG3}). It is easy to see from (\ref{AE3}) that for $0<t\le a\ve$ on the event $\mathcal{E}$ there is a constant $C>0$ depending only on $A_\infty T$ such that
\be \label{BI3}
\tilde{Y}_\ve(t)>0 \ {\rm if \ } \  \ \tilde{Z}_\ve(t)=\frac{\tilde{y}(t)}{1+Ca\ve/T}-Cc_1t \frac{y}{T} -\sqrt{\ve}\int_0^td\tilde{B}(t')  \ > 0,
\ee
where $c_1$ is the constant of (\ref{BG3}).  We conclude from (\ref{BG3}), (\ref{BI3}) that
\begin{multline} \label{BJ3}
P\left(\inf_{0\le s\le T} Y_\ve(s)>0\right) \ \ge \\
 P\left( \ \tilde{Z}_\ve(t)>0 , \ 0<t\le a\ve; \  \left|\int_0^{a\ve} d\tilde{B}(t')\right| \ < \ c_1\sqrt{T}\left(\frac{\ve}{T}\right)^{1/2-\al} \frac{y}{T}  \ \right) \ P(\mathcal{E}) \ .
\end{multline}

In order to bound $P(\mathcal{E})$ from below we consider for $\ga>0$ the event $\mathcal{E}_\ga$ defined by
\be \label{BK3}
 \left|\int_0^{1} d\tilde{B}(t')\right|  \ < \ \ga \ \quad {\rm and \ }  \sup_{1\le t\le (k+1)}\left|\int_{1}^t d\tilde{B}(t')\right| \ < \ k\ga \quad {\rm for \ } k=1,2,... 
\ee
Then we have that
\be \label{BL3}
P(\mathcal{E}) \ = \ P(\mathcal{E}_\ga) \quad {\rm where \ } \ga \ = \ c_1\left(\frac{T}{\ve}\right)^{\al/2}\left(\frac{y}{T}\right) \ .
\ee 
Using the fact that
\be \label{BM3}
P\left(  \    \sup_{1\le t\le (k+1)}\left|\int_{1}^t d\tilde{B}(t')\right| \ > \ k\ga       \   \right) \ \le \  4e^{-k\ga^2/2} \ ,
\ee
we conclude that
\be \label{BN3}
P(\mathcal{E}_\ga) \ \ge \ [1+e^{-\ga^2/4} \ ]^{-1} \quad {\rm if \ } \ga \ge 5.
\ee

We bound from below the first probability on the RHS of (\ref{BJ3}) by comparing it to the constant drift Brownian motion (\ref{E3}).  To do this we use the inequality   
\be \label{BO3}
\sig_A^2(T)y(s) \ \ge \ xm_{1,A}(s,T)\sig_A^2(0,s)+ym_{1,A}(0,s)\sig_A^2(s,T)+[\sig_A^2(0,s)-m_{2,A}(0,s)]\sig_A^2(s,T) \ ,
\ee
which follows from (\ref{AA2}) and the assumption that $A(s)\ge 0, \ 0\le s\le T$. Since the function $s\ra \sig_A^2(0,s)-m_{2,A}(0,s)$ is increasing we conclude from (\ref{AC3}), (\ref{BO3}) that there is a constant $C_1>0$ depending only on $A_\infty T$ such that for $0\le t\le a\ve$, 
\be \label{BP3}
\tilde{y}(t) \ \ge \ \frac{\la\ve+\mu_\ve t}{1+C_1a\ve/T}  \quad {\rm where \ } \mu_\ve\ = \ \frac{m_{1,A}(T)y}{\sig_A^2(T)}+\frac{\sig_A^2(0,s(a\ve))-m_{2,A}(0,s(a\ve))}{\sig_A^2(T)} \ .
\ee
It follows  from (\ref{BI3}), (\ref{BP3}) that for $0\le t\le a\ve$ there is a constant $C_2>0$ depending only on $A_\infty T$ such that 
\be \label{BQ3}
\tilde{Z}_\ve(t) \ \ge \ Z_\ve(t) \quad {\rm with  \ }  Z_\ve(0)=\frac{\la\ve}{1+C_2a\ve/T} \ , \  
 \mu \ = \ \frac{\mu_\ve}{1+C_2a\ve/T} 
 -C_2c_1 \frac{y}{T}  \  .
\ee
Hence the first probability on the RHS of (\ref{BJ3}) is bounded below by
\be \label{BR3}
P\left( \ Z_\ve(t)>0, \ 0<t\le a\ve; |Z_\ve(a\ve)-Z_\ve(0)-a\ve\mu|<\ga\ve\sqrt{a} \ \Big| \ Z_\ve(0)=\frac{\la\ve}{1+C_2a\ve/T} \ \right) \ ,
\ee
where $\ga$ is as in (\ref{BL3}). 

To bound the probability in (\ref{BR3})  we assume that the constant $c_1$ in (\ref{BL3}) is small enough so that $\mu>0$ and $\ga<\mu\sqrt{a}$.   Then similarly to  (\ref{J3}), (\ref{AN3}) we have that
 \begin{multline} \label{BS3}
P\left(\inf_{0<t<a\ve} Z_\ve(t)>0; \ |Z_\ve(a\ve)-\la'\ve-a\ve\mu|<\ga\ve\sqrt{a}  \ \big| \ Z_\ve(0)=\la'\ve \ \right) \ = \\ \left
\{1- e^{-2\mu\la'}\right\}\frac{1}{\sqrt{2\pi}}\int_{2\la'/\sqrt{a}-\ga}^{\ga} e^{-z^2/2} \ dz  \ 
+ \ \frac{1}{\sqrt{2\pi}}\int^{2\la'/\sqrt{a}-\ga}_{-\ga} e^{-z^2/2} \ dz \\
-e^{-2\mu\la'}\frac{1}{\sqrt{2\pi}}\int^{2\la'/\sqrt{a}+\ga}_{\ga} e^{-z^2/2} \ dz  \ \ge \ 
\left\{1- e^{-2\mu\la'}\right\}\frac{1}{\sqrt{2\pi}}\int_{-\ga}^{\ga} e^{-z^2/2} \ dz \ .
\end{multline}
We take $\la'=\la/[1+C_2a\ve/T]$ in (\ref{BS3}) and choose $\al=1/2$. Hence the drift $\mu$ of (\ref{BQ3}) satisfies the inequality 
\be\label{BT3}
\mu \ \ge \   \frac{1}{1+C_3\sqrt{\ve/T}}\left\{1-\frac{m_{2,A}(T)}{\sig_A^2(T)}+\frac{m_{1,A}(T)y}{\sig_A^2(T)}\right\} 
  -C_3c_1 \frac{y}{T}-C_3\left(\frac{\ve}{T}\right)^{1/2} \ ,
\ee
for some constant $C_3>0$ depending only on $A_\infty T$.  We choose now $c_1=c(\ve/T)^{1/8}$ where $c>0$ depends only on $A_\infty T$. It is clear that if $\ga\ge 5$ then the RHS of (\ref{BT3}) is bounded below by
\be \label{BU3}
 \frac{1}{1+C_4(\ve/T)^{1/8}}\left\{1-\frac{m_{2,A}(T)}{\sig_A^2(T)}+\frac{m_{1,A}(T)y}{\sig_A^2(T)}\right\} \quad {\rm where \ } C_4 \ {\rm depends \ only \ on \ } A_\infty T.
\ee 
The inequality (\ref{BF3}) follows now from (\ref{BN3}),  (\ref{BS3}), (\ref{BU3}).
\end{proof}

\vspace{.1in}

\section{Convergence as $\ve\ra 0$ of solutions to the diffusive CP Model}
\begin{lem}
Let $c_\ve(x,t),  \La_\ve(t), \ 0<x,t<\infty,$ be the solution to the diffusive CP system  (\ref{G1}), (\ref{H1}) with non-negative initial data $c_0(x), \ 0<x<\infty,$ which is a locally integrable function satisfying
\be \label{B4}
\int_0^\infty (1+x)c_0(x) \ dx<\infty, \quad \int_0^\infty xc_0(x) \ dx \ =  \ 1.
\ee
 Then for any $T>0$ there are positive constants $C_1,C_2$ depending only on $T$ and $c_0(\cdot)$  such that  
 \be \label{C4}
 C_1\le \La_\ve(t)\le C_2\quad  {\rm for \ } 0<\ve \le 1, \ 0\le t\le T. 
 \ee
 In addition the set of functions $\{\La_\ve:[0,T]\ra\R: \ 0<\ve\le 1\}$  form an equicontinuous family. 
Denote by $c_0(x,t),  \La_0(t), \ 0<x,t<\infty,$ the solution to the CP system (\ref{A1}),  (\ref{B1}) with $\ve=0$ and  initial data $c_0(x), \ 0<x<\infty$.  Then for all $x,t \ge 0$
\begin{eqnarray} \label{D4}
\lim_{\ve\ra 0}w_\ve(x,t)\ &=& \  w_0(x,t) \ ,  \\
\lim_{\ve\ra 0}\ \La_\ve(t) \ &=& \ \La_0(t) \ , \label{E4}
\end{eqnarray}
where $w_\ve$ is given in terms of $c_\ve$ by (\ref{AX7}). 
The limit in (\ref{D4}), (\ref{E4}) is uniform for $(x,t)$ in any finite rectangle $0<x\le x_0, \ 0<t\le T$. 
\end{lem}
\begin{proof}
 It follows from (\ref{J1}) that $\La_\ve(t)$ is an increasing function of $t$, whence the lower bound in (\ref{C4}) follows.  We first prove the upper bound for the CP model (\ref{A1}), (\ref{B1}) corresponding to $\ve =0$.  We see from  (\ref{B1}),  (\ref{C1}) that
\be \label{G4}
\int_{\La_0(0)/2}^\infty xc_0(x) \ dx \ \ge \frac{1}{2}\int_0^\infty xc_0(x) \ dx  \ = \ \frac{1}{2} \ .
\ee
Hence from (\ref{C7}) there is a positive constant $1/C_2$ depending only on $c_0(\cdot)$ such that $w_0(\La_0(0)/2,0)\ge 1/C_2$. It follows then from (\ref{B7}), (\ref{D7}) that $w_0(0,t) \ge 1/ C_2$ for $0\le t\le \La_0(0)/2$, whence  (\ref{B1}), (\ref{C1}) implies that $\La_0(t)\le C_2$ for $0\le t\le \La_0(0)/2$.  Furthermore we see from (\ref{C7}) that $\La_0(t)$ is continuous in the interval $0\le t\le \La_0(0)/2$. Since $\La_0(t)$ is an increasing function of $t$ we can extend this argument in a finite number of steps to any interval $0\le t\le T$.  We have proven (\ref{C4}) in the case $\ve=0$. 

To prove the upper bound in (\ref{C4}) for $0<\ve\le 1$ we use the representation 
\be \label{H4}
w_\ve(x,t)  \ = \ 
\int_0^\infty P\left(Y_\ve(t)>x; \inf_{0\le s\le t}Y_\ve(s) >0 \  \big| \ Y_\ve(0)=y\right) c_0(y)  dy \ ,
\ee
where $Y_\ve(s)$ is the solution to the SDE (\ref{H2}) with $b(y,s)=y/\La_\ve(s)-1$. Since $\La_\ve(s)\ge\La_\ve(0)=\La_0(0)$ it follows from (\ref{I2}) that for any $\del>0,  t>0$ there is a positive constant  $p_1$ depending only on $\del,t,\La_0(0)$ such that
\be \label{I4}
P\left( \inf_{0\le s\le t} \{ \ Y_\ve(s)-E[Y_\ve(s)] \ \} \ge -\del \right) \ \ge \ p_1  \quad {\rm for \ } 0<\ve\le 1.
\ee
We conclude from (\ref{I4}) by choosing $\del$ appropriately that there is a positive constant $p_2$ depending only on $\La_0(0)$ such that if $0<\ve\le 1$ then
\be \label{J4}
P\left(Y_\ve(t)>0; \inf_{0\le s\le t}Y_\ve(s) >0 \  \bigg| \ Y_\ve(0)=y\right)  \ \ge \ p_2 \quad {\rm for \ } t=\La_0(0)/2, \ y\ge \La_0(0)/2 \ .
\ee 
It follows now from (\ref{G4}), (\ref{H4}), (\ref{J4}) that there is a positive constant $C_2$ depending only on the initial data $c_0(\cdot)$ such that $w_\ve(0,t)\ge 1/C_2$ for $0<\ve\le 1$ if $t=\La_0(0)/2$.  The upper bound in (\ref{C4}) for all $T$ then follows as in the previous paragraph. 

To prove that  $\La_\ve(\cdot)$ is continuous we first note that for any fixed $t>0$ the function $w_\ve(x,t), \ x\ge 0,$ is  continuous by virtue of the representation (\ref{H4}), the fact that  $\La_\ve(s)\ge \La_0(0)$ for $0\le s\le t$ and (\ref{K2}).  The continuity is uniform for $\ve$ in the interval $0<\ve\le 1$ since $c_0(\cdot)$ is a locally integrable  function. Next we observe from (\ref{I2}) that for $\De t>0$ there exists $x(\De t)$ independent of $\ve$ in the interval $0<\ve\le 1$  such that $\lim_{\De t\ra 0}x(\De t)=0$ and
\begin{multline} \label{K4}
P\left(Y_\ve(t+\De t)>0; \inf_{0\le s\le t+\De t}Y_\ve(s) >0 \  \bigg| \ Y_\ve(0)=y\right)  \ \ge \\
[1-\De t]P\left(Y_\ve(t)>x(\De t); \inf_{0\le s\le t}Y_\ve(s) >0 \  \bigg| \ Y_\ve(0)=y\right) \quad {\rm  for \ } y\ge 0, \ 0<\ve\le 1. 
\end{multline}
It follows from (\ref{H4}), (\ref{K4}) that $w_\ve(0,t+\De t)\ge [1-\De t]w_\ve(x(\De t),t)$ for $0<\ve\le 1$. Using the continuity of the function $w_\ve(x,t), \ x\ge 0,$ we conclude that $\lim_{\De t\ra 0}w_\ve(0,t+\De t)=w_\ve(0,t)$ and the limit is uniform for $0<\ve\le 1$. Hence the function $\La_\ve(\cdot)$ is continuous, and in fact the family of functions $\La_\ve(\cdot), \ 0<\ve\le 1,$ is equicontinuous.

To prove (\ref{D4}), (\ref{E4}) we first observe from the Ascoli-Arzela theorem that since the family of functions $\La_\ve(\cdot), \ 0<\ve\le 1,$ is equicontinuous, the limit (\ref{E4}) holds uniformly on the interval $0\le t\le T$ for a subsequence of $\ve\ra 0$. For such a sequence  it follows from (\ref{B7}), (\ref{I2}), (\ref{H4}) that (\ref{D4}) holds with $w_0(x,t)=w_0(F_{1/\La_0}(x,t),0)$ and the conservation law (\ref{B1}) continues to hold for $\ve=0$. Hence the limits on the RHS of (\ref{D4}), (\ref{E4}) are the solution to the CP model (\ref{A1}), (\ref{B1}) and are therefore unique. Consequently (\ref{D4}), (\ref{E4}) hold for all $\ve\ra0$. The uniformity of the limits follows by similar argument. 
\end{proof}
To show that the coarsening rate (\ref{J1}) for the diffusive model (\ref{G1}), (\ref{H1}) converges as $\ve\ra 0$ to the coarsening rate (\ref{D1}) for the CP model (\ref{A1}), (\ref{B1}), it will be sufficient to prove the following:
\begin{lem}
Let $c_\ve(x,t),  \La_\ve(t), \ 0<x,t<\infty,$ and  $c_0(x), \ 0<x<\infty,$  be as in Lemma 6.1 and satisfy (\ref{B4}). If $c_0(\cdot)$ is a continuous function then
\be \label{L4}
\lim_{\ve\ra0}\frac{\ve}{2}\frac{\pa c_\ve(0,T)}{\pa x} \ = \ c_0(0,T) \quad {\rm for \ any  \ }T>0.
\ee
\end{lem}
\begin{proof}
We use the identity
\be \label{M4}
\frac{\ve}{2}\frac{\pa c_\ve(0,T)}{\pa x} \ = \  \lim_{\la\ra 0} \frac{c_\ve(\la\ve,T)}{2\la}
\ee
and the representation  for $c_\ve(\la\ve,T)$ from (\ref{M2}),
\be \label{N4}
c_\ve(\la\ve,T) \ = \ \int_{0}^\infty G_{\ve,D}(\la\ve,y,0,T) c_0(y) \ dy, 
\ee
where $G_{\ve,D}$  is the Dirichlet Green's function corresponding to the  drift $b(y,t)=y/\La_\ve(t)-1$.  From (\ref{M4}), (\ref{N4}) and Lemma 5.1 we have that
\begin{multline} \label{O4}
\frac{\ve}{2}\frac{\pa c_\ve(0,T)}{\pa x} \ \le \\
 \int_{0}^\infty \left\{1-\frac{m_{2,1/\La_\ve}(T)}{\sig_{1/\La_\ve}^2(T)}+\frac{m_{1,1/\La_\ve}(T)y}{\sig_{1/\La_\ve}^2(T)}+C\Ga\left(\frac{\ve}{T},\frac{y}{T}\right) \left[1+\frac{y}{T}\right] \ \right\}G_{\ve}(0,y,0,T) c_0(y) \ dy, 
\end{multline} 
where the constant $C$ depends only on $T/\La_0(0)$.   We conclude from Lemma 6.1, (\ref{K2}) and  (\ref{O4})  that 
\be \label{P4}
\limsup_{\ve\ra 0} \frac{\ve}{2}\frac{\pa c_\ve(0,T)}{\pa x} \ \le \  \frac{1}{m_{1,1/\La_0}(T)} c_0\left(\frac{m_{2,1/\La_0}(T)}{m_{1,1/\La_0}(T)}\right) \ ,
\ee
provided the function $c_0(y), \ y>0,$ is continuous at $y=m_{2,1/\La_0}(T)/m_{1,1/\La_0}(T)$.

We can obtain a lower bound on the LHS of (\ref{L4})  by using Lemma 5.2. Thus we have that 
\begin{multline} \label{Q4}
\left[1+C(\ve/T)^{1/8}\right] \frac{\ve}{2}\frac{\pa c_\ve(0,T)}{\pa x} \ \ge \\
\int_{5(\ve/T)^{1/8}T/c}^\infty \left[1+\exp\left(-c^2(T/\ve)^{1/4}(y^2/4T^2\right) \ \right]^{-2} \ \times  \\
\left\{1-\frac{m_{2,1/\La_\ve}(T)}{\sig_{1/\La_\ve}^2(T)}+\frac{m_{1,1/\La_\ve}(T)y}{\sig_{1/\La_\ve}^2(T)} \ \right\}G_{\ve}(0,y,0,T) c_0(y) \ dy, 
\end{multline} 
where the constants $C,c>0$ depend only on $T/\La_0(0)$. We conclude from Lemma 6.1, (\ref{K2}) and  (\ref{Q4})  that 
\be \label{R4}
\liminf_{\ve\ra 0} \frac{\ve}{2}\frac{\pa c_\ve(0,T)}{\pa x} \ \ge \  \frac{1}{m_{1,1/\La_0}(T)} c_0\left(\frac{m_{2,1/\La_0}(T)}{m_{1,1/\La_0}(T)}\right) \ ,
\ee
provided the function $c_0(y), \ y>0,$ is continuous at $y=m_{2,1/\La_0}(T)/m_{1,1/\La_0}(T)$. Finally we observe that  the RHS of (\ref{P4}), (\ref{R4}) is the same as $c_0(0,T)$. This follows by differentiating the function $w(x,t)=w_0(F_{1/\La_0}(x,t),0)$ with respect to $x$ at $x=0$, and using the formula (\ref{B7}) for the function $F_{1/\La_0}$.
\end{proof}

\vspace{.1in}

\section{Upper Bound on the Coarsening Rate of diffusive CP Model}
In this section we prove Theorem 1.2. First we show that $\lim_{t\ra\infty}\langle X_t\rangle=\infty$. 
\begin{lem}
Let $c_\ve(x,t),  \La_\ve(t), \ 0<x,t<\infty,$ be the solution to (\ref{G1}), (\ref{H1}) with $\ve>0$ and non-negative initial data $c_0(x), \ 0<x<\infty,$ which is a locally integrable function satisfying (\ref{B4}). Then $\lim_{t\ra\infty}\La_\ve(t)=\infty$. 
\end{lem} 
\begin{proof}
We have already noted that  $\La_\ve(t)$ is an increasing function of $t$. It will therefore be sufficient to show that if  for some finite $\La_\infty$ we have $\La_\ve(t)\le \La_\infty$ for all $t\ge 0$ then there is a contradiction. To see this we use the identity
\be \label{P5}
\frac{d}{dt}\int_0^\infty xc_\ve(x,t) \ dx \ = \  \frac{1}{\La_\ve(t)}\int_0^\infty xc_\ve(x,t) \ dx -\int_0^\infty c_\ve(x,t) \ dx \ ,
\ee
which follows  from (\ref{G1}). Using the conservation law (\ref{H1}) and (\ref{P5}) we see that
\begin{multline} \label{Q5}
\frac{d}{dt}\int_0^\infty xc_\ve(x,t) \ dx \ \ge \  \frac{1}{2\La_\infty}\int_0^\infty xc_\ve(x,t) \ dx -\int_0^{2\La_\infty} c_\ve(x,t) \ dx \\
= \ \frac{1}{2\La_\infty} -\int_0^{2\La_\infty} c_\ve(x,t) \ dx  \ \ge          \ 
 \frac{1}{2\La_\infty} -\int_0^{2\La_\infty} dx\int_0^\infty dy \ G_\ve(x,y,0,t)c_0(y) \  ,
\end{multline}
where $G_\ve$ is the function (\ref{K2}) with $A(s)=1/\La_\ve(s), \ s\ge 0$.  Hence we conclude that
\be \label{R5}
\frac{d}{dt}\int_0^\infty xc_\ve(x,t) \ dx \ \ge \  \frac{1}{2\La_\infty}- \frac{2\La_\infty }{\La_\ve(0)\sqrt{2\pi\ve\sig_{1/\La_\ve}^2(t)}} \ .
\ee 
From (\ref{J2}) we see that $\sig_{1/\La_\ve}^2(t)\ge t$ and hence (\ref{R5}) implies that
\be \label{S5}
\lim_{t\ra\infty}\int_0^\infty xc_\ve(x,t) \ dx  \ = \ \infty \ ,
\ee
but this is a contradiction to the conservation law (\ref{H1}). 
\end{proof}
We begin the proof of the inequality (\ref{M1}):
\begin{lem}
Suppose $c_0:[0,\infty)\ra\R^+$ satisfies (\ref{B4}) and $c_\ve(x,t), \ x\ge0, t>0$ is the solution to (\ref{G1}), (\ref{H1})  with initial data $c_0(\cdot)$ and  Dirichlet boundary condition $c_\ve(0,t)=0, \ t>0$. Assume that $\La_\ve(0)=1$ and that the function $h_\ve(x,t)$ defined by (\ref{AX7}) is log-concave in $x$ at $t=0$. Then there exist positive universal constants $C,\ve_0$ with $0<\ve_0\le 1$ such that
\be \label{T5}
c_\ve(\la\ve,1) \ \le \ C\la c_\ve(\ve,1) \quad {\rm for \ } 0<\ve\le \ve_0,\  0< \la\le 1. 
\ee
\end{lem}
\begin{proof}
Let $X_0$ be the positive random variable with pdf $c_0(x)/\int_0^\infty c_0(x') \ dx', \ x>0$.  Then from (\ref{I1})  and the assumption $\La_\ve(0)=1$ we see that $\langle X_0\rangle=1$. Since the beta function (\ref{AI7}) for $X_0$ is also bounded by $1$ it follows from the Chebshev inequality and the identity (29) of \cite{c2} that for $\del$ with $0<\del<1$, there exists a constant $\nu(\del)>0$ depending only on $\del$  such that
\be \label{U5}
P(X_0<\nu(\del))+P(X_0>1/[1-\del]) \  \le 1-\del/2  .
\ee
Now recall that the Green's function (\ref{K2})  has the form (\ref{V2}) where the function $y\ra q(x,y,0,t)$ takes its minimum at  $y=F_A(x,t)$ where $F_A$ is defined by (\ref{B7}). In the case of $A(\cdot)=1/\La_\ve(\cdot), \ x=\la\ve$ and $t=1$, we see that  $ 1-(1-\la\ve)/e \le F_{1/\La_\ve}(\la\ve,1)\le 1+ \la\ve$. This follows from the fact that the function $\La_\ve(\cdot)$ is increasing and $\La_\ve(0)=1$. We choose now $\del,\ve_0>0$ sufficiently small so that $F_{1/\La_\ve}(\la\ve,1)-\nu(\del)>1/(1-\del)-F_{1/\La_\ve}(\la\ve,1)>0$ for $ 0<\la\le 1, \ \ve\le \ve_0$. 
It follows then from Lemma 5.1 that
\be \label{V5}
c_\ve(\la\ve,1)  \ \le \ C_1\la \int_{\nu(\del)}^{1/[1-\del]} G_\ve(\ve,y,0,1) c_0(y) \ dy \qquad {\rm for \ } 0<\ve\le\ve_0, \ 0<\la\le 1,
\ee
where $C_1>0$ depends only on $\del,\ve_0$.   Next we apply Lemma 5.2 to conclude that
\be \label{W5}
\int_{\nu(\del)}^{1/[1-\del]} G_\ve(\ve,y,0,1) c_0(y) \ dy  \ \le \ C_2\int_{\nu(\del)}^{1/[1-\del]} G_{\ve,D}(\ve,y,0,1) c_0(y) \ dy  \ \le \ C_2 c_\ve(\ve,1) \ ,
\ee
for some constant $C_2>0$ depending only on $\del,\ve_0$. Actually a strict application of Lemma 5.2 requires us to impose an additional restriction on $\ve_0>0$ so that the condition $\ga\ge 5$ of Lemma 5.2 holds.  Now (\ref{T5}) follows from (\ref{V5}), (\ref{W5}). 
\end{proof}
\begin{lem}
Suppose  the initial data $c_0(\cdot)$ for (\ref{G1}), (\ref{H1}) satisfies the conditions of Lemma 7.2, the function $h_\ve(x,t)$ is log-concave in $x$ for all $t\ge 0$, and $0<\ve\le\ve_0$. Then there is a universal constant $C$ such that $d\La_\ve(t)/dt\le C$ for $t\ge 1$. 
\end{lem}
\begin{proof}
From (\ref{H1}), (\ref{J1}), (\ref{AI7}),  (\ref{AX7}) and Lemma 7.2 we see there is a universal constant $C_1$ such that
\be \label{X5}
\frac{d\La_\ve(1)}{dt} \ \le \  \frac{C_1c_\ve(\ve,1)h_\ve(0,1)}{w_\ve(0,1)^2} \ \le \ \frac{C_1\beta_{X_1}(\ve)h_\ve(0,1)}{h_\ve(\ve,1)} \ ,
\ee
where $X_t$ is the random variable with pdf proportional to $c_\ve(\cdot,t)$.
We can bound $h_\ve(\ve,1)$ below by a constant times $h_\ve(0,1)$. To see this consider a positive random variable $X$ and observe that
\begin{multline} \label{Y5}
E[ \ X-\langle X\rangle/2; \ X>\langle X\rangle/2 \ ] \ \ge \   E[ \ X-\langle X\rangle/2; \ X>3\langle X\rangle/4 \ ] \ \ge \\ \frac{1}{3} E[ \ X; \ X>3\langle X\rangle/4 \ ]  \ \ge  \ \frac{1}{12} \langle X\rangle \ .
\end{multline}  
Since the function $\La_\ve(\cdot)$ is increasing, it follows that $\langle X_1\rangle \ge 1$.  
We conclude  then from (\ref{Y5}) that
\be \label{Z5}
h_\ve(\ve,1) \ = \ \int_\ve^\infty (x-\ve) c_\ve(x,1) \ dx \ \ge \ \frac{1}{12} \int_0^\infty x c_\ve(x,1) \ dx \ = \ \frac{1}{12}h_\ve(0,1) \ ,
\ee
provided $\ve<1/2$. 
Since the log-concavity of $h_\ve(\cdot,1)$ implies that $\beta_{X_1}(\ve)\le 1$, we obtain from (\ref{X5}), (\ref{Z5}) an upper bound on $d\La_\ve(t)/dt$ when $t=1$. 

The upper bound for $t>1$ now follows from the scaling property of (\ref{G1}), (\ref{H1}) mentioned in the discussion following the statement of Theorem 1.2.  To see this we define a function $\tau(\la), \ \la\ge 1, $ as the solution to the equation $\La_\ve(\la \tau(\la))=\la$.  Observe from (\ref{J1}) and the Hopf maximum principle \cite{pw} that the function $\La_\ve(\cdot)$ is strictly increasing.  Hence $\tau(\la)$ is uniquely determined. Furthermore the function $\tau(\cdot)$ is continuous. Rescaling (\ref{G1}), (\ref{H1}) by $\la$, we conclude from the result of the previous paragraph  that
\be \label{AA5}
\frac{d}{dt} \ \La_\ve(\la[\tau(\la)+t]) \ \le C\la \quad {\rm at \ } t=1. 
\ee
We have shown then that $d\La_\ve(t)/dt\le C$ at $t=\la[\tau(\la)+1]$.  Since the function $\la\ra\la\tau(\la)$ is monotonically increasing with range $[0,\infty)$ the result follows. 
\end{proof}
To complete the proof of the inequality (\ref{M1}) we first observe that by Lemma 7.1 there exists $T_\ve\ge 0$ such that $\ve/\La_\ve(T_\ve)\le \ve_0$, where $\ve_0$ is the universal constant of Lemma 7.2. We now rescale (\ref{G1}), (\ref{H1}) with $\la=\La_\ve(T_\ve)$, which puts us into the situation of Lemma 7.3. The result follows on taking $T=T_\ve+\La_\ve(T_\ve)$, provided we have the log-concavity property of the function $h_\ve$ in the statement of Lemma 7.3. The assumption of Theorem 1.2, that the function $x\ra E[X_0-x \ | \ X_0>x], \ 0\le x<\|X_0\|_\infty$, decreases, is equivalent to the assumption that the function $h_\ve(\cdot,0)$ is log-concave. Hence it remains to be shown that if $h_\ve(\cdot,0)$ is log-concave, then $h_\ve(\cdot,t)$ is also log-concave for all $t>0$. 

If we make the approximation (\ref{AZ7}) for $h_\ve$, the log-concavity of $h_\ve(\cdot,t)$ follows from the Pr\'{e}kopa-Leindler theorem (Theorem 6.4 of \cite{villani}).  In our situation we  follow the approach of Korevaar \cite{kor} and differentiate the PDE (\ref{BN7}) which $h_\ve$ satisfies.  Thus $v_\ve(x,t)=-\frac{\pa}{\pa x} \log h_\ve(x,t)$ is a solution of the PDE (\ref{BP7}), whence $u_\ve(x,t)=\pa v_\ve(x,t)/\pa x$ is a solution to the PDE
\be \label{M5}
\frac{\pa u_\ve(x,t)}{\pa t}+\left[\frac{x}{\La_\ve(t)}-1+\ve v_\ve(x,t)\right]\frac{\pa u_\ve(x,t)}{\pa x}
 +\frac{2u_\ve(x,t)}{\La_\ve(t)} +\ve u_\ve(x,t)^2\ = \  \frac{\ve}{2} \frac{\pa^2 u_\ve(x,t)}{\pa x^2} \ .
\ee
Observe now that
\be \label{N5}
u_\ve(x,t) \  = \ v_\ve(x,t)^2[1-c_\ve(x,t)h_\ve(x,t)/w_\ve(x,t)^2] \ . 
\ee
Since $\lim_{x\ra 0} c_\ve(x,t)=0$ for $t>0$  it follows from (\ref{N5}) that $ \liminf_{x\ra 0} u_\ve(x,t)\ge 0$ for $t>0$. If $h_\ve(\cdot,0)$ is log concave  then the initial data $u_\ve(x,0), \ x>0,$ for (\ref{M5}) is also non-negative. We expect then from the maximum principle that $u_\ve(x,t)$ is non-negative for all $x,t>0$, and hence $h_\ve(\cdot,t)$ is a log concave function for all $t>0$.  

Although Korevaar's argument is simple in concept, the rigorous implementation requires that certain technical difficulties be overcome.  Our first step towards rigorous implementation is to approximate an arbitrary non-negative random variable $X$ satisfying $\langle X\rangle<\infty$, and having log-concave function $h_X(\cdot)$ as defined by (\ref{AG7}), by random variables with some regularity. The approximating random variables $Y$ have the properties:
\be \label{O5}
Y \ {\rm is \ nonnegative  \ with \ continuous \ pdf \ } c_Y(y), \ y\ge 0, \quad {\rm and \ } c_Y(0)=0.
\ee
$$
{\rm There \ exists \ } K,a,L,y_0>0 \ {\rm and \ } c_Y(y)=K\exp\left[-a(y-y_0)-\{a(y-y_0)\}^2/2L\right] \ {\rm for \ } y\ge y_0.
$$
$$
{\rm The \  beta \  function \  (\ref{AI7}) \ of \ } Y \ {\rm satisfies \ } \beta_Y(y)<1 \quad {\rm for  \ all \ } y\ge 0. 
$$
We assume that the function $\La_\ve:[0,\infty)\ra\R^+$  of (\ref{G1}) is positive and continuous, and consider solutions $c_\ve$ to (\ref{G1}) with Dirichlet boundary condition $c_\ve(0,t)=0, \ t>0$, and initial condition given by the integrable pdf $c_{X_0}(\cdot)$ of a non-negative random variable  $X_0$  satisfying $\langle X_0\rangle<\infty$. We denote by $X_t$ the random variable with pdf $c_\ve(\cdot,t), \ t>0$. 
\begin{lem}
Assume that the function $h_{X_0}(\cdot)$ for the initial data random variable $X_0$ of (\ref{G1}) is log-concave. Then there is a sequence of random variables $Y_0^k, \ k=1,2,..,$ satisfying (\ref{O5}) such that the functions $(x,t)\ra h_{Y^k_t}(x), \ k=1,2,..,$ converge  as $k\ra\infty$, uniformly in any rectangle $\{(x,t) \ : \ 0\le x\le x_0, \ 0\le t\le T\}$, to the function $(x,t)\ra h_{X_t}(x)$. 
\end{lem}
\begin{proof}
We first assume that $\|X_0\|_\infty<\infty$, and define the beta function for $Y_0^k$  in the interval $0\le x\le (1-2/k)\|X_0\|_\infty$ in terms of the beta function for $X_0$  as follows:
\begin{multline} \label{AB5}
\beta_{Y^k_0}(x) \ =  \ \left(1-\frac{1}{k}\right)\frac{k}{\|X_0\|_\infty}\int_0^x \beta_{X_0}(z) \ dz \quad {\rm for \ } 0\le x\le \frac{\|X_0\|_\infty}{k} \ , \\
\beta_{Y^k_0}(x) \ =  \ \left(1-\frac{1}{k}\right)\frac{k}{\|X_0\|_\infty}\int_{x-\|X_0\|_\infty/k}^x \beta_{X_0}(z) \ dz \quad {\rm for \ }  \frac{\|X_0\|_\infty}{k}\le x\le \left(1-\frac{2}{k}\right)\|X_0\|_\infty \ .
\end{multline}
It follows from (\ref{AI7}), (\ref{AB5}) that $c_{Y^k_0}(x)$ is continuous in the interval $0\le x\le (1-2/k)\|X_0\|_\infty$ and  $c_{Y^k_0}(0)=0$. To continue the definition of $\beta_{Y^k_0}(\cdot)$, we choose $L_k>0$ sufficiently large  so that the function $\beta_{L_k}$ of  Lemma 2.1 satisfies $\beta_L(0)\ge 1-1/k$ and $\beta_L(z)\le 1-1/2L(1+z/L)^2$ for $z\ge 0, \ L\ge L_k$.  We then define $\beta_{Y^k_0}(x)$ for $(1-2/k)\|X_0\|_\infty\le x\le (1-1/k)\|X_0\|_\infty$ by linear interpolation, taking the value $\beta_{Y^k_0}\left((1-2/k)\|X_0\|_\infty\right)$ at the left end of the interval and the value $\beta_{L_k}(0)$ at the right end.  Finally we extend $c_{Y^k_0}(x)$ to $x\ge (1-1/k)\|X_0\|_\infty$ by setting it equal to the Gaussian in (\ref{O5}) with $y_0=(1-1/k)\|X_0\|_\infty$, choosing $K$ so that $c_{Y^k_0}(\cdot)$ is continuous and  $a$ so that  $E[Y^k_0-x \ | \ Y^k_0>x]=E[X_0-x \ | \ X_0>x]$ when $x=(1-1/k)\|X_0\|_\infty$. The random variable $Y^k_0$ satisfies (\ref{O5}). In particular, since $h_{X_0}(\cdot)$ is log-concave it follows that $\beta_{Y^k_0}(x)\le 1-1/k$ for $x\le (1-2/k)\|X_0\|_\infty$.

To construct the pdf $c_{Y^k_0}(\cdot)$ from the function $\beta_{Y^k_0}(\cdot)$ defined in the previous paragraph we first observe from (\ref{AI7}) that
\be \label{AC5}
\frac{d}{dy} E[Y^k_0-y \ | \ Y^k_0>y] \ = \ \beta_{Y^k_0}(y)-1 \ , \quad E[Y^k_0-y^k \ | \ Y^k_0>y^k]=E[X_0-y^k \ | \ X_0>y^k] \ ,
\ee
where $y^k=(1-1/k)\|X_0\|_\infty$.  The function $v^k(y)=E[Y^k_0-y \ | \ Y^k_0>y]^{-1}, \ y\ge 0,$ is uniquely determined by (\ref{AC5}).  From (\ref{AG7}), (\ref{AH7}) we see that
\be \label{AD5}
h_{Y^k_0}(x) \ = \ A^k \exp\left[-\int_0^x v^k(y) \ dy\right]\ , \quad x\ge 0, 
\ee
for some constant $A^k$.  If we define the function $f_k(\cdot)$ by
\be \label{AE5}
f_k(x)  \ = \ \beta_{Y^k_0}(x)v^k(x)^2 \exp\left[-\int_0^x v^k(y) \ dy\right] \ , \quad x\ge 0,
\ee
then (\ref{AI7}) implies that $c_{Y^k_0}(x)=A^kf_k(x), \ x\ge 0$.  Using the normalization condition for the probability  measure $c_{Y^k_0}(\cdot)$, we conclude that the constant $K$ in (\ref{O5}) is given by the formula
\be \label{AF5}
K \ = \ f_k(y^k)\big/ \int_0^\infty f_k(x) \ dx \ .
\ee

We show that the functions $w_{Y^k_0}(\cdot)$ converge as $k\ra\infty$ to $w_{X_0}(\cdot)$, uniformly in $[0,\infty)$. To do this we use the identity
\be \label{AG5}
\int_x^\infty f_k(y) \ dy \ = \  v^k(x) \exp\left[-\int_0^x v^k(y) \ dy\right] \ , \quad x\ge 0,
\ee
obtained from (\ref{AC5}), (\ref{AE5}).  From (\ref{AC5}) we have that
\be \label{AH5}
 E[Y^k_0-y \ | \ Y^k_0>y]  \ = \ E[X_0-y^k \ | \ X_0>y^k]+\int_y^{y^k} \left[1-\beta_{Y^k_0}(y')\right] \ dy' \ , \quad 0\le y\le y^k \ .
\ee
 It follows from (\ref{AB5}), (\ref{AH5}) that
\be \label{AI5}
\lim_{k\ra\infty}  E[Y^k_0-y \ | \ Y^k_0>y]  \ = \ \int_y^{\|X_0\|_\infty} \left[1-\beta_{X_0}(y')\right] \ dy' \ , \quad 0\le y<\|X_0\|_\infty \ ,
 \ee
and the convergence is uniform in any interval $\{y \ : \ 0\le y<\|X_0\|_\infty(1-\del)\}$ for which $\del>0$. We conclude from (\ref{AG5}), (\ref{AI5}) upon setting $v^\infty(x)=E[X_0-x \ | \ X>x]^{-1}$ that
\be \label{AJ5}
\lim_{k\ra\infty}\int_x^\infty f_k(y) \ dy \ = \  v^\infty(x) \exp\left[-\int_0^x v^\infty(y) \ dy\right] \ , \quad 0\le x<\|X_0\|_\infty \ ,
\ee
and the convergence is uniform in any interval $\{x \ : \ 0\le x<\|X_0\|_\infty(1-\del)\}$ for which $\del>0$. Taking $x=0$ in (\ref{AJ5}) we have that $\lim_{k\ra\infty} A_k=E[X_0]$. Hence (\ref{AJ5}) implies that
\be \label{AK5}
\lim_{k\ra\infty}\int_x^\infty c_{Y^k_0}(y) \ dy  \ = \ \int_x^\infty c_{X_0}(y) \ dy \ , \quad  0\le x<\|X_0\|_\infty \ ,
\ee
and the convergence is uniform in any interval $\{x \ : \ 0\le x<\|X_0\|_\infty(1-\del)\}$ for which $\del>0$. In view of the integrability of $c_{X_0}(\cdot)$, we conclude that the convergence (\ref{AK5}) is uniform for $0\le x<\infty$. 

We can easily estimate $w_{Y^k_0}(x)$ for $x\ge\|X_0\|_\infty$ since the pdf of $Y^k_0$ is Gaussian when $x\ge\|X_0\|_\infty$.  To do this we define a function $g:[0,\infty)\ra\R$ by 
\be \label{AL5}
g(z) \ = \  E[Z \ |  \ Z>z]^{-1}  \ = \ e^{z^2/2}\int_z^\infty e^{-z'^2/2}  \ dz' \ = \ \int_0^\infty e^{-z'z-z'^2/2}  \ dz' \ ,
\ee
where $Z$ is the standard normal variable.  Evidently we have from (\ref{AL5}) that
\be \label{AM5}
g'(z) \ = \ -1+ z\int_0^\infty e^{-z'z-z'^2/2}  \ dz' \ , \quad g(0) \ = \ \sqrt{\pi/2} \ .
\ee
We conclude from (\ref{AL5}), (\ref{AM5}) that $g(\cdot)$ is a decreasing function and $\lim_{z\ra\infty} g(z)=0$.  We can estimate $g(z)$ for large $z$ from the final integral on the RHS of (\ref{AL5}) to obtain the inequality
\be \label{AO5}
0 \ < \  g(z) \ < \frac{1}{z}\left[1-\frac{1}{z^2}+\frac{3}{z^4} \right] \ ,
\ee
whence it follows that
\be \label{AP5}
E[Z \ |  \ Z>z] \ > \ z\left[1+\frac{1}{2z^2}\right] \quad {\rm for \ } z> \sqrt{6}.
\ee
Since the final integral on the RHS of (\ref{AL5}) is strictly less than $1/z$ for all $z>0$, we conclude from (\ref{AP5}) that there exists $\ga_0>0$ and
\be \label{AQ5}
E[Z \ |  \ Z>z] -z \ \ \ge \ \min\{\ga_0, 1/2z\} \ , \quad z\ge 0.
\ee
The random variable $Y$ of (\ref{O5}) has for $y\ge y_0$ the pdf of a normal variable with mean  $y_0-L/a$ and variance $L/a^2$.  We can estimate the value of $a$ when $Y=Y^k_0$  by using the equality $E[Y^k_0-y^k \ | \ Y^k_0>y^k]=E[X_0-y^k \ | \ X_0>y^k]\le \|X_0\|_\infty/k$. Observe now  that
\be \label{AR5}
E[Y^k_0-y^k \ | \ Y^k_0>y^k] \ = \ \frac{\sqrt{L_k}}{a}\left(E[Z \ |  \ Z>\sqrt{L_k}]-\sqrt{L_k}\right) \ .
\ee
Hence using the upper bound on the LHS of (\ref{AR5}),  we obtain from (\ref{AQ5}), (\ref{AR5}) a lower bound
\be \label{AS5}
a \ \ge \  \frac{k\sqrt{L_k}}{\|X_0\|_\infty} \min\{\ga_0,1/2\sqrt{L_k}\} \ 
\ee 
for $a$.  
Since $\lim_{k\ra\infty} L_k=\infty$, it follows from (\ref{AS5}) that  $\lim_{k\ra\infty} a_k=\infty,$ where $a_k$ is the value of $a$ in (\ref{O5})  when $Y=Y^k_0$. 

We have from (\ref{O5}) that
\be \label{AT5}
w_{Y^k_0}(x) \ = \ \int_{x-y_k}^\infty K_k\exp[-a_ky-(a_ky)^2/2L_k] \ dy \quad {\rm for \ } x\ge y_k \ .
\ee
Furthermore, from (\ref{AK5}) it follows that for any $\eta>0$ there exists an integer $k_\eta$ such that $w_{Y^k_0}(y_k)<\eta$ for $k\ge k_\eta$.  This implies a bound on $K_k$ in (\ref{AT5}) of the form $K_k\le \eta a_k, \ k\ge k_\eta$.  Hence from (\ref{AT5}) it follows that $ h_{Y^k_0}(y_k)\le \eta/ a_k$ and hence $\lim_{k\ra\infty} h_{Y^k_0}(y_k)=0$.   We conclude from this and (\ref{AK5}) that the functions $x\ra h_{Y^k_0}(x), \ k=1,2,..,$ converge  as $k\ra\infty$, uniformly in any interval $\{x \ : \ 0\le x\le x_0 \}$, to the function  $x\ra h_{X_0}(x)$. 

To prove that $h_{Y^k_t}(\cdot), \ k=1,2,..,$ converges to $h_{X_t}(\cdot)$ for $t>0$ we use the fact that the function $w_\ve$ defined by (\ref{AX7}) is a solution to the PDE (\ref{BM7}).  Furthermore, the Dirichlet boundary condition $c_\ve(0,t)=0$ for (\ref{G1})  becomes a Neumann boundary condition $\pa w_\ve(0,t)/\pa x=0$ for (\ref{BM7}). By the Hopf maximum principle \cite{pw} we then have that
\be \label{AU5}
\sup \left|w_{Y^k_t}(\cdot)-w_{X_t}(\cdot)\right| \ \le \ \sup \left|w_{Y^k_0}(\cdot)-w_{X_0}(\cdot)\right| \ , \quad t\ge 0. 
\ee
From (\ref{AK5}), (\ref{AU5}) we see that $w_{Y^k_t}(\cdot), \ k=1,2,..,$ converges to $w_{X_t}(\cdot)$ for any $t\ge 0$.  This implies convergence of  $h_{Y^k_t}(\cdot), \ k=1,2,..,$ to $h_{X_t}(\cdot)$ provided we can obtain a suitable uniform bound on  $w_{Y^k_t}(x), \ k=1,2,..,$ for large $x$. 
To carry this out we  use the representation (\ref{H4}) for $w_\ve$.  Thus we have that
\begin{multline} \label{AV5}
w_{Y^k_t}(x) \ \le \  \int_x^\infty dx'\int_0^\infty G_\ve(x',y,0,t) \ c_{Y^k_0}(y) \ dy \ = \\
 \int_x^\infty dx'  \ G_\ve(x',0,0,t )w_{Y^k_0}(0)+m_{1,1/\La_\ve}(t)\int_0^\infty G_\ve(x,y,0,t) \ w_{Y^k_0}(y) \ dy  \ ,
\end{multline}
where $G_\ve$ is given by (\ref{K2}) with $A(\cdot)=1/\La_\ve(\cdot)$.  Evidently (\ref{AK5}) and (\ref{AV5}) yield a uniform upper bound on $w_{Y^k_t}(x), \ k=1,2,..,$ for large $x$, which decays rapidly as $x\ra\infty$ to $0$.  

We have therefore proven the result for random variables $X_0$ which satisfy $\|X_0\|_\infty<\infty$.  In the case when $\|X_0\|_\infty=\infty$ we proceed similarly,  approximating $X_0$ with variables $Y^k_0$ satisfying (\ref{O5}) by averaging the beta function of $X_0$ over intervals of length $1/k$ as in (\ref{AB5}) for  $0\le x\le k, \ k=1,2,..$. 
\end{proof}
\begin{lem}
Assume that the function $h_{X_0}(\cdot)$ for the initial data random variable $X_0$ of (\ref{G1}) satisfies (\ref{O5}) and $T>0$. Then there exists $x_T>0$ such that $\beta_{X_t}(x)<1$ for all $x\ge x_T, \ 0\le t\le T$.  
\end{lem}
\begin{proof}
We first obtain a lower bound on  the ratio of the half line Dirichlet Green's function defined by (\ref{L2}) to the full line Green's function (\ref{K2}).  Letting $A_\infty$ be given by (\ref{A2}),  we show that  for any $\ga>0$, there are positive constants $C_1,C_2$ depending only on $\ga, \ A_\infty T$ such that 
\be \label{AW5}
\frac{G_{\ve,D}(x,y,0,T)}{G_{\ve}(x,y,0,T)} \ \ge \ 1-\exp\left[-\frac{x^2}{C_1\ve T}\right] \quad {\rm for \ } y\ge \ga x, \ x\ge C_2[T+\sqrt{\ve T}].
\ee
We see from (\ref{AA2}), (\ref{AM2})  that in order to establish (\ref{AW5}) it is sufficient to show that there are constants $C_1,C_2$ depending only on $A_\infty T$ such that
\be \label{AX5}
P\left(\sup_{0\le s\le T}\left|\int_0^T k(s,s') \ dB(s')\right|> z\right) \ \le \ \exp\left[-\frac{z^2}{C_1 T}\right] \quad {\rm for \ } z\ge C_2\sqrt{T}.
\ee
The inequality (\ref{AX5}) follows from Doob's Martingale inequality as in the proof of Proposition 5.1.

We have now that
\begin{multline} \label{AY5}
c_{X_t}(x) \ = \ \int_0^\infty G_{\ve,D}(x,y,0,t)c_{X_0}(y) \ dy \ \le \\
\int_0^{y_0} G_{\ve}(x,y,0,t)c_{X_0}(y) \ dy+\int_{y_0}^\infty G_{\ve}(x,y,0,t)c_{X_0}(y) \ dy \ ,
\end{multline} 
where $A(\cdot)=1/\La_\ve(\cdot)$ in (\ref{K2}) and $y_0$ is given in (\ref{O5}).  We can bound the first integral on the RHS of (\ref{AY5}) using integration by parts to obtain the inequality
\be \label{AZ5}
\int_0^{y_0} G_{\ve}(x,y,0,t)c_{X_0}(y) \ dy \ \le \ G_{\ve}(x,0,0,t)w_{X_0}(0) -m_{1,1/\La_\ve}(t)\frac{\pa}{\pa x} \int_0^{y_0} G_{\ve}(x,y,0,t)w_{X_0}(y) \ dy \ .
\ee
Since $c_{X_0}(y)$ is Gaussian for $y\ge y_0$ as given in (\ref{O5}),   the second integral on the RHS of (\ref{AY5}) is bounded by  a Gaussian,
\begin{multline} \label{BA5}
\int_{y_0}^\infty G_{\ve}(x,y,0,t)c_{X_0}(y) \ dy \ \le \ K\int_{-\infty}^\infty G_{\ve}(x,y,0,t)\exp[-a(y-y_0)-\{a(y-y_0)\}^2/2L] \ dy \\
 = \ K\exp[L/2]\left( \frac{L}{a^2\ve \sig_{1/\La_\ve}^2(t)+Lm_{1,1/\La_\ve}(t)^2}\right)^{1/2} \ \times \\
 \exp\left[-\frac{\left(x+m_{2,1/\La_\ve}(t)-m_{1,1/\La_\ve}(t)y_0+m_{1,1/\La_\ve}(t)L/a\right)^2}{2\left\{\ve \sig_{1/\La_\ve}^2(t)+Lm_{1,1/\La_\ve}(t)^2/a^2\right\} }\right] \ .
\end{multline}
We conclude from (\ref{K2}), (\ref{AY5})-(\ref{BA5}) that there exist positive constants  $x_T,M_T$  such that  for $0<t\le T, \ x\ge x_T$,
\be \label{BB5}
c_{X_t}(x) \ \le \ \left[1+\exp[-x^2/M_T\right]K_t\exp\left[-\frac{\left(x+ \bar{x}_t\right)^2}{2\left\{\ve \sig_{1/\La_\ve}^2(t)+Lm_{1,1/\La_\ve}(t)^2/a^2\right\} }\right] \ , 
\ee
where $K_t,\bar{x}_t$ are given by the formulas
\begin{eqnarray} \label{BC5}
K_t \ &=& \ K\exp[L/2]\left( \frac{L}{a^2\ve \sig_{1/\La_\ve}^2(t)+Lm_{1,1/\La_\ve}(t)^2}\right)^{1/2} \ , \\
\bar{x}_t \ &=& \ m_{2,1/\La_\ve}(t)-m_{1,1/\La_\ve}(t)y_0+m_{1,1/\La_\ve}(t)L/a \ . \nonumber
\end{eqnarray}

We can use (\ref{AW5}) to obtain a lower bound on $c_{X_t}(\cdot)$.  Thus we have that
\be \label{BD5}
c_{X_t}(x) \ \ge \ \left\{1-\exp\left[-\frac{x^2}{C_1\ve T}\right]\right\} \int_{\ga x}^\infty G_\ve(x,y,0,t)c_{X_0}(y) \ dy \ . 
\ee
Assuming $x$ sufficiently large so that $\ga x\ge y_0$ then from (\ref{BA5}) we see that
\begin{multline} \label{BE5}
 \int_{\ga x}^\infty G_\ve(x,y,0,t)c_{X_0}(y) \ dy \ = \ K_t\exp\left[-\frac{\left(x+ \bar{x}_t\right)^2}{2\left\{\ve \sig_{1/\La_\ve}^2(t)+Lm_{1,1/\La_\ve}(t)^2/a^2\right\} }\right] \\
 - \ K\int_{-\infty}^{\ga x} G_{\ve}(x,y,0,t)\exp[-a(y-y_0)-\{a(y-y_0)\}^2/2L] \ dy \ .
\end{multline}
If $\ga>0$ is sufficiently small then the the second term on the RHS of (\ref{BE5}) is much smaller  than the first term. Hence we conclude that 
there exist positive constants  $x_T,M_T$  such that  for $0<t\le T, \ x\ge x_T$,
\be \label{BF5}
c_{X_t}(x) \ \ge \ \left[1-\exp[-x^2/M_T\right]K_t\exp\left[-\frac{\left(x+ \bar{x}_t\right)^2}{2\left\{\ve \sig_{1/\La_\ve}^2(t)+Lm_{1,1/\La_\ve}(t)^2/a^2\right\} }\right] \ . 
 \ee
 
 We can use (\ref{BB5}), (\ref{BF5}) to obtain an upper bound for $\beta_{X_t}(x)$ when $0<t\le T, \ x\ge x_T$. In fact from the formula (\ref{AI7})
 we immediately conclude that
 \be \label{BG5}
 \beta_{X_t}(x) \ \le \ \left[1+\exp[-x^2/M_T\right]^2\left[1-\exp[-x^2/M_T\right]^{-2} \beta_Y(x) \ , \quad {\rm for \ } x\ge x_T, 
 \ee
 where $\beta_Y(\cdot)$ is the beta function for a Gaussian variable $Y$ with mean $m$ and variance $\sig^2$ given by
 \be \label{BH5}
 m \ = \ -\bar{x}_t \ , \quad \sig^2 \ = \ \ve \sig_{1/\La_\ve}^2(t)+Lm_{1,1/\La_\ve}(t)^2/a^2 \ .
 \ee
 Now from  Lemma 2.1 we see that if $Y$ is Gaussian with mean $m$ and variance $\sig^2$, then there exists a universal constant $C$ such that
 \be \label{BI5}
 \beta_Y(x) \ \le \ 1-\frac{\sig^2}{2(x-m)^2} \quad {\rm for \ } x\ge C\sig+m \ .
 \ee
 The result follows from (\ref{BG5}), (\ref{BI5})  upon choosing $x_T$ sufficiently large.
\end{proof}
\begin{rem}
The inequality (\ref{AW5}) was easy to obtain from the explicit representation for the stochastic process $Y_\ve(s), \ 0\le s\le T,$ of (\ref{H2}) conditioned on $Y_\ve(0)=y, \ Y_\ve(T)=x$, when the drift $b(\cdot,\cdot)$ is linear.  It is much more difficult to obtain estimates on probabilities for the conditioned process in the case of non-linear $b(\cdot,\cdot)$. Theorem 1.2 of \cite{cg}  proves a result for the cdf of $Y_\ve(t), \ 0<t<T,$ in the case of $b(\cdot,\cdot)$ satisfying the uniform Lipschitz condition (\ref{A2}),  which is analagous to (\ref{AW5}). 
\end{rem}
\begin{proof}[Proof of log-concavity of the function $h_\ve$] We first assume that the initial condition random variable $X_0$ for (\ref{G1}) satisfies (\ref{O5}).  Then by standard regularity theorems \cite{fried}  for solutions to (\ref{G1}), the function $u_\ve$ of (\ref{M5}) is continuous 
on the closed set $\{(x,t) \ : \ x\ge 0, t\ge 0\}$.  Furthermore Lemma 7.5 implies that for any $T>0$ there exists $x_T>0$ such that $u_\ve(x,t)>0$ for $0\le t\le T, \ x\ge x_T$.  In addition, (\ref{N5}) implies that $u_\ve(0,t)>0,  \ 0\le t\le T,$ and (\ref{O5}) that $u_\ve(x,0)>0, \ x\ge 0$. Since $u_\ve$ is a classical solution to (\ref{M5}), the maximum principle \cite{pw} implies that $u_\ve(x,t)>0$ for $0\le t\le T, \ 0\le x\le x_T$.
We have therefore proven that the function $h_\ve(\cdot,t)$ is log-concave for $0\le t\le T$ when the initial data random variable $X_0$ satisfies (\ref{O5}). The log-concavity of $h_\ve(\cdot,t), \ t>0,$ for general log-concave initial data random variable $X_0$ now follows from Lemma 7.4.
\end{proof}
\begin{rem}
The main difficulty in implementing Korevaar's argument is to show that the solution of the PDE is log-concave on the boundary of the region. We accomplished this here by taking advantage of the fact that the  full line Green's function is  Gaussian when the drift $b(\cdot,\cdot)$ for (\ref{B2}) is linear.  In the case of non-linear $b(\cdot,\cdot)$ it is not possible to argue in this way. An alternative approach is to use Korevaar's observation \cite{kor} that a Dirichlet boundary condition implies log-concavity close to the boundary on account of the Hopf maximum principle \cite{pw}.  Some log-concavity theorems   for non-linear $b(\cdot,\cdot)$  are proved in the appendix of \cite{cg} using this method. 
\end{rem}

\end{document}